\documentclass[
	fontsize=11pt,
	twoside=false,
	abstract=true, 
	a4paper,
]{scrartcl}
 
\usepackage[bookmarkdepth=subsection]{templ}
\usepackage{bm}
\usepackage{mathrsfs}

\usepackage{tikz}
\usetikzlibrary{arrows}
\usetikzlibrary{positioning}
\usetikzlibrary{shapes.geometric}
\usepackage{pgf}
\usetikzlibrary{patterns}
\usetikzlibrary{patterns.meta}
\usepgflibrary{plothandlers} % LaTeX and plain TeX and pure pgf
\usepgflibrary[plothandlers] % ConTeXt and pure pgf
\usetikzlibrary{plothandlers} % LaTeX and plain TeX when using TikZ
\usetikzlibrary[plothandlers] % ConTeXt when using TikZ
\usetikzlibrary{topaths} % LaTeX and plain TeX
\usetikzlibrary[topaths] % ConTeXt

\makeatletter
\newcommand{\leqnomode}{\tagsleft@true\let\veqno\@@leqno}
\newcommand{\reqnomode}{\tagsleft@false\let\veqno\@@eqno}
\makeatother
\newcommand{\eps}{\varepsilon}

\newcommand{\RR}{\mathbb{R}}

\newcommand{\Ss}{\mathscr{S}}

\newcommand{\dd}{\mathrm{d}}

\newcommand{\norm}[1]{\left\lVert#1\right\rVert}
\newcommand{\abs}[1]{\left\lvert#1\right\rvert}
\newcommand{\del}{\partial}
\newcommand{\supp}{\mathrm{supp}}

\newcommand{\td}[1]{\tilde{#1}}
\newcommand{\an}{\quad\text{and}\quad}

\newcommand{\id}{\mathrm{id}}

\newcommand{\itdo}{1_{\tilde{\Omega}}}
\newcommand{\bdry}{(h\td{\del}_\nu+(\td{\del}_\nu\phi))}

\usepackage{xparse}

\NewDocumentCommand{\op}{O{h}m}{\mathrm{Op}_{#1}\left(#2\right)}
\NewDocumentCommand{\tbdry}{O{}}{\tau_{#1}\bdry}

\begin{document}

\title{\LARGE An Inverse Problem with Partial Neumann Data and $L^{n/2}$ Potentials}

\author{Leonard Busch \and Leo Tzou}

\date{}

\maketitle

\vspace{-2.5\baselineskip}
\begin{abstract}
	\noindent
	We consider a partial data inverse problem with unbounded potentials. Rather than rely on functional analytic arguments or Carleman estimates, we construct an explicit Green's function with which we construct complex geometric optics (CGO) solutions and show unique determinability of potentials in $L^{n/2}$ for the Schr\"odinger equation with partial Neumann data.

	\medskip

	{\noindent\textbf{Keywords:} Schr\"odinger equation, Calder\'on problem, inverse problems, pseudo-differential operators, partial data, Green's function}
\end{abstract}

\section{Introduction}\label{sec:intro}

Let $\Omega \subset \RR^n$ be a fixed, bounded, domain with smooth boundary in $\RR^n$ with $n \geq 3$, let $q \in L^{n/2}(\Omega)$ be a potential.  

For $\Gamma\subset \del\Omega$ and $f\in H^{-1/2}(\Gamma)$, consider the forward problem concerned with finding a solution $u \in H^1(\Omega)$ to 
\begin{equation}\label{eq:neumannproblem}
	(\Delta+q)u = 0 \text{ in } \Omega\,, \an \del_\nu u\vert_\Gamma = f\,,\quad \del_\nu u\vert_{\del\Omega\setminus\Gamma} = 0\,.
\end{equation}
In the inverse problem associated to \cref{eq:neumannproblem}, we seek to address whether the potential $q$ is uniquely determined by the set of boundary values $u\vert_{\Gamma}$ for solutions $u$ of \cref{eq:neumannproblem} with 
$f$ ranging in $H^{-1/2}(\Gamma)$. 

When the Neumann-to-Dirichlet map $N_q\colon H^{-1/2}(\del\Omega) \to H^{1/2}(\del\Omega), f\mapsto u\vert_{\del\Omega}$ is well-defined for \cref{eq:neumannproblem}, \cite{chungpartial} showed unique determination of potentials $q\in L^\infty$ in dimensions $n\geq 3$ under some restrictions on $\Gamma$. In dimension $2$, \cite{imanuvilov2012inverse} establish the same uniqueness for potentials $q \in W^{1,p}, p>2$.  

The problem is motivated by what is now called electrical impedance tomography or the Calder\'on problem from the eponymous \cite{MR590275}. Here, one attempts to recover the electrical conductivity of a medium by taking current and voltage measurements only at the medium's boundary. After a change of variables, the associated conductivity equation becomes the Schr\"odinger equation as it is in \cref{eq:neumannproblem}. In this setting, the boundary data $u\vert_{\del\Gamma}$ and $\del_\nu u\vert_{\del\Gamma}$ are interpreted as voltage and current flux multiplied by surface conductivity, respectively. See also a survey of recent developments regarding electrical impedance tomography \cite{zbMATH05655833}.

Also motivated by electrical impedance tomography is the closely related Dirichlet boundary formulation of the inverse problem arising from \cref{eq:neumannproblem}, which has been studied much more extensively. In this formulation, relying on work from \cite{zbMATH03939884}, \cite{zbMATH04015323} were the first to show uniqueness of potentials in dimensions $\geq 3$ given full boundary data, and \cite{zbMATH04105476} gave a reconstruction algorithm. The first such results in dimension $2$ for the Schr\"odinger equation came in \cite{zbMATH05254953}.

In practical settings, measurements of the full data cannot be guaranteed, and one may ask whether the potential is uniquely determined when $\Gamma \neq \del\Omega$, possibly with some restrictions upon $\Gamma$. For the Dirichlet formulation in dimension $2$ in the case of partial data, uniqueness was shown in \cite{zbMATH05775680} and was generalized to Riemann surfaces in \cite{zbMATH05904312}. In dimensions $\geq 3$, \cite{recoveringCauchydata} proved uniqueness of potentials in $L^\infty$. In partial data results, the restrictions upon the partial boundary are generally non-mutually inclusive, for a survey regarding other partial data results with the Dirichlet boundary formulation see \cite{recentProg} and the references therein.

Of particular interest to us is \cite{CT20}, which proved the uniqueness of potentials $q\in L^{n/2}(\Omega)$ in the partial data case where $n \geq 3$. The reason the exponent $n/2$ may not come entirely out of nowhere is the following. The inverse problem for the Schr\"odinger equation may be considered related to the unique continuation property with Cauchy data, and according to \cite{zbMATH03953442}, the value $n/2$ is the optimal exponent of Lebesgue spaces for the strong unique continuation property to hold.

The goal of this paper is to reproduce \cite{CT20} with partial Neumann data, that is, show uniqueness of potentials in $L^{n/2}(\Omega)$ in the formulation from \cref{eq:neumannproblem}.

The motivation for studying the Neumann data case instead of the Dirichlet case is the following: in physical implementations of electrical impedance tomography, a set of electrodes is attached to a medium. A subset of these electrodes apply currents and the remaining ones measure voltages, see \cite{Adler_2021} for an overview. This corresponds to prescribed data $\del_\nu u\vert_{\del\Omega}$ and measured data $u\vert_{\del\Omega}$, which is the Neumann formulation of the inverse problem.

\subsection{Statement of Results}

Let $\Omega \subset \RR^n$ be a fixed bounded smooth domain with $n \geq 3$, let $\omega \in \mathbb{S}^n$, and denote by $(y',y_n)$ the coordinate system spanned by $\omega^\perp\oplus\RR\omega$. Define $p' = \frac{2n}{n+2}$ and $p = \frac{2n}{n-2}$.

Let $\partial\Omega^\pm \coloneqq \{y\in \partial\Omega \colon \pm \nu_n(y)\geq 0\}$, where $\nu$ is the outward pointing normal unit vector and $\nu_n$ its $n$-th component.  Let the front and back faces $F, B \subset \partial \Omega$ satisfy $\partial\Omega^+ \Subset F \subset \partial\Omega$ and $\partial\Omega^- \Subset B \subset \partial\Omega$. 
Assume that both $\del\Omega\setminus B$ and $\del\Omega\setminus F$ are the union of open, positively separated sets $\Gamma_j \subset \partial\Omega$ so that for each $j$, there is a neighbourhood of $\Gamma_j$ in $\partial \Omega$ that is the graph of a function $g_j\in C_c^\infty(\RR^{n-1})$ (see \cref{eq:gdescribesboundary} for a precise definition).

\begin{figure}
%			\begin{tikzpicture}
%		%\draw[gray] (10pt,11pt) node {\footnotesize$\rightarrow$} (20pt,30pt) node {\footnotesize$\rightarrow$} (11pt,50pt) node {\footnotesize$\uparrow$} (-47pt,10pt) node {\footnotesize$\leftarrow$} (-47pt,-10pt) node {\footnotesize$\leftarrow$} (30pt,-20pt) node {\footnotesize$\rightarrow$};
%		\pgfplothandlerclosedcurve
%		\pgfplotstreamstart
%		\pgfplotstreampoint{\pgfpoint{10pt}{10pt}}
%		\pgfplotstreampoint{\pgfpoint{20pt}{30pt}}
%		\pgfplotstreampoint{\pgfpoint{10pt}{50pt}}
%		\pgfplotstreampoint{\pgfpoint{-40pt}{25pt}}
%		\pgfplotstreampoint{\pgfpoint{-47pt}{10pt}}
%		\pgfsetplottension{0}
%		\pgfplotstreampoint{\pgfpoint{-47pt}{-10pt}}
%		\pgfplotstreampoint{\pgfpoint{-40pt}{-25pt}}
%		\pgfsetplottension{0.5}
%		\pgfplotstreampoint{\pgfpoint{0pt}{-30pt}}
%		\pgfplotstreampoint{\pgfpoint{30pt}{-20pt}}
%		\pgfplotstreamend
%		\pgfusepath{stroke}
%		\draw[->] (-60pt,0) -- (45pt,0);
%		\draw[->] (0,-50pt) -- (0,60pt);
%		\node[below] at (45pt,0) {\footnotesize$y'$};
%		\node[left] at (0,60pt) {\footnotesize$y_n$};
%	\end{tikzpicture}
\begin{tikzpicture}
	\draw[->] (-65pt,0) -- (45pt,0);
	\draw[->] (0,-50pt) -- (0,60pt);
	\node[below] at (45pt,0) {\footnotesize$y'$};
	\node[left] at (0,60pt) {\footnotesize$y_n$};
	\draw (10pt,10pt) to [out=90,in=-90, looseness=0.7] (20pt,30pt);
	\draw (20pt,30pt) to [out=90,in=0] (10pt,50pt);
	\draw (10pt,50pt) to [out=180,in=90] (-40pt,20pt);
	\draw (-40pt,20pt) to [out=-90,in=90] (-40pt,-20pt);
	\draw (-40pt,-20pt) to [out=-90,in=180] (0pt,-40pt);
	\draw (0pt,-40pt) to [out=0,in=-90] (30pt,-20pt);
	\draw (30pt,-20pt) to [out=90,in=-90] (10pt,10pt);
	%\draw (10pt,11pt) node {\footnotesize$\rightarrow$} (20pt,30pt) node {\footnotesize$\rightarrow$} (11pt,50pt) node {\footnotesize$\uparrow$} (-42pt,10pt) node {\footnotesize$\leftarrow$} (-42pt,-10pt) node {\footnotesize$\leftarrow$} (30pt,-20pt) node {\footnotesize$\rightarrow$} (-2pt,-40pt) node {\footnotesize$\downarrow$};
\end{tikzpicture}
\hspace{1cm}
\begin{tikzpicture}
	\draw[->] (-65pt,0) -- (45pt,0);
	\draw[->] (0,-50pt) -- (0,60pt);
	\node[below] at (45pt,0) {\footnotesize$y'$};
	\node[left] at (0,60pt) {\footnotesize$y_n$};
	\draw[dashed] (10pt,10pt) to [out=90,in=-90, looseness=0.7] (20pt,30pt);
	\draw[dotted] (20pt,30pt) to [out=90,in=0] (10pt,50pt);
	\draw[dotted] (10pt,50pt) to [out=180,in=90] (-40pt,20pt);
	\draw (-40pt,20pt) to [out=-90,in=90] (-40pt,-20pt);
	\draw[dashed] (-40pt,-20pt) to [out=-90,in=180] (0pt,-40pt);
	\draw[dashed] (0pt,-40pt) to [out=0,in=-90] (30pt,-20pt);
	\draw[dotted] (30pt,-20pt) to [out=90,in=-90] (10pt,10pt);
	\draw (10pt,11pt) node {\footnotesize$\rightarrow$} (20pt,30pt) node {\footnotesize$\rightarrow$} (11pt,50pt) node {\footnotesize$\uparrow$} (-40pt,20pt) node {\footnotesize$\leftarrow$} (-40pt,-20pt) node {\footnotesize$\leftarrow$} (30pt,-20pt) node {\footnotesize$\rightarrow$} (-2pt,-40pt) node {\footnotesize$\downarrow$};
\end{tikzpicture}
\hspace{1cm}
\begin{tikzpicture}
	\draw[->] (-65pt,0) -- (45pt,0);
	\draw[->] (0,-50pt) -- (0,60pt);
	\node[below] at (45pt,0) {\footnotesize$y'$};
	\node[left] at (0,60pt) {\footnotesize$y_n$};
	\draw (10pt,10pt) to [out=90,in=-90, looseness=0.7] (20pt,30pt);
	\draw (20pt,30pt) to [out=90,in=0] (10pt,50pt);
	\draw (10pt,50pt) to [out=180,in=90] (-40pt,20pt);
	\draw (-40pt,-20pt) to [out=-90,in=180] (0pt,-40pt);
	\draw (0pt,-40pt) to [out=0,in=-90] (30pt,-20pt);
	\draw (30pt,-20pt) to [out=90,in=-90] (10pt,10pt);
	\node[circle,draw=white, fill=white, inner sep=0pt,minimum size=9pt] at (10pt,11pt) {};
	\node[circle,draw=white, fill=white, inner sep=0pt,minimum size=9pt] at  (20pt,30pt) {};
	\node[circle,draw=white, fill=white, inner sep=0pt,minimum size=9pt] at (-40pt,20pt) {};
	\node[circle,draw=white, fill=white, inner sep=0pt,minimum size=9pt] at (-40pt,-20pt) {};
	\node[circle,draw=white, fill=white, inner sep=0pt,minimum size=9pt] at  (30pt,-20pt) {};
	\node at (-30pt,47pt) {\footnotesize$\Gamma_1$};
	\node at (15pt,-10pt) {\footnotesize$\Gamma_3$};
	\node at (25pt,15pt) {\footnotesize$\Gamma_2$};
	\node at (-25pt,-27pt) {\footnotesize$\Gamma_4$};
	%\draw (10pt,11pt) node {\footnotesize$\rightarrow$} (20pt,30pt) node {\footnotesize$\rightarrow$} (11pt,50pt) node {\footnotesize$\uparrow$} (-40pt,20pt) node {\footnotesize$\leftarrow$} (-40pt,-20pt) node {\footnotesize$\leftarrow$} (30pt,-20pt) node {\footnotesize$\rightarrow$} (-2pt,-40pt) node {\footnotesize$\downarrow$};
\end{tikzpicture}
%\begin{tikzpicture}
%	\pgfplothandlercurveto
%	\begin{scope}
%		\pgfplotstreamstart
%		\pgfplotstreampoint{\pgfpoint{-47pt}{-10pt}}
%		\pgfplotstreampoint{\pgfpoint{-40pt}{-25pt}}
%		\pgfplotstreampoint{\pgfpoint{0pt}{-30pt}}
%		\pgfplotstreampoint{\pgfpoint{30pt}{-20pt}}
%		\pgfplotstreampoint{\pgfpoint{10pt}{10pt}}
%		\pgfplotstreampoint{\pgfpoint{20pt}{30pt}}
%		\pgfplotstreampoint{\pgfpoint{10pt}{50pt}}
%		\pgfplotstreampoint{\pgfpoint{-40pt}{25pt}}
%		\pgfplotstreampoint{\pgfpoint{-47pt}{10pt}}
%		\pgfplotstreamend
%		\pgfusepath{stroke}
%	\end{scope}
%	\node[circle,draw=white, fill=white, inner sep=0pt,minimum size=5pt] at (10pt,11pt) {};
%	\node[circle,draw=white, fill=white, inner sep=0pt,minimum size=5pt] at  (20pt,30pt) {};
%	\node[circle,draw=white, fill=white, inner sep=0pt,minimum size=5pt] at (-47pt,10pt) {};
%	\node[circle,draw=white, fill=white, inner sep=0pt,minimum size=5pt] at (-47pt,-10pt) {};
%	\node[circle,draw=white, fill=white, inner sep=0pt,minimum size=5pt] at  (30pt,-20pt) {};
%	\draw[->] (-65pt,0) -- (45pt,0);
%	\draw[->] (0,-50pt) -- (0,60pt);
%	\node[below] at (45pt,0) {\footnotesize$y'$};
%	\node[left] at (0,60pt) {\footnotesize$y_n$};
%	\node at (-30pt,38pt) {\footnotesize$\Gamma_1$};
%	\node at (15pt,-10pt) {\footnotesize$\Gamma_3$};
%	\node at (25pt,15pt) {\footnotesize$\Gamma_2$};
%	\node at (-30pt,-19pt) {\footnotesize$\Gamma_4$};
%\end{tikzpicture}
\caption{An example of a valid set $\Omega$ meant as the interior of the drawn line. The arrows in the middle picture are select points on which the outward pointing normal vector has been sketched. Furthermore, the part of the boundary marked by dotted lines is where $\nu(y) \cdot y_n >0$, those marked by dashed lines is where $\nu(y) \cdot y_n <0$, and the part marked by a full line is where $\nu(y) \cdot y_n =0$. Finally, the right-most sketch depicts $\del\Omega\setminus F = \Gamma_2\cup\Gamma_4$ and $\del\Omega\setminus B = \Gamma_1\cup\Gamma_3$ where the $\Gamma_j$ are positively separated from each other. Note that the perfectly vertical part of the boundary on the left side is contained in both $F$ and $B$.}
\end{figure}
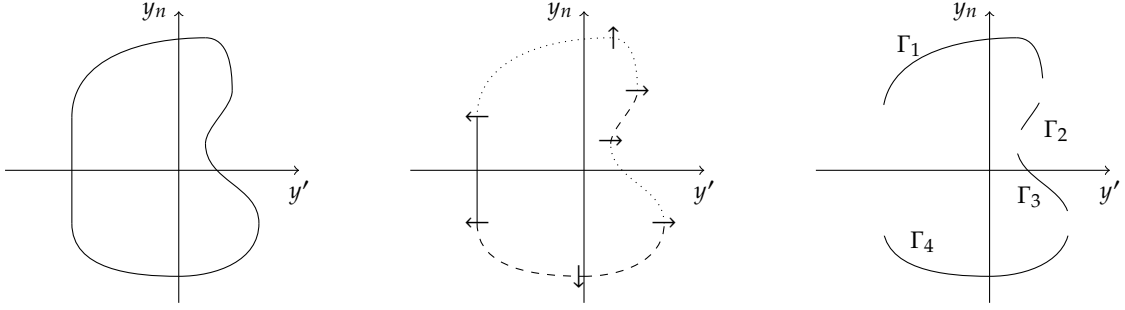

Following \cite{recentProg}, for sets $\Gamma_D, \Gamma_N \subset \del\Omega$ and $q\in L^{n/2}(\Omega)$, define the partial Cauchy data set to be 
\begin{equation}\label{eq:defofcauchydata}
	C_{q}^{\Gamma_D,\Gamma_N} \coloneqq \left\{(u\vert_{\Gamma_D}\,,\,\del_\nu u\vert_{\partial\Omega}) \colon (\Delta+q)u=0 \text{ in } \Omega\,,\, u\in H_{\Delta}^\sharp(\Omega)\,,\, \supp(\del_\nu u\vert_{\partial\Omega}) \subset \Gamma_N\right\}\,,
\end{equation}
where
\[
	H_{\Delta}^\sharp(\Omega) = \left\{u \in \mathcal{D}'\colon u\in H^1(\Omega); \Delta u\in L^{p'}(\Omega)\right\}
\]
is a Banach space with norm $\norm{u}_{H^\sharp_\Delta} = \norm{u}_{H^1}+\norm{\Delta u}_{L^{p'}}$.

\begin{theorem}\label{thm:uniqueness}
	If $q_\pm \in L^{n/2}(\Omega)$, and
	\[
		C_{q_+}^{B,F} = C_{q_-}^{B,F}\,,
	\]
	then $q_+ = q_-$.
\end{theorem}

After a change of variables coming from \cite{zbMATH04015323}, 
	we will also be able to make a deduction about the Calder\'on problem for partial Neumann data, see also \cite[\S~1]{chungpartial}.
	\begin{corollary}\label{cor:calderon}
	Let $\gamma_\pm \in W^{2,n/2}(\Omega)$ be strictly positive $\gamma_\pm \geq \eps > 0$. Moreover, assume $\del_\nu \gamma_+\vert_{\del\Omega} = \del_\nu \gamma_-\vert_{\del\Omega} = 0$ and $\gamma_+\vert_{\del\Omega} = \gamma_-\vert_{\del\Omega}$, and define
	\[
		C_{\gamma_\pm}^{B,F} \coloneqq \left\{(u\vert_{B}\,,\,\gamma_\pm\del_\nu u\vert_{\partial\Omega}) \colon \mathrm{div}(\gamma_\pm\nabla u)=0 \text{ in } \Omega\,,\, u\in H^1(\Omega)\,,\, \supp(\del_\nu u\vert_{\partial\Omega}) \subset F\right\}\,.
	\]
	If
	\[
		C_{\gamma_+}^{B,F} = C_{\gamma_-}^{B,F}\,,
	\]
	then $\gamma_+ = \gamma_-$.
\end{corollary}

The main ingredient necessary to prove \cref{thm:uniqueness} comes from finding so-called complex geometric optics (CGO) solutions. Let $\Xi = \del\Omega\setminus B$. 

\begin{proposition}\label{cgosolsprop} 
	For any $\omega' \in \mathbb{S}^n$ with $\omega \perp \omega'$, and any $q\in L^{n/2}(\Omega)$ there exists $a_h \in C^\infty(\RR^n)$ so that for any $\varepsilon >0$, there are $r \in H^1(\Omega)$ and linear $\psi_h$ so that
	\[
	u(y) = e^{\frac{\omega+i\omega'+ih\psi_h}{h}\cdot y}(1+a_h(y)+r(y)) \in H^1(\Omega)
	\]
	is a solution to
	\[
	(\Delta + q)u = 0\,,\an \del_\nu u\vert_\Xi = 0\,,
	\]
	with $\norm{a_h}_{L^\infty}\leq C$, $\lim_{h\to 0}a_h(y) = 0$ pointwise for all $y\in \Omega$, $\lim_{h\to 0}\norm{r}_{L^2} \leq \varepsilon$ and $\norm{r}_{L^p} \leq C$ for $h$ small enough.
\end{proposition}

Obtaining these CGO solutions begins by conjugating the Laplacian $h^2\Delta_\phi \coloneqq e^{-\phi/h}h^2\Delta e^{\phi/h}$, $\phi = y_n$, and constructing a Green's function for this conjugated Laplacian in a distributional sense. 

\begin{theorem}\label{thm:G}
	There is a linear operator $G_{\Xi}$ so that for small $h>0$, all $v\in L^{p'}(\Omega)$ and all $f\in L^2(\Xi)$,
	\begin{align}\label{eq:inverseG}
		\langle u, h^2\Delta_\phi G_{\Xi} (v,f)\rangle &= \langle u,v\rangle\,,\quad \forall u\in C_c^\infty(\Omega)\,,
	\shortintertext{and}
		(h\del_\nu + \del_\nu\phi)G_{\Xi}(v,f)\vert_\Xi &= -f\,.\label{eq:boundaryG}
	\end{align}
	Furthermore,
	\begin{equation}\label{eq:Gshorth}
	G_\Xi\colon L^{p'}(\Omega)\times L^2(\RR^{n-1}) \to_{h^{-2}} L^p(\Omega)\cap H^1(\Omega)\,,\an G_\Xi\colon L^{2}(\Omega)\times L^2(\RR^{n-1}) \to_{h^{-1}} H^1(\Omega)\,,
	\end{equation}
	with bounds
	\begin{alignat}{5}
		&\norm{G_{\Xi}(v,f)}_{L^p}&&\lesssim h^{-2}\norm{v}_{L^{p'}} &+& h^{-1}&&\norm{f}_{L^2}\label{eq:hG1} \\
		&\norm{G_{\Xi}(v,f)}_{H^1}&&\lesssim h^{-1}\norm{v}_{L^{2}} &+& &&\norm{f}_{L^2}\,,\label{eq:hG2}
	\end{alignat}
	where the constant implied in $\lesssim$ is independent of $h$.
\end{theorem}

As we rely heavily on tools developed in \cite{CT20}, in \cref{sec:setup} we recap the most relevant results from \cite{CT20} and introduce a change of variables. In \cref{sec:parametrices} we construct a parametrix that will be instrumental in showing \cref{thm:G}, and in \cref{sec:nonlinearR} we prove \cref{thm:G}. \cref{sec:CGO} is dedicated to finding the CGO solution in \cref{cgosolsprop} and \cref{sec:uniq} establishes the uniqueness of potentials and conductivities, namely \cref{thm:uniqueness} and \cref{cor:calderon}.

\section{Setup}\label{sec:setup}

The strategy in our proofs will be to flatten out the boundary of $\Omega$ along each of its boundary defining functions, proving results there and transforming back to the original coordinates of $\Omega$. 

\subsection{Change of Coordinates}\label{changeofVars}

Fix some open $\Gamma \subset \partial\Omega$ and $g \in C_c^\infty(\RR^{n-1})$ so that $g$ defines the boundary of $\Omega$ in a neighborhood of $\Gamma$, by which we mean that
\begin{equation}\label{eq:gdescribesboundary}
	\Omega \cap U = \{y\in U \colon y_n - g(y') > 0\}\,,\an  \partial\Omega\cap U = \{y\in U \colon y_n -g(y')= 0\}\,,
\end{equation}
for some open, bounded neighborhood $U\subset \RR^n$ of $\Gamma$ with $\Gamma\Subset \partial\Omega\cap U$.

Note that, given the boundary defining function $y_n-g$, the outward unit normal vector $\nu$ is defined as 
\begin{equation}\label{normalDer}
\nu(y) = -\frac{\nabla (y_n-g(y'))}{\abs{\nabla (y_n-g(y'))}} = (1+\abs{K(y')}^2)^{-1/2} \begin{pmatrix}
	\nabla'g(y') \\ -1
\end{pmatrix}\,,
\end{equation}
where we put 
\begin{equation}\label{eq:defofK}
	K(y') \coloneqq \nabla' g(y')\,.
\end{equation}

In order to flatten out the boundary we define 
\begin{equation}\label{eq:defofgamma}
	\gamma \colon (y',y_n) \mapsto (x',x_n) = (y', y_n - g(y'))\,,
\end{equation}
so that $\td{\Omega} \coloneqq \gamma(\Omega) \subset \RR^n_+$, and $\td{\Gamma} \coloneqq \gamma(\Gamma) \subset \{x_n = 0\}$, where $\RR^n_+ = \{(x',x_n) \in \RR^n\colon x_n > 0\}$.

From here on out, expressions marked with a tilde generally denote objects in these transformed variables, and as a convention we will use $x$ as a variable when working on objects defined in the upper half space $\RR^n_+$.

In particular, we will now calculate the normal derivative in transformed coordinates. Let $\gamma^\ast, (\gamma^{-1})^\ast$ be pullbacks by $\gamma, \gamma^{-1}$ respectively. If $u(y) = u\in C^\infty(\bar\Omega)$, where $\bar\Omega$ is the closure of $\Omega$, then the transformed function is $(\gamma^{-1})^\ast \circ u(x) \eqqcolon \td{u}(x) = u(x',x_n+g(x')) \in C^\infty(\bar{\RR}^n_+)$, so that $u(y)=\td{u}(y',y_n-g(y'))$.

Using \cref{normalDer}, we see that 
\begin{equation*}
	\del_{\nu}u(y) 
	= (1+\abs{K}^2)^{-1/2}(K\cdot (\nabla_{x'}\td{u})(\gamma (y)) - (1+\abs{K}^2)(\del_{x_n}\td{u})(\gamma (y))) = (\gamma^\ast \circ  \td{\del}_\nu \td{u})(y)\,, 
\end{equation*}
where
\begin{equation}\label{defofboundary}
	\td{\del}_\nu \coloneqq (1+\abs{K(x')}^2)^{-1/2}(K(x')\cdot \nabla_{x'} - (1+\abs{K(x')}^2)\del_{x_n})\,,
\end{equation}
so that if $\phi(y)=y_n$, then $\td{\del}_\nu \phi = \del_\nu \phi = -(1+\abs{K}^2)^{-1/2}$, and
\begin{equation}\label{eq:boundarycalcwithgamma}
	(h\del_{\nu}+(\del_\nu\phi)) u = \gamma^\ast \circ \bdry \circ (\gamma^{-1})^\ast u  = \gamma^\ast \circ \left[ (1+\abs{K}^2)^{-1/2}\left(hK\cdot\nabla_{x'}-h(1+\abs{K}^2)\del_{x_n}-1\right) \td{u}\right] \,.
\end{equation}

Now if $A$ is any operator, then 
\[
	A u = \gamma^\ast \circ (\gamma^{-1})^\ast \circ A \circ \gamma^\ast \circ (\gamma^{-1})^\ast u = \gamma^\ast \circ \td{A} \circ (\gamma^{-1})^\ast u = \gamma^\ast \circ \td{A}\td{u}\,,
\]
where $\td{A} = (\gamma^{-1})^\ast \circ A \circ \gamma^\ast$, and $\gamma^\ast \circ \td{A}\td{u}(y) = \td{A}\td{u}(x)\vert_{x = \gamma(y)}$.

Thus our strategy will be to construct an operator $\td{A}$ that has the properties we want on functions defined on $\RR^n_+$ and transform this operator back to the original coordinates.

\subsection{Recap of \cite{CT20}}\label{subsec:recap}

As this paper uses the same strategies as \cite{CT20}, we will make extensive use of some of their results. For the reader's convenience, we will state some of the most relevant ones here. Statements made in this section will not be proven aside from \cref{lem:Jsandwich}.

Recall first the definition of mixed semiclassical Sobolev spaces,
\begin{equation}\label{eq:defofmixedsemiclasssobo}
W^{k,l,r}(\RR^{n-1},\RR^n) \coloneqq W^{k,l,r} \coloneqq\{u\in \Ss'(\RR^n) \colon \langle hD'\rangle^k\langle hD\rangle^l u \in L^r(\RR^n)\}\,,
\end{equation}
where $k,l \in \RR$, $1<r<\infty$, $D' = (D_1,\dots,D_{n-1})$ and $h$ denotes the semi-classical parameter.

The pseudo-differential operators in this paper are all meant semiclassically as well. For an introduction we refer to \cite{zworski}. When we write $S^{m}_0, S^m_1$ we mean the usual semiclassical symbol classes of order $m$. For mapping properties of semiclassical pseudo-differential operators on weighted $H^k$ spaces, we refer to \cite[Prop.~2.2]{salo}.

Following \cite{CT20}, for $r,s\in \RR, j\in\{0,1\}$ introduce the mixed symbol classes
\[
	S^r_1S^s_j \coloneqq \{ b(x',\xi') c(x,\xi) \colon  b \in S^r_1(\RR^{n-1}), c \in S^s_j(\RR^n) \} \subset S_0^{r+s}\,,
\]
which obeys the calculus 
\begin{lemma}[{\cite[Prop.~2.2,Prop.~2.3]{CT20}}]\label{lem:calc}
	There is $k(n) \in \mathbb{N}$ depending only on the dimension $n$ so that the following holds.

	For $a\in S^{k_1}_1S^{s_1}_1 \cup S^{k_1}_1S^{-k(n)+s_1}_0$ and $b\in S^{k_2}_1S^{s_2}_1 \cup S^{k_2}_1S^{-k(n)+s_2}_0$ we have
	\[
		b(x',hD)\circ a(x',hD) = ab(x',hD) + h\sum_{\abs{\alpha}=1}(\del_{\xi}^\alpha b \del^\alpha_{x'} a)(x',hD) + h^2m(x',hD)
	\]
	for some $m(x',hD) \colon W^{k,s,r} \to W^{k-k_1-k_2,s-s_1-s_2,r}$. 

	Furthermore, for $c(x',\xi') \in S^{k_1}_1(\RR^{n-1}), d(x',\xi) \in S^{s_1}_1(\RR^n) \cup S_0^{-k(n)+s_1}(\RR^n)$ we have
	\[
		cd(x',hD) \colon W^{k,s,r} \to W^{k-k_1,s-s_1,r}\,,
	\]
	with norm bounded in $h >0$ uniformly. 

	Analogous statements hold for the weighted $H^k$ Sobolev spaces $H^k_\delta$ by using $S^k_1S^s_0 \subset S^{k+s}_0(\RR^n)$ and the calculus for these symbols, see for example \cite[Prop.~2.2]{salo}.
\end{lemma}

Fix $\phi(y) = y_n$, and let
\[
	h^2\Delta_\phi \coloneqq e^{-\phi/h}h^2\Delta e^{\phi/h}
\]
be the conjugated Laplacian. 

Let $c < 1$ so that $\frac{\abs{K}^2}{1+\abs{K}^2} < c$ for all $x'$, where $K$ is from \cref{eq:defofK}. Define two cut-off functions $\td{\rho},\td{\rho}_0 \in C_c^\infty(\RR^{n-1})$ so that 
\begin{alignat}{3}
	\td{\rho}_0 &\equiv 1 \text{ on } \{\abs{\xi'}^2 \leq c\}\,, \quad && \text{ and } \quad \supp \td{\rho}_0 && \Subset \{\abs{\xi'}^2 < 1\}\,,\label{defofrho0} \\ 
	\td{\rho} &\equiv 1 \text{ on } \supp\td{\rho}_0\,, \quad && \text{ and } \quad \supp \td{\rho} && \Subset \{\abs{\xi'}^2 < 1\}\,.\label{defofrho}
\end{alignat}

Denoting by $\td{\Delta}_\phi \coloneqq (\gamma^{-1})^\ast\circ \Delta_\phi \circ \gamma^\ast$ the transformed conjugated Laplacian, we may write
\begin{equation}\label{splitupofLaplace}
	\frac{1}{1+\abs{K}^2}h^2 \tilde{\Delta}_\phi = QJ + \tilde{A}_0 - h\tilde{E}_1+h^2\tilde{E}_0\,.
\end{equation}
In this decomposition we have
\begin{align}
	J = \mathrm{Op}_h\left(i\xi_n - i\tilde{a}_+-ihm_0\right)\,, \qquad
	Q = -\mathrm{Op}_h\left(i\xi_n - i\tilde{a}_- + ihm_0\right)\label{eq:defofJQ}
\end{align}
and 
\begin{equation}\label{eq:defofA0E1E0}
\tilde{A}_0  = \op[h,x']{\tilde{a}_0}\,,\quad \tilde{E}_1 = \op[h,x']{m_0 \tilde{a}_-}\,,\quad \tilde{E}_0 = \op[h,x']{\td{e}_0}\,,
\end{equation}
where $e_0 \in S^0_1(\RR^{n-1})$, $1_h = 1+ \frac{h\Delta_{x'}g}{2}$,
\begin{alignat}{2}
	r_0 &= (1-\tilde{\rho}_0)\sqrt{(1_h+iK\cdot\xi')^2 -(1-\abs{\xi'}^2)(1+\abs{K}^2)} && \in S^1_1(\RR^{n-1})\,, \label{eq:defofr} \\
	\tilde{a}_0 &= \frac{(1_h+iK\cdot\xi')^2-(1-\abs{\xi'}^2)(1+\abs{K}^2)-r_0^2}{(1+\abs{K}^2)^2} &&\in S^{-\infty}_1(\RR^{n-1})\,,\label{eq:defofa0}
\end{alignat}
and
\begin{align}\label{herem0isdefined}
	\tilde{a}_\pm = i\left(\frac{1_h+iK\cdot\xi' \pm r_0}{1+\abs{K}^2}\right) \in S^1_1(\RR^{n-1})\,, \an m_0 = \tilde{a}_+^{-1} \sum_{\abs{\alpha}=1}(-i)\del_{\xi'}^\alpha\tilde{a}_-\del_{x'}\tilde{a}_+ \in S_1^0(\RR^{n-1})\,.
\end{align}
Finally, we introduce 
\begin{equation}\label{splitupJintoF}
F_+ \coloneqq \op[h,x']{-i\tilde{a}_+-ihm_0}\,,\quad\text{so that}\quad J = h\del_{x_n} + F_+\,.
\end{equation}

\begin{lemma}[{\cite[Prop.~3.3]{CT20}}]\label{Jinverse}
	For $h>0$ small enough, there is an inverse $J^{-1}$ of $J$ so that 
	\[
		J^{-1} \colon L^r(\RR^n) \to W^{1,r}(\RR^n)\,,\qquad 1<r<\infty\,,
	\]
	and for all $u \in L^r(\RR^n)$ with $\supp u\subset \RR^n_+$, $\supp J^{-1}u \subset \RR^n_+$ and $\tau J^{-1}u = 0$, where $\tau$ is the trace onto $x_n = 0$.
\end{lemma}

Furthermore,
\begin{lemma}\label{lem:Jsandwich}
	If $a,b \in S^0_1(\RR^{n-1})$ are symbols with $\supp a \cap \supp b = \emptyset$, then for $A = \op[h,x']{a}$ and $B = \op[h,x']{b}$ we have
	\begin{equation}\label{eq:jsandwich}
		A J^{-1}B  \colon L^r(\RR^n) \to_{h^2} W^{1,r}(\RR^n)\,,\an A J^{-1} B \colon L^2_\delta(\RR^n) \to_{h^2} H^1_\delta(\RR^n)
	\end{equation}
	for all $1<r<\infty, \delta\in\RR$.
\end{lemma}
\begin{proof}
	This argument follows the proof of \cite[Lemm.~3.4]{CT20}, to which we refer for more details. Throughout this proof we use the calculus given in \cref{lem:calc} without explicit reference. 
	
	According to \cite[Prop.~3.3]{CT20} we can write
	\begin{equation}\label{eq:jinvform}
		J^{-1} = j^{-1}(x',hD)(1+hm_1(x',hD) + h^2 m_2(x',hD))^{-1}\,,
	\end{equation}
	where	
	\begin{align*}
		m_1, m_2 &\colon L^r(\RR^n) \to L^r(\RR^n)\,,&\an& &m_1,m_2 &\colon L^2_\delta(\RR^n) \to L^2_\delta(\RR^n) \\ 
		 j^{-1}(x',hD) &\colon L^r(\RR^n)\to W^{1,r}(\RR^n)\,,&\an& & j^{-1}(x',hD)&\colon L^2_{\delta}(\RR^n) \to H^1_{\delta}(\RR^n) 
	\end{align*}
	the part about $j^{-1}$ coming from \cite[Eq.~(3.6)]{CT20}.
	Additionally, from \cite[Eq.~(3.5)]{CT20}, 
	\[
		j^{-1} \in \mathrm{span}(S^0_1 S^{-1}_1 + S^{-\infty} S^{-1-k(n)}_0 + S^1_1 S^{-2}_1)\,,
	\]
	and writing out part of the Neumann inversion in \cref{eq:jinvform}, we have
	\[
		J^{-1} = j^{-1}(x',hD)(1+hm_1(x',hD)) + h^2M\,,
	\]
	where
	\[
		M \colon L^r(\RR^n)\to W^{1,r}(\RR^n)\,,\an M \colon L^2_{\delta}(\RR^n) \to H^1_{\delta}(\RR^n)\,,
	\]
	so that $A h^2 M B$ satisfies the estimates $A h^2 M B \colon L^r(\RR^n) \to_h W^{1,r}(\RR^n)$ and $A h^2 M B \colon L^2_\delta(\RR^n) \to_{h^2} H^1_\delta(\RR^n),\delta\in\RR$. 

	 We now perform the explicit estimate for the $S^0_1 S^{-1}_1$ term of $j^{-1}$, where all other terms (and those from $j^{-1}(x',hD) m_1(x',hD)$) can be dealt with similarly. Let $c_1 \in S_1^0(\RR^{n-1}), c_2 \in S^{-1}_1(\RR^n)$. Then
	\begin{equation}\label{eq:writeout}
		A \op{c_1c_2} B = A\op{c_1}\op{c_2} B + h A \sum_{\abs{\alpha}=1} \op{\del_{\xi'}^\alpha c_1} \op{\del_{x'}^\alpha c_2} B + h^2 A m B\,,
	\end{equation}
	where $m \colon L^r(\RR^n) \to W^{1,r}(\RR^n)$ and $m\colon L^2_\delta(\RR^n)\to H^1_\delta(\RR^n)$. 

	According to the usual pseudo-differential disjoint support properties, the first two terms in \cref{eq:writeout} satisfy \cref{eq:jsandwich}, whereas the term $h^2 A m B$ does so due to the mapping properties of $m$.
\end{proof}

Here, and from now onward, we will use the notation: for an operator $B$ and spaces $X,Y$ we write $B\colon X\to_{h^k} Y$ if the operator norm of $B: X\to  Y$ is bounded by $\mathcal{O}(h^k)$. 

\begin{lemma}[{\cite[Prop.~4.3]{CT20}}]\label{prop43}
	There are operators $G_\phi$ and $\td{G}_\phi \coloneqq (\gamma^{-1})^\ast \circ G_\phi \circ \gamma^\ast  = (\td{G}_\phi-\td{G}_\phi^c)+\td{G}_\phi^c$ so that for all $0 < \delta < 1$,
	\[
		h^2\td{\Delta}_\phi \td{G}_\phi = \id\,,\an h^2\Delta_\phi G_\phi = \id\,,
	\]
	and 
	\[
	G_\phi, \td{G}_\phi \colon L^2_\delta(\RR^n) \to_{h^{-1}} H^1_{\delta-1}(\RR^n)\,,\quad  G_\phi, \td{G}_\phi \colon L^{p'}(\RR^n) \to_{h^{-2}} L^p(\RR^n)\,.
	\]
	Furthermore, 
	\[
	\td{G}_\phi^c \colon L^2_\delta(\RR^n) \to_{h^{-1}} H^k_{\delta-1}(\RR^n)\,,\quad \td{G}_\phi^c \colon L^{p'}(\RR^n) \to_{h^{-2}} W^{k,p}(\RR^n) \quad\forall k\in\mathbb{N}\,,
	\]
	and $\td{G}_\phi - \td{G}_\phi^c$ is a pseudo-differential operator with symbol in $S^{-2}_1(\RR^n)$.
\end{lemma}

Finally, we will require
\begin{lemma}[{\cite[Lem.~4.5]{CT20}}]
\label{lem:EGestimate}
For $\td{E}_1$ from \cref{eq:defofA0E1E0} and $\td{G}_\phi$ from \cref{prop43}, we have
\[
\td{E}_1\td{G}_\phi = (\td{E}_1\td{G}_\phi)^c + \op{S^1_1S^{-2}_1} + h\op{S^0_1(\RR^n)}\td{G}_\phi
\]
where for all $k\in\mathbb{N}_0$,
\[
	(\td{E}_1\td{G}_\phi)^c \colon L^2(\RR^n) \to_{h^0} H^k(\RR^n)\,,\an (\td{E}_1\td{G}_\phi)^c \colon L^{p'}(\RR^n) \to_{h^{-1}} H^k(\RR^n)\,.
\]
\end{lemma}

\section{Parametrices}\label{sec:parametrices}

The aim of this section is to construct a parametrix for the operator
	\[
		\begin{pmatrix}
			\itdo h^2\td{\Delta}_\phi \\
			\tbdry[\td{\Gamma}]
		\end{pmatrix}\,.
	\]
	As the parametrix will be based on those constructed in \cite{CT20}, we distinguish between small and large frequencies: in \cref{sec:small} we construct a parametrix for small frequencies, and in \cref{sec:large} for large frequencies. Finally, in \cref{sec:combine} we will combine these two into a global parametrix. 

\subsection{Small Frequency Parametrix}\label{sec:small}

To invert $h^2\td{\Delta}_\phi$, we would like to use the small frequency parametrix defined in \cite{CT20}, but it will not have the boundary conditions we desire. Nevertheless, our strategy will be to determine its boundary properties (see \cref{lemmaOpNoBoundary}) and define a new operator which can correct this to what we want (see \cref{mappingPropsoft}) in such a way that this new operator is, up to lower order in $h>0$, in the null-space of $h^2\td{\Delta}_\phi$. 

Recall from \cite[p.~117]{CT20} that the symbol $p$ of $h^2\td{\Delta}_\phi$ is given by
\begin{equation}\label{eq:defofsymbolp}
	p(x',\xi) \coloneqq (1+\abs{K}^2)\xi_n^2-2i\xi_n (1_h+iK\cdot\xi') -(1-\abs{\xi'})\,,
\end{equation}
and define
\begin{equation}\label{eq:defofPs}
	 P_s\coloneqq \op{\frac{\td{\rho}}{p}}\,,
\end{equation}
where $\td{\rho}$ is from \cref{defofrho}. The expression $P_s$ is denoted by the same symbol and called the small frequency parametrix in \cite{CT20}.

We introduce the notation $\tau$ for the trace onto $x_n = 0$, where if $\tau$ is written with a subset of $\RR^n$ in the subscript, then the trace is meant onto that set instead.

\begin{lemma}\label{lemmaOpNoBoundary}
	If $v\in L^r(\RR^n)$ with $1<r<\infty$ is supported on the closure of $\RR^n_+$, then 
	\begin{equation}\label{eq:Psprops}
		h^2\tilde\Delta_\phi P_sv = \tilde \rho(hD') + hR_s\,,\an \tbdry P_sv = 0\,.
	\end{equation}
	where $R_s : L^r(\RR^n) \to_{h^0} L^r(\RR^n)$ is bounded uniformly in $h>0$.
\end{lemma}
\begin{proof}
	Fix $v\in L^r(\RR^n)$, $1<r<\infty$ with $\supp v\subset \bar{\RR}^n_+$. 
	Note first that \cite[Prop.~5.6]{CT20} gives the first part of \cref{eq:Psprops}, and by \cite[Prop.~5.7]{CT20} we have $\tau P_sv = 0$.
	
	We will establish that $\tau h\del_{x_n}P_sv = 0$. To see this, recall that by the proof of \cite[Prop.~5.7]{CT20}, for $u\in C_c^\infty(\RR^n)$ supported in $\RR^n_+$,
	\begin{equation}\label{eq:normsmallopboundr}
		P_su(x',x_n) = 0 \quad\text{on}\quad x_n \leq 0\,,\qquad \text{and so}\quad \tau h\del_{x_n}P_su = 0\,,
	\end{equation}
	because $P_su \in C^\infty(\RR^n)$.

	Furthermore, for every $w\in L^r(\RR^n)$, the trace theorem together with \cite[Prop.~5.6]{CT20} give
	\begin{equation}\label{eq:normsmallopwaway}
		h\norm{\tau h\del_{x_n}P_s w}_{L^r(\RR^{n-1})} \leq \norm{h\del_{x_n}P_s w}_{W^{1,r}(\RR^{n})} \leq\norm{P_s w}_{W^{2,r}(\RR^{n})} \lesssim \norm{w}_{L^r(\RR^{n})}\,.
	\end{equation}
	Thus, in the topology of $L^r(\RR^{n})$, approximating $v$ by $u_m \in C_c^\infty(\RR^n)$ supported in $\RR^n_+$, by \cref{eq:normsmallopboundr} and \cref{eq:normsmallopwaway} we see that 
	\[
		\tau h\del_{x_n}P_sv = \lim_{m\to\infty}\tau h\del_{x_n}P_s(v-u_m) + \lim_{m\to\infty}\tau h\del_{x_n}P_su_m = 0\,.
	\]
			
Finally, using \cref{eq:boundarycalcwithgamma}, we invoke the representation
\[
	\bdry = (1+\abs{K}^2)^{-1/2}(hK\cdot\nabla_{x'} -(1+\abs{K}^2)h\del_{x_n}-1)\,,
\]
where because $\tau$ and $\nabla_{x'}$ commute when applied to an object with trace $0$, 
\[
	\tbdry P_sv =0\,.
\]
which concludes the proof.
\end{proof}
 Lemma \ref{lemmaOpNoBoundary} states that while the operator $P_s$ provides us with a parametrix in the small frequency, it does not allow us to prescribe non-homogeneous boundary conditions. In order to remedy this, we will construct an operator which acts as a Poisson kernel which (up to first order) gets killed by $h^2\tilde \Delta_\phi$ but allows us to prescribe non-homogeneous boundary condition:
\begin{proposition}\label{mappingPropsoft}
Denoting by $I_+$ the indicator of $\RR^n_+$. There is a smooth function $\ell (x,\xi')$ satisfying that $(x',\xi')\mapsto \ell(x',x_n,\xi') $ is a symbol in $S^{-\infty}_1(\RR^{n-1})$ for every fixed value of $x_n \geq 0$ such that

	\begin{equation}\label{eq:remainderforsmall}
		I_+h^2\td{\Delta}_\phi \ell(x',x_n, hD') = hR_{\ell}
	\end{equation}
	for some $R_{\ell} \colon L^r(\RR^{n-1}) \to_{h^0} C(\RR_+;L^r(\RR^{n-1})), 1<r<\infty$, and  
	\begin{equation}\label{boundaryofsmall}
		\tbdry \op[h,x']{\ell} = -(1+\abs{K}^2)^{1/2}\td{\rho}(hD') + hR_b\,,
	\end{equation}
	for some $R_b \colon L^r(\RR^{n-1}) \to_{h^0} L^r(\RR^{n-1})$.

	Furthermore, for every open, bounded set $A\subset \RR^n_+$, 
	\begin{equation}\label{eq:ellsemiclass}
		1_{A}\op[h,x']{\ell} \colon L^r(\RR^{n-1}) \to W^{2,r}(A)\,,\quad 1<r<\infty\,.
	\end{equation}
\end{proposition}
This section will be devoted to proving this Proposition. First we begin with some preliminary results on solving for a initial value problem in ODE.

\subsubsection{ODE Setup}\label{subsec:ode}

To motivate the discussion in this section, we wish to construct approximate solutions to $h^2\td{\Delta}_\phi w = 0$ which approximates the boundary condition
\[
	\bdry w = f\,, \quad \text{at}\  x_n= 0
\]
for an arbitrary $f\in L^2(\RR^{n-1})$.

Following \cite[p.~108]{CT20}, we will take Fourier transform in the $x'$ variable and view this equation as an ODE in the $x_n$ variable. This approach is fruitful because on the support of $\td{\rho}(\xi')$, 
the symbol $p$ from \cref{eq:defofsymbolp} is elliptic. According to \cite[p.~108]{CT20} we have
\begin{align}
	h^2\tilde{\Delta}_\phi &= -(1+\abs{K}^2)h^2\del_{x_n}^2 - 2(1_h+K\cdot h\nabla_{x'})h\del_{x_n} - h^2\Delta_{x'}-1\notag\\
	&= \op[h,x']{-(1+\abs{K}^2)h^2\del_{x_n}^2 - 2(1_h+iK\cdot\xi')h\del_{x_n} -(1-\abs{\xi'}^2)}\,,\label{eq:deltatildeop}
\end{align}
which leads us to try finding a solution $v$ of
\begin{equation}
	\left[(1+\abs{K}^2)h^2\del_{x_n}^2 + 2(1_h+iK\cdot\xi')h\del_{x_n} + (1-\abs{\xi'}^2)\right]v = 0\,.\label{ODEeqn}
\end{equation}
Viewing $(x',\xi')$ as fixed parameters for the ODE in the $x_n$ variable, \cref{ODEeqn} has solutions of the form
\begin{equation}\label{eq:firstv}
	v(x',x_n,\xi') = k_1e^{-x_n\frac{b+\sqrt{b^2-4ac}}{2a}} + k_2e^{-x_n\frac{b-\sqrt{b^2-4ac}}{2a}}
\end{equation}
where
\begin{equation}\label{eq:defofabc}
	a(x',\xi') = (1+\abs{K}^2)h^2\,, \quad b(x',\xi')= 2(1_h+iK\cdot\xi')h\,,\quad c(x',\xi')= 1-\abs{\xi'}^2\,.
\end{equation}
and $k_1, k_2$ depend on $x'$ and $\xi'$ but not on $x_n$.
To ease notation, taking the principal square root, we introduce a new variable for a part of the exponent,
\begin{equation}\label{eq:defofw}
	{q} \coloneqq \frac{\sqrt{b^2 -4ac}}{2a} = h^{-1}\frac{\sqrt{(1_h+iK\cdot\xi')^2-(1-\abs{\xi'}^2)(1+\abs{K}^2)}}{(1+\abs{K}^2)}\,,
\end{equation}
which might at first seem problematic because it is not smooth, but will turn out to be benign: the expression $q$ will only be present in the argument of the holomorphic, even function $\cosh$.

Now \cref{eq:firstv} gives us a solution ansatz but we must still choose $k_1$ and $k_2$ in such a way that we generate desirable boundary conditions. 
By \cref{defofboundary} we have that the $\bdry$ operator can be written as
\begin{equation}\label{eq:symbolofboundaryconditions}
	\bdry = \op[h,x']{(1+\abs{K}^2)^{-1/2}(iK\cdot \xi' -(1+\abs{K}^2)h\del_{x_n} - 1)}\,.
\end{equation}
And note that $v$ as given in \cref{eq:firstv},
\begin{equation}\label{eq:onederivofv}
	\del_{x_n}v(x',x_n)= e^{-x_n\frac{b}{2a}}\left(-k_1\left({q}+\frac{b}{2a}\right)e^{-x_n{q}}+k_2\left({q}-\frac{b}{2a}\right)e^{x_n{q}}\right)\,.
\end{equation}

For arbitrary $f\in L^2(\RR^{n-1})$ we need to choose $k_1, k_2$ so that $v$ satisfies the boundary condition $\bdry \mathcal F_{x'}^{-1} (v) = f$ at $x_n = 0$. The expression \eqref{eq:symbolofboundaryconditions} motivates us to consider, on the level of symbols, the equation
\[ 
	(1+\abs{K}^2)^{-1/2}(iK\cdot \xi' -(1+\abs{K}^2)h\del_{x_n} - 1)v(x) = -(1+\abs{K}^2)^{1/2}\mathcal{F}_{x'}f\,.
\]
Using \cref{eq:onederivofv} this condition is equivalent to
\begin{gather}
	 -(1+\abs{K}^2)^{1/2}\mathcal{F}_{x'}f\label{eq:smallboundarywewant}\\
	= 
	(1+\abs{K}^2)^{1/2}e^{-x_n\frac{b}{2a}}\bigg[\left((1+\abs{K}^2)^{-1}(iK\cdot\xi'-1)+\frac{b}{2a}\right) (k_1e^{-x_n{q}}+k_2e^{x_n{q}}) + h{q}(k_1e^{-x_n{q}}-k_2e^{x_n{q}})\bigg]\notag\,,
\end{gather}
 at $x_n=0$.

If we choose $k_2=-k_1$, then \cref{eq:smallboundarywewant} can be rephrased as 
\[
	2k_1hq = -\mathcal{F}_{x'}f\,,
\]
for which we will choose
\[
	k_1 \coloneqq -\frac{1}{2h}{q}^{-1}\mathcal{F}_{x'}f
\]
which leads to
\begin{align}
\label{multiplier}
	v(x',x_n) &=
 h^{-1}e^{-x_n\frac{b}{2a}}\int_0^{x_n}\cosh(t {q})\dd t \mathcal{F}_{x'}f\,.
\end{align}

As we can only reasonably hope for this approach to work on small frequencies, we will introduce the cutoff $\td{\rho} \in C_c^\infty(\RR^{n-1})$ from \cref{defofrho}, and 
motivated by the multiplier of \eqref{multiplier}, we define
\begin{align}
	\ell(x,\xi') &\coloneqq  h^{-1}\td{\rho}e^{-x_n\frac{b}{2a}}\int_0^{x_n}\cosh(t {q})\dd t \label{defOft}\\
	&= h^{-1}\td{\rho}(\xi') e^{-\frac{x_n}{h}\frac{1_h+iK\cdot\xi'}{1+\abs{K}^2}}  \int_0^{x_n} \cosh\left(\frac{t}{h}\frac{\sqrt{(1_h+iK\cdot\xi')^2-(1-\abs{\xi'}^2)(1+\abs{K}^2)}}{1+\abs{K}^2}\right) \dd t\,. \notag 
\end{align}

For any $s \in [1,\infty]$ and $k\in \mathbb{N}_0$, we endow the space $C(\RR_+, W^{k,s}(\RR^{n-1}))$ with the norm
\begin{equation}\label{eq:normonC}
	\norm{u} = \sup_{x_n \geq 0} \norm{u}_{W^{k,s}(\RR^{n-1})}\,.
\end{equation}

\subsubsection{The Symbol $\ell$}
We now prove that $\ell$ defined in \eqref{defOft} is a symbol in the sense that:

\begin{proposition}
\label{prop: ell is symbol}
	Let $h>0$ be small enough. The expression $\ell$ defined in \cref{defOft} is a symbol $\ell\in S^{-\infty}_1(\RR^{n-1})$ for every fixed value of $x_n \geq 0$ and as a function, $\ell$ depends smoothly on $x_n$. Furthermore, 
	\begin{equation}\label{mappingpropsofsmallteqn}
		h^s\del_{x_n}^s\op[h,x']{\ell} \colon L^r(\RR^{n-1}) \to_{h^0} C(\RR_+; W^{k,r}(\RR^{n-1}))\,,\quad  k\in\mathbb{N},\quad 1<r<\infty,\quad s \in \{0,1,2\}\,,
	\end{equation}
	with the norm on $C(\RR_+; W^{k,r}(\RR^{n-1}))$ defined by \cref{eq:normonC}.
\end{proposition}
We will prove this proposition at the end of the subsection as a consequence of the following technical Lemmas.
	\begin{lemma}\label{lemm:exponentboundedbelow}
		There exists an $\varepsilon_1 > 0$ so that for $h > 0$ small and $\xi'\in \supp\td{\rho}$,
		\begin{equation}\label{eq:theboundonexpwewant}
			\mathrm{Re}\left[1_h - \sqrt{(1_h+iK\cdot\xi')^2 - (1-\abs{\xi'}^2)(1+\abs{K}^2)}\right] > \frac{\varepsilon_1}{2}\,,
		\end{equation}
		where $\varepsilon_1$ is independent of $h,x$ and $\xi'$.
	\end{lemma}

\begin{proof}
	Note 
	that since we are taking the principal square root, we have $\mathrm{Re}\sqrt{d+is} = \sqrt{\frac{\sqrt{d^2+s^2}+d}{2}}$ for any $s,d\in\RR$. We will show that there exists some $\varepsilon_1 >0$ so that the left-hand side of \cref{eq:theboundonexpwewant} is bounded below by $\varepsilon_1$ on $\supp\td{\rho}$ if we take $h=0$. {By continuity of the real part of the principal square root}, this means that the left-hand side of \cref{eq:theboundonexpwewant} is uniformly greater than $\varepsilon_1/2$ for small $h>0$.	
	For the moment put $h=0$ so that $1_h = 1$, where $1_h$ is defined in \cref{subsec:recap}.
	Taking 
	\begin{align}
		d &= \mathrm{Re}((1_h+iK\cdot\xi')^2 - (1-\abs{\xi'}^2)(1+\abs{K}^2)) = 1-(K\cdot\xi')^2- (1-\abs{\xi'}^2)(1+\abs{K}^2)\,,\\ 
		s &= \mathrm{Im}((1_h+iK\cdot\xi')^2 - (1-\abs{\xi'}^2)(1+\abs{K}^2)) = 2K\cdot\xi'\,,
	\end{align}
	we see that 
	\begin{align}
		 \varepsilon_1 &< \mathrm{Re}\left(1_h - \sqrt{(1_h+iK\cdot\xi')^2 - (1-\abs{\xi'}^2)(1+\abs{K}^2)}\right) \notag
	\shortintertext{if and only if}
		1-\varepsilon_1 &> \sqrt{\frac{\sqrt{d^2+s^2}+d}{2}}\,, \notag
	\shortintertext{which is equivalent to}
	2(1-\varepsilon_1)^2-d &> \sqrt{d^2+s^2}\,.\label{sqrtEqn}
	\end{align}
	Since for small enough $\varepsilon_1>0$ and all $\xi' \in \supp\td{\rho}$,
	\[
	2(1-\varepsilon_1)^2-d = -4\varepsilon_1+2\varepsilon_1^2 +1+(K\cdot\xi')^2+(1-\abs{\xi'}^2)(1+\abs{K}^2) > 0\,,
	\]
	\cref{sqrtEqn} holds if and only if 
	\begin{align}
		(2(1-\varepsilon_1)^2-d)^2 &>d^2+s^2\,, \notag
	\shortintertext{which is equivalent to}
		4(1-d) -\mathcal{O}(\varepsilon_1) &> s^2\,,\label{someothereqn}
	\end{align}
	where in introducing $\mathcal{O}$ we used the fact that both $\abs{K}^2$ and $\abs{\xi'}$ are bounded on the support of $\td{\rho}$.
	
	Finally, \cref{someothereqn} is equivalent to 
	\[
		 4(1-\abs{\xi'}^2)(1+\abs{K}^2) -\mathcal{O}(\varepsilon_1) > 0\,,
	\]
	and since the support of $\td{\rho}$ was chosen in such a way that $\supp \td{\rho} \Subset\{\abs{\xi'} < 1\}$, we can be sure that such an $\varepsilon_1$ exists.
\end{proof}

	\begin{lemma}\label{lemm:goodderofcosh}
	Let $h>0$ be small enough and $\xi' \in \supp \td{\rho}$. There exist some $\varepsilon_2>0$ and $C_{\gamma,\beta} >0$ independent of $h,x,\xi'$ so that 
	\begin{equation}\label{eq:coshgood}
	\abs{\del^{\gamma}_{x'}\del_{\xi'}^\beta \cosh(tq)} 
	\leq C_{\gamma,\beta}e^{\frac{t}{h}\frac{1_h - \varepsilon_2}{1+\abs{K}^2}}\max\left\{1,(th^{-1})^{2\abs{\gamma}+2\abs{\beta}}\right\}\,
	\end{equation}
	for all multi-indices $\gamma$ and $\beta$.
	\end{lemma}

\begin{proof}
	We introduce $q'(x',\xi')$ by
	\begin{equation}\label{eq:defofqprime}
		q'(x',\xi') \coloneqq \frac{(1_h+iK\cdot\xi')^2-(1-\abs{\xi'}^2)(1+\abs{K}^2)}{(1+\abs{K}^2)^2}\in C^\infty_b(\RR^{n-1}\times\supp\td{\rho})
	\end{equation}
	so that $ tq = \frac{t}{h}\sqrt{q'}$.
Here $C_b^\infty$ is the space of functions whose derivative of every order is bounded uniformly in $h>0$ small. That every derivative of $q'$ is bounded follows from the definition $K(x') = \nabla g(x')$ where $g$ is the boundary defining function introduced in \cref{changeofVars}.

	We will find two formulas for $\del^{\gamma}_{x'}\del_{\xi'}^\beta \cosh(\frac{t}{h}\sqrt{q'})$ and use each to give a slightly different bound which we combine at the end of this proof.

	Using \cref{lemm:exponentboundedbelow}, 
	we see that for $h>0$ small, $\xi'\in\supp\td{\rho}$ there is an $\varepsilon_1 >0$ so that
	\begin{equation}\label{eq:realq}
		\mathrm{Re}\sqrt{q'} < \frac{1_h - \varepsilon_1/2}{1+\abs{K}^2}\,.
	\end{equation}
	
	\textbf{\textsc{Case 1:}} We have that $\abs{\mathrm{Im}\sqrt{q'}} < \frac{\varepsilon_1}{4(1+\abs{K}^2)}$ in which case we see by \cref{eq:realq} that 
		\begin{equation}\label{eq:absq}
			\abs{\sqrt{q'}} \leq \abs{\mathrm{Re}\sqrt{q'}} + \abs{\mathrm{Im}\sqrt{q'}} < \frac{1_h-\varepsilon_1/4}{1+\abs{K}^2}\,.
		\end{equation}
	
		Now we must do a calculation: for $\sigma,\kappa \in \RR, \omega\in\mathbb{N}_0$, 
	\begin{equation}\label{eq:quickcalc}
		\del_{\kappa}^{\omega} \cosh(\sigma\sqrt{\kappa}) = \sigma^{2\omega}\sum_{k\geq \omega} (\sigma^2\kappa)^{k-\omega}\frac{k!}{(2k)!(k-\omega)!} = \sigma^{2\omega}\sum_{k\geq 0} \frac{(\sigma\sqrt{\kappa})^{2k}}{(2k)!} \frac{(k+\omega)!(2k)!}{k!(2k+2\omega)!}\,,
	\end{equation}
	where we note that for all $k,\omega\in\mathbb{N}_0$, $\frac{(k+\omega)!(2k)!}{k!(2k+2\omega)!}\leq 1$.
	
	For every multi-index $\alpha\in\mathbb{N}_0^{2n-2}$, let
	\begin{equation}\label{eq:defPialpha}
		\Pi_\alpha \coloneqq \{\text{partitions of } \{1,\dots,\abs{\alpha}\}\}\,,
	\end{equation}
	and for every $\pi \in \Pi_\alpha$ and every $P\in \pi$, let $\alpha_P \in \mathbb{N}_0^{2n-2}$ be the multi-index with 
	\begin{equation}
		(\alpha_P)_j = \abs{P\cap \left\{1+\sum_{i=1}^{j-1} \alpha_i, \dots, \sum_{i=1}^{j}\alpha_i\right\}}\,.\label{eq:faaindex}
	\end{equation}

	Using \cref{eq:quickcalc} and Fa\`a di Bruno's formula (with notation introduced in \cref{eq:defPialpha,eq:faaindex}), we calculate that 
	\begin{align}
		\del^{\gamma}_{x'}\del_{\xi'}^\beta \cosh\left(\frac{t}{h}\sqrt{q'}\right) = \sum_{\pi_1\in\Pi_{(\gamma,\beta)}} \left(\frac{t}{h}\right)^{2\abs{\pi_1}}\sum_{k\geq 0} \frac{(th^{-1}\sqrt{q'})^{2k}}{(2k)!} \frac{(k+\abs{\pi_1})!(2k)!}{k!(2k+2\abs{\pi_1})!}
		\prod_{P_1\in\pi_1} \del^{(\gamma,\beta)_{P_1}} q'\label{eq:imqsmall}
	\end{align}

	Now, by \cref{eq:imqsmall}, we have that 
	\begin{align}\label{eq:qsmallabs}
		\abs{\del^{\gamma}_{x'}\del_{\xi'}^\beta \cosh\left(\frac{t}{h}\sqrt{q'}\right)} &\leq  \sum_{\pi_1\in\Pi_{(\gamma,\beta)}} \left(\frac{t}{h}\right)^{2\abs{\pi_1}}\sum_{k\geq 0} \frac{(th^{-1}\abs{\sqrt{q'}})^{2k}}{(2k)!} \frac{(k+\abs{\pi_1})!(2k)!}{k!(2k+2\abs{\pi_1})!}
		\prod_{P_1\in\pi_1} \abs{\del^{(\gamma,\beta)_{P_1}} q'}  \notag\\
		&\leq\sum_{\pi_1\in\Pi_{(\gamma,\eta)}} \left(\frac{t}{h}\right)^{2\abs{\pi_1}}\sum_{k\geq 0} \frac{(th^{-1}\abs{\sqrt{q'}})^{2k}}{(2k)!}
		\prod_{P_1\in\pi_1} \abs{\del^{(\gamma,\beta)_{P_1}} q'} \notag \\
		&= \sum_{\pi_1\in\Pi_{(\gamma,\beta)}} \left(\frac{t}{h}\right)^{2\abs{\pi_1}}\cosh(th^{-1}\abs{\sqrt{q'}}) \prod_{P_1\in\pi_1} \abs{\del^{(\gamma,\beta)_{P_1}} q'} \notag \\
		&\leq \sum_{\pi_1\in\Pi_{(\gamma,\beta)}} C_{\pi_1}' \left(\frac{t}{h}\right)^{2\abs{\pi_1}} e^{th^{-1}\abs{\sqrt{q'}}}
	\end{align}
	for some $C_{\pi_1}'\geq 0$ independent of $h,x$ and $\xi'$ in view of \cref{eq:defofqprime}. Using \cref{eq:absq} and \cref{eq:qsmallabs} we see that
	\begin{align}\label{eq:smallimdone}
		\abs{\del^{\gamma}_{x'}\del_{\xi'}^\beta \cosh\left(\frac{t}{h}\sqrt{q'}\right)} 
		&\leq \sum_{\pi_1\in\Pi_{(\gamma,\beta)}} C_{\pi_1}' \left(\frac{t}{h}\right)^{2\abs{\pi_1}} e^{\frac{t}{h}\frac{1_h-\varepsilon_1/4}{1+\abs{K}^2}}\,.
	\end{align}
	We suspend working in this case until the end of the proof when we make statements valid in both cases.

	\textsc{\textbf{Case 2:}} We have that $\abs{\mathrm{Im}\sqrt{q'}} \geq \frac{\varepsilon_1}{4(1+\abs{K}^2)}$. 

	Again, using Fa\`a di Bruno's formula, we can calculate that
	\begin{gather}
		\del^{\gamma}_{x'}\del_{\xi'}^\beta \cosh\left(\frac{t}{h}\sqrt{q'}\right) = \notag \\
		\sum_{\pi_1 \in\Pi_{(\gamma,\beta)}}(th^{-1})^{\abs{\pi_1}}\cosh^{(\abs{\pi_1})}\left(\frac{t}{h}\sqrt{q'}\right)\prod_{P_1\in\pi_1}\sum_{\pi_2\in \Pi_{(\gamma,\beta)_{P_1}}}(q')^{1/2-\abs{\pi_2}}(-1)^{\abs{\pi_2}+1}\prod_{s=0}^{\abs{\pi_2}-1}\left(\frac{1}{2}+s\right)\prod_{P_2\in\pi_2} \del^{((\gamma,\beta)_{P_1})_{P_2}} q' \label{eq:imqlarge}
	\end{gather} 

	Now note that for any $z \in \mathbb{C}$
	\begin{equation}\label{eq:coshandsinh}
		\abs{\sinh(z)} \leq e^{\abs{\mathrm{Re}z}}\,,\an \abs{\cosh(z)} \leq e^{\abs{\mathrm{Re}z}}\,,
	\end{equation}
	and that $\abs{\sinh(z)}, \abs{\cosh(z)}$ are monotonously increasing in $\mathrm{Re}z \geq 0$. 

	Using \cref{eq:coshandsinh}, we remark that if $(\gamma,\beta)=0$, we may bound 
	\begin{equation}
		\abs{\cosh(th^{-1}\sqrt{q'})} \leq e^{th^{-1}\mathrm{Re}\sqrt{q'}}\,. \label{eq:zerocase}
	\end{equation}
	and thus the lemma is already proven by combining \cref{eq:zerocase} and \cref{eq:realq}. Let us therefore take $(\gamma,\beta)\neq 0$ from now on.

	In view of \cref{eq:imqlarge}, \cref{eq:coshandsinh} and $\abs{\sqrt{q'}} \geq \abs{\mathrm{Im}{\sqrt{q'}}}$, for $(\gamma,\eta)\neq 0$, we have that 
	\begin{align}\label{eq:qlargeabs}
	&\abs{\del^{\gamma}_{x'}\del_{\xi'}^\eta \cosh\left(\frac{t}{h}\sqrt{q'}\right)} \notag \\ 
	&\leq \sum_{\pi_1 \in\Pi_{(\gamma,\eta)}}\left(\frac{t}{h}\right)^{\abs{\pi_1}}e^{th^{-1}\mathrm{Re}\sqrt{q'}}\prod_{P_1\in\pi_1}\sum_{\pi_2\in \Pi_{(\gamma,\eta)_{P_1}}}\abs{\mathrm{Im}\sqrt{q'}}^{1-2\abs{\pi_2}}\prod_{s=0}^{\abs{\pi_2}-1}\left(\frac{1}{2}+s\right)\prod_{P_2\in\pi_2} \abs{\del^{((\gamma,\eta)_{P_1})_{P_2}} q'} \notag \\
	&\leq \sum_{\pi_1 \in\Pi_{(\gamma,\eta)}}\left(\frac{t}{h}\right)^{\abs{\pi_1}}e^{th^{-1}\mathrm{Re}\sqrt{q'}}\prod_{P_1\in\pi_1}\sum_{\pi_2\in \Pi_{(\gamma,\eta)_{P_1}}}C_{\pi_1,\pi_2}\abs{\mathrm{Im}\sqrt{q'}}^{1-2\abs{\pi_2}} 
	\end{align}
	where $C_{\pi_1,\pi_2} \geq 0$ is independent of $h,x,\xi'$, and we used the fact that $q'$ and all of its derivatives are bounded. 
	
	We continue the calculation \cref{eq:qlargeabs} using \cref{eq:realq} and the assumption of this case:
		\begin{align}\label{eq:largeimdone}
			\abs{\del^{\gamma}_{x'}\del_{\xi'}^\beta \cosh\left(\frac{t}{h}\sqrt{q'}\right)} 
&\leq \sum_{\pi_1 \in\Pi_{(\gamma,\beta)}}\left(\frac{t}{h}\right)^{\abs{\pi_1}}e^{\frac{t}{h} \frac{1_h - \varepsilon_1/2}{1+\abs{K}^2}}\prod_{P_1\in\pi_1}\sum_{\pi_2\in \Pi_{(\gamma,\beta)_{P_1}}}C_{\pi_1,\pi_2}\abs{\frac{\varepsilon_1}{4(1+\abs{K}^2)}}^{1-2\abs{\pi_2}} \notag\\
&\leq \sum_{\pi_1 \in\Pi_{(\gamma,\beta)}}C''_{\pi_1}\left(\frac{t}{h}\right)^{\abs{\pi_1}}e^{\frac{t}{h}\frac{1_h - \varepsilon_1/2}{1+\abs{K}^2}}\,.
		\end{align}

	At this point, we begin making statements about both cases simultaneously.

	Combining \cref{eq:smallimdone} and \cref{eq:largeimdone}, there exist some $\varepsilon_2,C_{\gamma,\beta} >0$ independent of $h,x,\xi'$ with $\varepsilon_2$ also independent of $\gamma,\beta$, so that 
	\begin{equation}
	\abs{\del^{\gamma}_{x'}\del_{\xi'}^\beta \cosh\left(\frac{t}{h}\sqrt{q'}\right)} 
	\leq C_{\gamma,\beta}e^{\frac{t}{h}\frac{1_h - \varepsilon_2}{1+\abs{K}^2}}\sum_{\pi_1 \in\Pi_{(\gamma,\beta)}}\max\{(th^{-1})^{\abs{\pi_1}},(th^{-1})^{2\abs{\pi_1}}\}\,.
	\end{equation}
	Finally, recalling from \cref{eq:defofqprime} that $\frac{t}{h}\sqrt{q'} = tq$, we have 
	\begin{equation}
	\abs{\del^{\gamma}_{x'}\del_{\xi'}^\beta \cosh(tq)} 
	\leq C'_{\gamma,\beta}e^{\frac{t}{h}\frac{1_h - \varepsilon_2}{1+\abs{K}^2}}\max\{1,(th^{-1})^{2\abs{\gamma}+2\abs{\beta}}\}
	\end{equation} 
	for a possibly different $C_{\gamma,\beta}' \geq 0$ independent of $h, x$ and $\xi'$.
\end{proof}

\begin{proof}[Proof of Proposition~\ref{prop: ell is symbol}]
	Checking that $\ell(x_n, x', \xi') = h^{-1}\td{\rho}e^{-x_n\frac{b}{2a}}\int_0^{x_n}\cosh(t {q})\dd t$ is a symbol in the $(x',\xi')$ variable begins by determining that it is smooth. 
	Let $x_n \geq 0$ be fixed. Due to the fact that $\cosh$ is an even function, the square root that is contained within ${q}$ is not an issue so that $\cosh(t{q})$ is a smooth function with respect to $t,x'$ and $\xi'$ (recall that the Taylor expansion of $\cosh$ only has even exponents). 
	All other parts of $\ell$ are clearly smooth and so $\ell \in C^\infty(\RR^n\times\RR^{n-1})$.
	
	Now, with notation from \cref{eq:defPialpha,eq:faaindex}, by the Fa\`a di Bruno formula and the Leibniz rule, for multi-indices $\alpha,\beta \in \mathbb{N}_0^{n-1}$, 
	\begin{align}\label{eq:derivofell2} 
	\del^\alpha_{x'}\del^\beta_{\xi'} \ell(x,\xi') &= \sum_{\substack{\gamma\leq \alpha, \delta\leq \beta, \eta\leq \beta-\delta\\\pi\in \Pi_{(\alpha-\gamma,\beta-\delta)}}} {\alpha \choose \gamma}{\beta \choose \delta}{\beta-\delta \choose \eta} e^{-\frac{x_n}{h}\frac{1_h+iK\cdot\xi'}{1+\abs{K}^2}} \prod_{P \in \pi}\left(\frac{x_n}{h}\right)^{\abs{\pi}} \del_{(x',\xi')}^{(\alpha-\gamma,\beta-\delta)_{P}}\left(-\frac{1_h +iK\cdot\xi'}{1+\abs{K}^2}\right) \notag\\
	&\qquad\cdot \left(\del_{\xi'}^{\beta-\delta-\eta}\td{\rho}\right) h^{-1} \int_0^{x_n} \del^{\gamma}_{x'}\del_{\xi'}^\eta \cosh(tq) \dd t\,.
	\end{align}

	Applying monotonicity of the right-hand side of \cref{eq:coshgood} with respect to $t\geq 0$, for $h>0$ small and $\xi'\in\supp\td{\rho}$,	
	\begin{align}
	h^{-1}\abs{\int_0^{x_n} \del^{\gamma}_{x'}\del_{\xi'}^\eta \cosh(tq) \dd t} &\leq h^{-1} C_{\gamma,\eta}e^{\frac{x_n}{h}\frac{1_h - \varepsilon_2}{1+\abs{K}^2}}\max\{1,(x_nh^{-1})^{2\abs{\gamma}+2\abs{\eta}}\}\int_0^{x_n}1\dd t\notag \\
	&\leq C_{\gamma,\eta}e^{\frac{x_n}{h}\frac{1_h - \varepsilon_2}{1+\abs{K}^2}}\max\left\{1,\left(\frac{x_n}{h}\right)^{2\abs{\gamma}+2\abs{\eta}+1}\right\}\,.\label{eq:monoton}
	\end{align}

	Now by \cref{eq:monoton}, and the fact that $\left(\del_{\xi'}^{\beta-\delta-\eta}\td{\rho}\right)$ is supported on $\supp\td{\rho}$,
	\begin{align}
		\Bigg\vert\left(\del_{\xi'}^{\beta-\delta-\eta}\td{\rho}\right)&\left.e^{-\frac{x_n}{h}\frac{1_h+iK\cdot\xi'}{1+\abs{K}^2}} \left(\frac{x_n}{h}\right)^{\abs{\pi}}h^{-1} \int_0^{x_n} \del_{x'}^\gamma\del_{\xi'}^\delta \cosh(t{q}) \dd t\right\vert \notag
		\\
		&\leq C_{\gamma,\eta}\abs{\del_{\xi'}^{\beta-\delta-\eta}\td{\rho}}\max\left\{1,\left(\frac{x_n}{h}\right)^{2\abs{\gamma}+2\abs{\eta}+1+\abs{\pi}}\right\}e^{-\frac{x_n}{h}\frac{\varepsilon_2}{(1+\abs{K}^2)}}\,.\label{eq:thesupremumofexppart}
	\end{align}
	Furthermore, we find that 
	\begin{gather}
		\sup_{x_n\geq 0}\max\left\{1,\left(\frac{x_n}{h}\right)^{2\abs{\gamma}+2\abs{\eta}
		+1+\abs{\pi}}\right\}e^{-\frac{x_n}{h}\frac{\varepsilon_2}{(1+\abs{K}^2)}} \notag\\
		= \max\left\{1,\left(\frac{(1+\abs{K}^2)(2\abs{\gamma}+2\abs{\eta}+1+\abs{\pi})}{\varepsilon_2 e}\right)^{2\abs{\gamma}+2\abs{\eta}+1+\abs{\pi}}\right\}\,,\label{eq:supofpolyvsexp}
	\end{gather}
	and so combining \cref{eq:derivofell2}, \cref{eq:thesupremumofexppart} and \cref{eq:supofpolyvsexp},
	\begin{align}\label{eq:finalderivbound}
	\sup_{x_n\geq 0}\abs{\del^\alpha_{x'}\del^\beta_{\xi'} \ell(x,\xi')} &\leq\sum_{\substack{\gamma\leq \alpha, \delta\leq \beta, \eta\leq \beta-\delta\\\pi\in \Pi_{(\alpha-\gamma,\beta-\delta)}}} {\alpha \choose \gamma}{\beta \choose \delta}{\beta-\delta \choose \eta} \prod_{P \in \pi}\abs{\del_{(x',\xi')}^{(\alpha-\gamma,\beta-\delta)_{P}}\left(-\frac{1_h +iK\cdot\xi'}{1+\abs{K}^2}\right)} \notag \\ 
	&\cdot C_{\gamma,\eta}\max\left\{1,\left(\frac{(1+\abs{K}^2)(2\abs{\gamma}+2\abs{\eta}+1+\abs{\pi})}{\varepsilon_2 e}\right)^{2\abs{\gamma}+2\abs{\eta}+1+\abs{\pi}}\right\}\notag\\
	&\leq C_{\alpha,\beta}\max_{\eta\leq \abs{\beta}}\abs{\del_{\xi'}^{\eta}\td{\rho}}\,,
	\end{align}
	where we used the fact that $\abs{K}$ is bounded above and we consumed the last dependence in $h$ (in the form of $1_h$, see \cref{subsec:recap}) by bounding it above. The constant $C_{\alpha,\beta}\geq 0$ in \cref{eq:finalderivbound} is independent of $h, x$ and $\xi'$. Because $\td{\rho}$ is compactly supported, for every fixed $x_n \geq 0$, \cref{eq:finalderivbound} implies that $\ell \in S^{-\infty}_1(\RR^{n-1})$. 
	
	We proceed to showing \cref{mappingpropsofsmallteqn} for $s=0$. Note that \cref{eq:finalderivbound} proves that $\{\ell(\cdot,x_n,\cdot)\colon x_n\geq 0\}$ is a bounded subset of $S^{-\infty}_1(\RR^{n-1})$ each element of which satisfies \cref{eq:goodsymbol}. Here, bounded means that there exist constants $C_{\alpha,\beta}$ so that for all multi-indices $\alpha,\beta$ we have 
	\begin{equation}
		\sup_{x_n\geq 0}\abs{\del^\alpha_{x'}\del^\beta_{\xi'}\ell(x',x_n,\xi')} \leq C_{\alpha,\beta}\,.\label{eq:wahtisboundedness}
	\end{equation}
	Using \cref{lemm:semi}, \cref{eq:wahtisboundedness} directly implies that the operator norm of $\op[h,x']{\ell}$ on $L^r(\RR^{n-1})\to W^{k,r}(\RR^{n-1})$ as a function of $x_n \geq 0$ is uniformly bounded with respect to $x_n \geq 0$, which proves \cref{mappingpropsofsmallteqn} for $s=0$.

	We proceed to checking the first derivative of $\ell$. We calculate
	\begin{equation}\label{eq:ellfirstderiv}
		h\del_{x_n}\ell(x,\xi') = \td{\rho}e^{-\frac{x_n}{h}\frac{1_h +iK\cdot\xi'}{1+\abs{K}^2}}\left(\cosh(x_n {q})-h^{-1}\frac{1_h +iK\cdot\xi'}{1+\abs{K}^2}\int_0^{x_n} \cosh(t{q})\dd t\right)\,,	
	\end{equation}
	where the second term of \cref{eq:ellfirstderiv} can be rewritten as 
	\begin{equation}\label{eq:secondtermellfirstderiv}
		-\frac{1_h +iK\cdot\xi'}{1+\abs{K}^2}\ell(x,\xi')\,,
	\end{equation}
	whose seminorms in $S^{-\infty}_1(\RR^{n-1})$ can be bounded by the seminorms of $\ell$ multiplied by a constant independent of $x_n$.
	
	Let us consider the first term of \cref{eq:ellfirstderiv}. So that we do not have to rewrite the entire computation, note that for multi-indices $\alpha,\beta$, 
	\begin{align}\label{eq:ellfirstderifirstpart}
		\del_{x'}^\alpha \del_{\xi'}^\beta\left(\td{\rho}e^{-\frac{x_n}{h}\frac{1_h +iK\cdot\xi'}{1+\abs{K}^2}}\cosh(x_n {q})\right)
	\end{align}
	is equal to the expression in \cref{eq:derivofell2} where we replace 
	\[
		h^{-1}\int_0^{x_n} \del^{\gamma}_{x'}\del^{\eta}_{\xi'}\cosh(tq) \dd t\qquad \text{by}\qquad \del^{\gamma}_{x'}\del^{\eta}_{\xi'}\cosh(x_nq)\,.
	\]
	Using \cref{eq:coshgood} we come to the same conclusion that for every $\alpha,\beta$, the expression \cref{eq:ellfirstderifirstpart} is bounded uniformly in $x_n \geq 0$ by a constant $C_{\alpha,\beta}$ times a compactly supported function in $\xi'$.

	Together with the remark after \cref{eq:secondtermellfirstderiv}, we thus find that $\{h\del_{x_n}\ell(\cdot,x_n,\cdot)\colon x_n\geq 0\}$ is a bounded subset of $S^{-\infty}_1(\RR^{n-1})$ each element of which satisfies \cref{eq:goodsymbol}, which implies \cref{mappingpropsofsmallteqn} for $s=1$ if we use \cref{lemm:semi}.
	
	Finally, note that by direct computation on the symbol, $(a\del_{x_n}^2 + b\del_{x_n} + c) \ell = 0$, see \cref{eq:defofabc} for the definitions of $a,b,c$. We use this to write that
	\[
		\{ h^2\del_{x_n}^2\ell(\cdot,x_n,\cdot)\colon x_n\geq 0\} = \left\{-\frac{2(1_h+iK\cdot\xi')h\del_{x_n}\ell(\cdot,x_n,\cdot) + (1-\abs{\xi'}^2)\ell(\cdot,x_n,\cdot)}{1+\abs{K}^2}\colon x_n \geq 0\right\}
	\]
	is a bounded subset of $S^{-\infty}_1(\RR^{n-1})$, so that we have fully shown \cref{mappingpropsofsmallteqn} by a final application of \cref{lemm:semi}.
\end{proof}
\subsubsection{Proof of Proposition \ref{mappingPropsoft}}

	By the definition of semiclassical Sobolev spaces, \cref{mappingpropsofsmallteqn} implies \cref{eq:ellsemiclass}.

	Defining
	\begin{align}
		hR_{\ell,1} &\coloneqq h^2\Delta_{x'} \op[h,x']{\ell} - \mathrm{Op}_{h,x'}(-\abs{\xi'}^2\ell)\,,\\
	\shortintertext{and} 
	hR_{\ell,2} &\coloneqq hK\cdot\nabla_{x'} \op[h,x']{\ell} - \mathrm{Op}_{h,x'}(iK\cdot\xi'\ell)\,,
	\end{align}
	by the compositional calculus of semiclassical pseudo-differential operators and \cref{mappingpropsofsmallteqn} we have $R_{\ell,1},R_{\ell,2}\colon L^r(\RR^{n-1}) \to_{h^0} C(\RR_+;L^r(\RR^{n-1}))$. 

	By direct calculation $(a \del_{x_n}^2 + b\del_{x_n} +c)\ell = 0$, see \cref{eq:defofabc}, and so by \cref{eq:deltatildeop},
	\[
		h^2\td{\Delta}_\phi \op[h,x']{\ell} = \op[h,x']{\left(a \del_{x_n}^2 + b\del_{x_n} +c\right)\ell} + h(R_{\ell,1}+R_{\ell,2}) = h(R_{\ell,1}+R_{\ell,2}) \eqqcolon hR_{\ell}\,,
	\]
	where $R_{\ell}\colon L^r(\RR^{n-1}) \to_{h^0} C(\RR_+;L^r(\RR^{n-1}))$ is as desired.
	
	Now, by \cref{eq:symbolofboundaryconditions} the symbol of $\bdry$ is affine with respect to $\xi'$, so that the composition calculus gives
	\begin{equation}\label{eq:boundaryellcalc}
	\bdry\op[h,x']{\ell} = \op[h,x']{(1+\abs{K}^2)^{-1/2}\left(iK\cdot \xi' -(1+\abs{K}^2)h\del_{x_n} -1 \right)\ell} + hR_b'
	\end{equation}
	for the remainder $R_b' = \op[h,x']{(1+\abs{K}^2)^{-1/2} K \cdot \nabla_{x'}\ell}$. 
	Taking $s=0$ in \cref{mappingpropsofsmallteqn} and using the explicit representation of $R_b'$ implies $R_b' \colon L^r(\RR^{n-1}) \to_{h^0} C(\RR_+; W^{1,r}(\RR^{n-1}))$.
	Furthermore, by using \cref{eq:ellfirstderiv} and directly calculating on symbols, 
	\[ 
		\tau (1+\abs{K}^2)^{-1/2}(iK\cdot\xi'-(1+\abs{K}^2)h\del_{x_n}-1)\ell(x,\xi') = - (1+\abs{K}^2)^{1/2}\td{\rho}(\xi')
	\] 
	and so using \cref{eq:boundaryellcalc},
	\[
		\tbdry\op[h,x']{\ell} f = -(1+\abs{K}^2)^{1/2}\td{\rho}(hD') f + \tau hR_b'f\,, 
	\]
	and taking $R_b \coloneqq \tau R_b'$ concludes the proof.
\qed

\subsection{Large Frequency Parametrix}\label{sec:large}
Following \cite{CT20}, we define
\begin{equation}\label{eq:definitionofPl}
	P_l \coloneqq (1-\td{\rho}(hD'))J^+J\td{G}_\phi\,,
\end{equation}
where $\td{G}_\phi$ is as in \cref{prop43}, $\td{\rho}$ is from \cref{defofrho}, $J$ from \cref{eq:defofJQ} and 
\begin{equation}\label{eq:defofJplus}
	J^+ \coloneqq J^{-1}I_+\,,
\end{equation}
with $J^{-1}$ defined in \cref{Jinverse}.

This operator is exactly what is called the large frequency parametrix in \cite{CT20}, which approximately solves
\[
h^2\tilde\Delta_\phi P_l\approx 1 - \tilde \rho(hD')\,.
\]
This operator is insufficient for our needs as it does not allow us to control the value of $\tbdry P_l$ at $x_n = 0$. Indeed, we will see at the end of this section that
\begin{proposition}
\label{Prop: bv of Pl}
The operator $P_l$ satisfies the boundary condition
\begin{equation}\label{traceofpl}
	\tbdry P_l = -(1+\abs{K}^2)^{1/2}(1-\td{\rho}(hD'))\tau J\td{G}_\phi\,.
\end{equation}
\end{proposition}

The identity in \cref{Prop: bv of Pl} suggests that we should introduce a Poisson kernel, which we will call $B_l$, to correct this boundary term. Furthermore, in order to achieve \cref{eq:boundaryG}, we will need another Poisson kernel that allows us to introduce arbitrary boundary data, which we will call $B_c$. The Poisson kernels will be annihilated (up to leading order) by $h^2\tilde\Delta_\phi$. The operator $\tbdry B_l$ will be roughly equal to the right hand side of \eqref{traceofpl} as to achieve the desired cancellation, while $\tbdry B_c \approx (1+\abs{K}^2)^{1/2}$ which can be thought of as the identity map (modulo multiplication by a uniformly positive function). 

Recall the definition of $\td{\Omega}\subset \mathbb R^n$ 
defined in \cref{changeofVars}, denote by $\itdo$ the indicator function of $\td{\Omega}$. For $v\in L^{r}(\td{\Omega})$ let $\itdo v \in L^r(\RR^n)$ be its trivial extension. Finally, recall the set $L^2_c(\RR^n) = \{v \in L^2(\RR^n) \colon \supp v\,\text{compact}\}$.

\begin{proposition}\label{propMappingOfNewPthing}
There exist bounded linear maps
\begin{equation}
	B_l \colon L^{2}_\delta(\RR^n) \to_{h^{-1}} H^1_{\delta-1}(\RR^n)\,,\ 0<\delta<1\,,\quad  B_l \colon L^{p'}(\RR^n) \to_{h^{-2}} L^{p}(\RR^n_+) \cap \left(H^1(\RR^n_+)\cap H^1_{\mathrm{loc}}(\RR^n)\right)\,, \label{mappropsofPdLpMpart}
\end{equation}
and
\begin{equation}
B_c \colon L^2(\RR^{n-1}) \to_{h^0} H^1(\RR^n)\,\label{mappropsofPdLpCpart}
	\end{equation}
such that for all $v\in L^{p'}(\td{\Omega})$ and all $f\in L^2(\RR^{n-1})$,
\begin{alignat}{2}\label{mapofPlfwithDelta1}
	&\itdo h^2\tilde{\Delta}_\phi B_l \itdo v &&= (R_{m}+hR_{m}')v\,, \\
	&\itdo h^2\tilde{\Delta}_\phi B_c f &&= hR_{c}f\,,\label{mapofPlfwithDelta2}
	\end{alignat}
	where 
	\begin{align}\label{Plfremainderprops}
		R_{m} &\colon L^2(\td{\Omega})\to_h L^2(\td{\Omega})\,, \quad R_{m} \colon L^{p'}(\td{\Omega}) \to_{h^0} L^2(\td{\Omega})\,,\notag \\
		R_{m}' &\colon L^r(\td{\Omega})\to_{h^0} L^r(\td{\Omega}),\quad 1<r<\infty\,, \\
		R_{c} &\colon L^2(\RR^{n-1}) \to_{h^0} L^2(\td{\Omega})\,.\notag
	\end{align} 
	Furthermore,
	\begin{align}\label{Plfboundaryprops}
		\tbdry B_l \itdo v &= (1+\abs{K}^2)^{-1/2}(1-\td{\rho}(hD'))\tau J\td{G}_\phi\itdo v + h(R_{bl}'+R_{bl}'')v
	\shortintertext{and} 
	\label{Plfboundarypropscpart}
	\tbdry B_cf &= (1+\abs{K}^2)^{-1/2}(1-\td{\rho}(hD'))f + hR_{bc}f
	\end{align}
	where $\td{\rho}$ is from \eqref{defofrho} and
	\begin{equation}\label{eq:bdryremainBlBc}
	\begin{aligned}
		R_{bl}' &\colon L^{2}(\td{\Omega}) \to_{h^0} L^2_{\delta-1}(\RR^{n-1})\,,\ 0<\delta<1\,&\an R_{bl}'&\colon L^{p'}(\td{\Omega}) \to_{h^0} L^2(\RR^{n-1})\,,\\ 
		R_{bl}'' &\colon L^{r}(\td{\Omega}) \to_{h^0} L^r(\RR^{n-1})\,,\quad 1<r<\infty\,,&\an R_{bc} &\colon L^2(\RR^{n-1}) \to_{h^0} L^2(\RR^{n-1})\,.
	\end{aligned}
	\end{equation}
\end{proposition}

The proof of \cref{propMappingOfNewPthing} will be delayed until \cref{sec:proofofPthing}. However, we are already in position to prove Proposition~\ref{Prop: bv of Pl}.
\begin{proof}[Proof of Proposition~\ref{Prop: bv of Pl}]
	Let $v\in W^{1,r}(\RR^n), 1<r<\infty$. We begin by proving that for $\td{\rho}$ from \cref{defofrho}, and $J^+$ from \cref{eq:defofJplus}, we have
	\begin{equation}\label{boundaryPropsofJplus}
		\tbdry (1-\td{\rho}(hD'))J^+v = -(1+\abs{K}^2)^{1/2}(1-\td{\rho}(hD'))\tau v\,.
	\end{equation}
To show  \cref{boundaryPropsofJplus}, we first observe that by \cref{defofboundary} 	
	\[
		\bdry = (1+\abs{K}^2)^{-1/2}(hK\cdot\nabla' - (1+\abs{K}^2)h\del_{x_n} - 1)\,.
	\]
	As we would like the operator $J$ to appear in \eqref{traceofpl}, we will use  \cref{splitupJintoF} to substitute $h\partial_{x_n} = J - F_+$ in the above expression:
	\[
		\bdry = (1+\abs{K}^2)^{1/2}J + (1+\abs{K}^2)^{-1/2}(hK\cdot\nabla'+F_+-1)\,.
	\]
 
	Now we may calculate that
	\begin{equation}\label{eq:bdrcalc}
	\begin{aligned}
		\tbdry (1-\td{\rho}(hD'))J^+v &= \tau (1+\abs{K}^2)^{1/2} J(1-\td{\rho}(hD'))J^+v \\
		&+ \tau (1+\abs{K}^2)^{-1/2}(hK\cdot \nabla'+F_+-1)(1-\td{\rho}(hD'))J^+v\,.
	\end{aligned}
\end{equation}
	
	Now recall that by \cref{Jinverse} and the definition of $J^+$ in \cref{eq:defofJplus}, 
	\begin{equation}\label{eq:traceproperty}
		\tau J^+v =0\,.
	\end{equation}
	Since $\tau$ and $\nabla'$ commute when applied to an object with trace $0$, with \cref{eq:traceproperty} we have 
	\[
		\tau (1+\abs{K}^2)^{-1/2}(hK\cdot \nabla'+F_+-1)(1-\td{\rho}(hD'))J^+v = 0\,,
	\]
	which takes care of the second term in \cref{eq:bdrcalc}.

	We observe that
	\begin{equation}
		J(1-\td{\rho}(hD')) = (1-\td{\rho}(hD'))J + [\td{\rho}(hD'),J] = (1-\td{\rho}(hD'))J + [\td{\rho}(hD'),F_+] \label{eq:commuteJrho}
	\end{equation}
	due to the fact that the first term of $J = h\del_{x_n}+F_+$ commutes with the operator $\td{\rho}(hD')$ which only depends on $x'$.

	Therefore, continuing the calculation from \cref{eq:bdrcalc},
	\begin{align*}
		\tbdry (1-\td{\rho}(hD'))J^+v &= (1+\abs{K}^2)^{1/2}(1-\td{\rho}(hD'))\tau JJ^+ v + (1+\abs{K}^2)^{1/2} [\td{\rho}(hD'),F_+] \tau J^+v \\
		&= (1+\abs{K}^2)^{1/2}(1-\td{\rho}(hD'))\tau v\,,
	\end{align*}
	after invoking \cref{eq:traceproperty} and the fact that $\tau J^+J v = \tau I_+v = \tau v$ since we are taking the trace from the $x_n >0$ direction.
	Thus, \cref{boundaryPropsofJplus} has been shown.		

	Finally, the result follows by directly computing $\tbdry P_l$ using the definition of $P_l$ in \cref{eq:definitionofPl} and applying \cref{boundaryPropsofJplus}.
\end{proof}

\subsubsection{Setup}

This section is devoted to preparatory lemmas required to construct the Poisson kernels $B_l$ and $B_c$. The first objective we will pursue is to find an operator with good boundary control. Thus we begin by inverting a part of the operator $\bdry$.

Let us first introduce $\td{\rho}_1 \in C_c^\infty(\RR^{n-1})$ with
\begin{equation}
	\td{\rho}_1 \equiv 1 \text{ on } \supp\td{\rho}_0\,, \quad \text{ and } \quad \supp \td{\rho}_1 \Subset \{\td{\rho} = 1\}\,, \label{defofrho4}
\end{equation}
where $\td{\rho}_0,\td{\rho}$ are from \cref{defofrho0,defofrho}.

Make use of \cref{defofboundary} to write
\[
	\bdry = (1+\abs{K}^2)^{-1/2}(hK\cdot\nabla' - (1+\abs{K}^2)h\del_{x_n} - 1)\,.
\]
As we have good control over the operator $J$, we will use \cref{splitupJintoF} to substitute $h\partial_{x_n} = J - F_+$ in the above expression:
\begin{align}
	\bdry &= (1+\abs{K}^2)^{-1/2}(hK\cdot\nabla'-(1+\abs{K}^2)(J-F_+)-1) \notag \\ 
	&= (1+\abs{K}^2)^{-1/2}({Z}-(1+\abs{K}^2)J-(1+\abs{K}^2)\td{\rho}_1(hD'))\,, \label{boundaryinJlemma}
\end{align}
where
\begin{equation}\label{eq:defofZ}
	{Z} \coloneqq  (1+\abs{K}^2)\td{\rho}_1(hD')+ hK \cdot\nabla'+ (1+\abs{K}^2)F_+ - 1 \in \op[h,x']{S^{1}_1(\RR^{n-1})}\,.
\end{equation}

Using the definition of $F_+$ in \eqref{splitupJintoF}, we can write $Z = \op[h,x']{z_0}$ where
\begin{align}
	{z}_0 &\coloneqq (1+\abs{K}^2)\td{\rho}_1 + 1_h-1 + 2iK\cdot\xi' + r_0 - (1+\abs{K}^2)ihm_0 \in S^1_1(\RR^{n-1}) \label{eq:defofzsymbol}
\end{align}
is its symbol, $r_0$ is given in \cref{eq:defofr}, and $m_0$ in \cref{herem0isdefined}.

We are now in position to invert the operator $Z$.
\begin{lemma}\label{lemm:Zinverse}
	For any $1<r<\infty, \delta\in\RR$ and any $s,k\in\mathbb{N}_0$ there is some $h_0>0$ so that if $0<h<h_0$, there are operators $M$ and $Z^{-1}$ which satisfy
	\begin{equation}\label{eq:definitionofZinverse}
		M, Z^{-1}\colon W^{s,r}(\RR^{n-1}) \to_{h^0} W^{s+1,r}(\RR^{n-1})\,,\an M, {Z}^{-1} \colon H^s_\delta(\RR^{n-1}) \to_{h^0} H^{s+1}_\delta(\RR^{n-1})\,.
	\end{equation}
	Furthermore, 	
	\begin{equation}\label{eq:Zneumann}
		Z^{-1} = \op[h,x']{S^{-1}_1(\RR^{n-1})} + h^kM\,,
	\end{equation}
	and $Z^{-1}$ is the right-inverse of the operator ${Z}$ from \cref{eq:defofZ}. 
\end{lemma}
\begin{proof}
	If, for simplicity, we put $h=0$ in the expression for ${z}_0$ from \cref{eq:defofzsymbol}, we see that
	\[
		\mathrm{Re}({z}_0) = (1+\abs{K}^2)\td{\rho}_1 + \mathrm{Re}r_0\,,
	\]
	and let us write
	\begin{equation}\label{eq:defofr1r2}
		r_0 = (1-\td{\rho}_0)\sqrt{r_1 + ir_2}\,,\quad\text{where}\quad r_1,r_2\quad\text{are real}\,. 
	\end{equation}
	As $r_2 = 2K\cdot\xi'$, the branch cut of the square root is only met when $K\cdot\xi' = 0$ and $r_1 \leq 0$. Now when $K\cdot\xi'=0$, we have $r_1 = -\abs{K}^2 + \abs{\xi'}^2(1+\abs{K}^2)$, which is uniformly positively away from zero on the support of $1-\td{\rho}_0$, see \cref{defofrho0}, which also means that $r_0$ and so $z_0$ are smooth. 
	
	In particular, we have that on the support of $1-\td{\rho}_0$ we have $\mathrm{Re}r_0 > 0$, and due to the definition of $\td{\rho}_1$, on the complement of the support of $1-\td{\rho}_0$, we have $\td{\rho}_1 = 1$. This means that $\mathrm{Re}{z}_0$, and especially ${z}_0$, are uniformly bounded away from $0$ for small $h$.

	We now perform some rough estimates to show that $\frac{{z}_0}{\langle \xi'\rangle}$ is bounded uniformly away from $0$ for large $\xi'$ and small $h>0$.
	
	Take again $h=0$ for the moment, and note that by Cauchy-Schwartz, for $r_1,r_2$ from \cref{eq:defofr1r2},
	\[
		r_1 = 1-(K\cdot\xi')^2-(1-\abs{\xi'}^2)(1+\abs{K}^2) \geq \abs{\xi'}^2- \abs{K}^2\,,
	\]
	where we can choose $\abs{\xi'}$ large enough and $C>0$ so that 
	\begin{equation}\label{eq:somecauchyschwatry}
		\abs{\xi'}^2- \abs{K}^2 > 2C(1+\abs{\xi'}^2)\,.
	\end{equation}
	Since we are taking the principal square root in $r_0$, we have that 
	\[
		\mathrm{Re}\sqrt{r_1+ir_2} = \sqrt{\frac{\sqrt{r_1^2+r_2^2}+r_1}{2}}\,,
	\]
	and so for large $\abs{\xi'}^2$ and $h=0$, \cref{eq:somecauchyschwatry} gives us
	\begin{equation}\label{eq:r0estimateeps}
		(\mathrm{Re} r_0)^2 \geq r_1 > C(1+\abs{\xi'}^2)\,.
	\end{equation}
	Taking square roots on both sides of \cref{eq:r0estimateeps} and using continuity in $h$, we have shown that for small $h>0$ and large $\abs{\xi'}$, the expression $\frac{(1+\abs{K}^2)\td{\rho}_1+\mathrm{Re}r_0}{\langle\xi'\rangle}$ is bounded uniformly away from $0$.
	
	In particular, this means that the symbol of ${Z}\langle hD'\rangle^{-1}$ is of order $0$ and uniformly bounded away from $0$ and has an inverse we denote by ${w} \in S^0_1(\RR^{n-1})$.
	
	Writing this out means that 
	\begin{equation}\label{Z parametrix}
		{Z}\langle hD'\rangle^{-1}\op[h,x']{{w}} = 1+ hR\,,
	\end{equation}
	for some $R\in \op[h,x']{S^0_1(\RR^{n-1})}$, and thus the object $\op[h,x']{{w}}(1+hR)^{-1}$ is an inverse of ${Z}\langle hD'\rangle^{-1}$. 
	
	Thus, defining  
\begin{equation}\label{eq:defofZinv}
{Z}^{-1}\coloneqq \langle hD'\rangle^{-1}\op[h,x']{w}(1+hR)^{-1}
\end{equation}
we have by \cref{Z parametrix}
\[
	{Z}{Z}^{-1} = {Z}\langle hD'\rangle^{-1}\langle hD'\rangle {Z}^{-1} = \id\,.
\]
Finally, for any fixed $k\in\mathbb{N}$, we can write out part of the Neumann series of $(1+hR)^{-1}$ to get \cref{eq:Zneumann}, which concludes the proof.
\end{proof}
\begin{corollary}\label{cor:sandwich}
	Let $\chi_1, \chi_2 \in C_c^\infty(\RR^{n-1})$ be cut-offs so that $\supp \chi_1 \cap \supp (1-\chi_2) = \emptyset$. For all $s,k \in \mathbb{N}_0, 1<r<\infty$, and $\delta\in\RR$,
	\begin{equation}\label{eq:sandwichZ}
	\begin{aligned}
		\chi_1(hD') Z^{-1} (1-\chi_2(hD')) &\colon W^{s,r}(\RR^{n-1}) \to_{h^k} W^{s+k,r}(\RR^{n-1})\,, \\
		\chi_1(hD') Z^{-1} (1-\chi_2(hD')) &\colon L^2_\delta(\RR^{n-1}) \to_{h^k} H^k_\delta(\RR^{n-1})\,.
	\end{aligned}
	\end{equation}
\end{corollary}
\begin{proof}
	Let $k \in \mathbb{N}_0$ be fixed. By \cref{eq:Zneumann}, write	
	\[
		Z^{-1} = \op[h,x']{S^{-1}_1(\RR^{n-1})} + h^k M\,, 
	\] 
	where
	\[
		M\colon W^{s,r}(\RR^{n-1}) \to_{h^0} W^{s+1,r}(\RR^{n-1})\,,\an M\colon L^2_\delta(\RR^{n-1}) \to_{h^0} H^1_\delta(\RR^{n-1})\,.
	\]
	By disjoint supports, $\chi_1(hD')\op[h,x']{S^{-1}_0(\RR^{n-1})}(1-\chi_2(hD'))$ satisfies the estimates in \cref{eq:sandwichZ} and because $\chi_1 \in S^{-\infty}_1(\RR^{n-1})$, the same is true of $h^k \chi_1(hD') M (1-\chi_2(hD'))$.
\end{proof}

From now on we will also interpret operators acting only on $n-1$ dimensions as those that can act on all of $x$ by acting only on the first $n-1$ components: if $A$ is an operator acting on $n-1$ dimensions and $u\colon \RR^n\to\RR$, put $Au \colon \RR^n \to \RR$ by considering the last component $x_n$ of $ (x_1,\dots,x_n) \in \RR^n$ to be a parameter instead of a variable. 

As a common building block for the construction of $B_l$ and $B_c$, we will now introduce the operator $(\id - J^+J)Z^{-1}$.
\begin{lemma}\label{defnewJ}
	Let $v \in W^{1,r}(\RR^n)\cup C_c^1(\RR; L^r(\RR^{n-1})), 1<r<\infty$. The operator
	\begin{align}
		v&\mapsto  (\id-J^+J){Z}^{-1}v \label{defofPb}\\
 	\shortintertext{satisfies}
		\tbdry (\id-J^+J)Z^{-1}v &= (1+\abs{K}^2)^{-1/2}(1-(1+\abs{K}^2)\td{\rho}_1(hD'){Z}^{-1})\tau v\,.\label{boundaryofPb}
	\end{align}
	Additionally, with $W^{1,1,r}$ defined as in \cref{eq:defofmixedsemiclasssobo},
	\begin{equation}\label{mappingpropsofPb}
		(\id-J^+J)Z^{-1} \colon W^{1,r}(\RR^n) \to_{h^0} W^{2,r}(\RR^n_+) \cap W^{1,1,r}\,,\quad (\id-J^+J)Z^{-1} \colon C_c^1(\RR;L^r(\RR^{n-1})) \to_{h^0} W^{1,r}(\RR^n)\,. 
	\end{equation}
\end{lemma}
\begin{proof}
	Throughout this proof fix 
	\[
		v \in W^{1,r}(\RR^n) \cup C_c^1(\RR; L^r(\RR^{n-1}))\,,\quad\text{where}\quad 1<r<\infty\,.
	\]
	Note first that for this choice of $v$, by \cref{lemm:Zinverse}, 
	\begin{equation}\label{eq:Zinverseonv}
		{Z}^{-1}v \in W^{1,r}(\RR^n)\,,
	\end{equation}
	which means that $JZ^{-1}v \in L^r(\RR^n)$, meaning that $(\id-J^+J)Z^{-1}$ acting on $v\in W^{1,r}(\RR^n) \cup C_c^1(\RR; L^r(\RR^{n-1}))$ is well-defined. Due to the fact that $\id-J^+J \colon W^{1,r}(\RR^n) \to W^{1,r}(\RR^n)$, we already see the second half of \cref{mappingpropsofPb}. 
	
	To see the first part of \cref{mappingpropsofPb}, let us take $v\in W^{1,r}(\RR^n)$. Note that by \cref{lemm:Zinverse}, and the definition of mixed semiclassical Sobolev spaces in \cref{eq:defofmixedsemiclasssobo}, $Z^{-1}\colon W^{1,r}(\RR^n) \to W^{1,1,r}$ and $\id-J^+J \colon W^{1,1,r}\to W^{1,1,r}$ where we use the fact that $I_+$ does not alter regularity in the $x'$ direction. Thus, 
	\[
		(\id-J^+J)Z^{-1} =(\id - J^+J) Z^{-1} \colon W^{1,r}(\RR^n) \to W^{1,1,r}.
	\]
	Furthermore, since by \cref{splitupJintoF}, $h\del_{x_n} = J-F_+$, and by direct calculation $J(\id-J^+J) = (\id-I_+)J$, we have that 
	\begin{equation}\label{eq:deaddistro1}
		h\del_{x_n}(\id-J^+J){Z}^{-1}v = (\id-I_+)J{Z}^{-1}v+F_+(\id-J^+J){Z}^{-1}v\,.
	\end{equation}
	Here we note that the first term in \cref{eq:deaddistro1} vanishes in the sense of distributions on $\RR^n_+$ and that for the second term we have 
	\begin{equation}\label{eq:deaddistro2}
		F_+(\id-J^+J){Z}^{-1}: W^{1,r}(\RR^n) \xrightarrow{{Z}^{-1}} W^{1,1,r} \xrightarrow{\id-J^+J} W^{1,1,r} \xrightarrow{F_+} W^{1,r}(\RR^n)\,.
	\end{equation}
	In particular, \cref{eq:deaddistro1,eq:deaddistro2} imply that
	\begin{equation}\label{halfspaceW2r}
	h\del_{x_n} (\id-J^+J)Z^{-1} v = F_+ (\id-J^+J)Z^{-1}v\quad\quad {\text {in\ the\ sense\ of}}\ \mathcal D'(\RR^n_+)\,,
\end{equation}
	where $F_+(\id-J^+J)Z^{-1} v \in W^{1,r}(\RR^n)$. Furthermore, since
	\[
	(\id-J^+J)Z^{-1}= (\id-J^+J)	 Z^{-1}: W^{1,r}(\RR^n) \xrightarrow{{Z}^{-1}} W^{1,1,r} \xrightarrow{\id-J^+J} W^{1,1,r}\,,
	\]
	we have that $h^{\abs{\alpha}}\del^\alpha (\id-J^+J)Z^{-1} v \in L^r(\RR^n)$ for all $\alpha\in\mathbb{N}^n_0,\abs{\alpha} \leq 2$, $\alpha_n \leq 1$. Therefore, in order to show $(\id-J^+J)Z^{-1} v \in W^{2,r}(\RR^n_+)$, it suffices to show that $\del^2_{x_n}(\id-J^+J)Z^{-1} v \in L^r(\RR^n_+)$. This fact is a consequence of taking the $\del_{x_n}$ derivative of \eqref{halfspaceW2r} to see that 
	\[
		h^2\del^2_{x_n}(\id-J^+J)Z^{-1} v = h\del_{x_n}F_+(\id-J^+J)Z^{-1}v \in L^r(\RR^n_+),\ \ {\text {in\ the\ sense\ of\ }} \mathcal D'(\mathbb R^n_+)\,.
\]
This completes \cref{mappingpropsofPb}.

	To see \cref{boundaryofPb}, by 
	\cref{boundaryinJlemma}, we calculate
	\begin{align}
		&(1+\abs{K}^2)^{1/2}\tbdry (\id-J^+J)Z^{-1} v \notag \\
		&=\tau ({Z}-(1+\abs{K}^2)J-(1+\abs{K}^2)\td{\rho}_1(hD'))(\id-J^+J)Z^{-1}v \notag \\
		&= \tau ({Z}{Z}^{-1}v - (1+\abs{K}^2)\td{\rho}_1(hD')(\id-J^+J)Z^{-1}v-{Z}J^+J{Z}^{-1}v + (1+\abs{K}^2)(I_+-id)J{Z}^{-1}v) \label{eq:wantvanish3rd4th}
	\end{align}
	All terms in the final line of \cref{eq:wantvanish3rd4th} are well-defined by \cref{eq:Zinverseonv}, and the third term vanishes because $\tau$ and $Z^{-1}$ commute and $\tau J^+ v =0$ (see also \cref{eq:traceproperty}). The fourth term in \cref{eq:wantvanish3rd4th} is zero because it is only supported on $x_n \leq 0$.
	
	Thus, again using $\tau J^+ v= 0$,
	\begin{align*}
		(1+\abs{K}^2)^{1/2}\tbdry (\id-J^+J)Z^{-1} v &= \tau v - \tau (1+\abs{K}^2)\td{\rho}_1(hD')(\id-J^+J)Z^{-1} v \\
		&= (1-(1+\abs{K}^2)\td{\rho}_1(hD'){Z}^{-1})\tau v\,,
	\end{align*}
	which shows \cref{boundaryofPb}.
\end{proof}

\subsubsection{Proof of Proposition~\ref{propMappingOfNewPthing}}\label{sec:proofofPthing}

Before proceeding with the proof of Proposition \ref{propMappingOfNewPthing}, we have to address the concept of extending boundary conditions to all of $\RR^n$. Take $\eta \in C_c^\infty(\RR)$ so that $\supp \eta \subset \{\abs{x_n} \leq 2\}$ and $\eta \equiv 1$ on $\abs{x_n} \leq 1$, and define 
\begin{equation}
	\mathcal{E} \colon W^{s,r}(\RR^{n-1}) \to C^\infty_c(\RR; W^{s,r}(\RR^{n-1}))\,, \quad \mathcal{E} t(x) = \eta(x_n)t(x')\,, \label{defofExt}
\end{equation}
for all $s\in\mathbb{N}_0, 1<r<\infty$.

\begin{proof}[Proof of \cref{propMappingOfNewPthing}]	
We begin by setting
\begin{align}\label{defofPlf}
	B_l &\coloneqq (\id-J^+J)Z^{-1}(1-\td{\rho}(hD'))J\tilde{G}_\phi 
		\shortintertext{and}
B_c &\coloneqq (\id-J^+J)Z^{-1}(1-\td{\rho}(hD'))\mathcal{E}\label{def Pc}
\end{align}
where $(\id-J^+J)Z^{-1}$ is the operator in \eqref{defofPb}.

Let $v \in L^{p'}(\td{\Omega})$ and $f\in L^2(\RR^{n-1})$. To check that $B_cf$ is well-defined, note that by \cref{defofExt,mappingpropsofPb},
\begin{equation}\label{eq:extfworks}
(\id-J^+J)Z^{-1}(1-\td{\rho}(hD'))\mathcal E \colon L^2(\RR^{n-1}) \xrightarrow{(1-\td{\rho}(hD'))\mathcal{E}} C^\infty_c(\RR;L^2(\RR^{n-1})) \xrightarrow{(\id-J^+J)Z^{-1}} H^1(\RR^n)\,,
\end{equation}
which, in fact, proves \eqref{mappropsofPdLpCpart}.

In order to see that $B_l \itdo v$ is well-defined we turn to showing \cref{mappropsofPdLpMpart}. Observe that by \cref{prop43}, splitting up $\td{G}_\phi = (\td{G}_\phi - \td{G}_\phi^c) + \td{G}_\phi^c$ and using \cref{mappingpropsofPb},
	\begin{align*}
		&L^{p'}(\RR^n) \xrightarrow[h^{-2}]{\td{G}_\phi^c} W^{2,p}(\RR^n) \xrightarrow{(1-\td{\rho}(hD'))J} W^{1,p}(\RR^n) \xrightarrow{(\id-J^+J)Z^{-1}} W^{1,p}(\RR^n)\,,\\
		&L^{p'}(\RR^n) \xrightarrow{\td{G}_\phi - \td{G}_\phi^c} W^{2,p'}(\RR^n) \xrightarrow{(1-\td{\rho}(hD'))J} W^{1,p'}(\RR^n) \xrightarrow{(\id-J^+J)Z^{-1}} W^{2,p'}(\RR^n_+) \hookrightarrow_{h^{-1}} H^1(\RR^n_+) \hookrightarrow_{h^{-1}} L^p(\RR^n_+)\,.
	\end{align*}
	Omitting the last Sobolev embedding and remarking that $W^{1,p} \subset H^1_{\mathrm{loc}}$, we also have that $B_l \colon L^{p'}(\RR^n) \to H^1(\RR^n_+)\cap H^1_{\mathrm{loc}}(\RR^n)$ so that the second half of \cref{mappropsofPdLpMpart} has been shown.
	
	Furthermore,
	\begin{align*}
		&L^{2}_\delta(\RR^n) \xrightarrow[h^{-1}]{\td{G}_\phi^c} H^2_{\delta-1}(\RR^n) \xrightarrow{(1-\td{\rho}(hD'))J} H^{1}_{\delta-1}(\RR^n) \xrightarrow{(\id-J^+J)Z^{-1}} H^1_{\delta-1}(\RR^n)\,, \\
		&L^{2}_\delta(\RR^n) \xrightarrow{\td{G}_\phi - \td{G}_\phi^c} H^2_\delta(\RR^n) \xrightarrow{(1-\td{\rho}(hD'))J} H^1_\delta(\RR^n) \xrightarrow{(\id-J^+J)Z^{-1}} H^1_\delta(\RR^n)\,,
	\end{align*}
	which means that
	$B_l\colon L^2_\delta(\RR^n) \to_{h^{-1}} H^1_{\delta-1}(\RR^n)$,
	proving the first half of \cref{mappropsofPdLpMpart}.

	We now begin showing \cref{mapofPlfwithDelta1,mapofPlfwithDelta2}. Recall again that by direct calculation $J(\id-J^+J) = (\id-I_+)J$ and so due to \cref{splitupofLaplace}\,, 
	\begin{equation}\label{eq:initialsplitupodDelta}
	h^2\tilde{\Delta}_\phi(\id-J^+J) = (1+\abs{K}^2)(\tilde{A}_0 - h\tilde{E}_1+h^2\tilde{E}_0)(\id-J^+J) + (1+\abs{K}^2)(Q(\id-I_+)J)\,,
	\end{equation}
	where $\td{A}_0,\td{E}_1,\td{E}_0$ are defined in \cref{eq:defofA0E1E0}, and $Q$ is defined in \cref{eq:defofJQ}.
	Due to the fact that $Q^\ast$ does not shift support downward because with respect to $x_n$ it is a differential operator, for all $w\in C_c^\infty(\RR^n), u\in C_c^\infty(\RR^n_+)$,
	\begin{equation}\label{Qpartfallsaway}
		\langle  (1+\abs{K}^2)Q(\id-I_+)Jw, u\rangle = \langle (\id-I_+)Jw, ((1+\abs{K}^2)Q)^\ast u\rangle = 0\,,
	\end{equation}
	because the term $((1+\abs{K}^2)Q)^\ast$ also does not shift the support of $u$ downward. We now let 
	\begin{equation}\label{eq:defwlwc}
	w_l = {Z}^{-1}(1-\td{\rho}(hD'))J\tilde{G}_\phi\itdo v\in L^p(\RR^n_+)\,,\an w_c = {Z}^{-1}(1-\td{\rho}(hD'))\mathcal{E}f\in H^1(\RR^n)
\end{equation}
so that recalling the definitions of $B_l,B_c$ in \cref{defofPlf,def Pc},
	\[
	 h^2\tilde{\Delta}_\phi B_l \itdo v = h^2\tilde{\Delta}_\phi(\id-J^+J) w_l\,,\an h^2\tilde{\Delta}_\phi B_c f= h^2\tilde{\Delta}_\phi(\id-J^+J) w_c\,.
	\]
	Replacing $h^2\tilde {\Delta}_\phi$ in the above equations by its factorization \eqref{splitupofLaplace}, we get
 	\begin{align}\label{A0E1E0split1'}
			h^2\tilde{\Delta}_\phi B_l \itdo v &= (1+\abs{K}^2)Q(id-I_+)Jw_l +(1+\abs{K}^2)(\tilde{A}_0 - h\tilde{E}_1+h^2\tilde{E}_0)(\id-J^+J)w_l\,,\\
			 h^2\tilde{\Delta}_\phi B_c f&= (1+\abs{K}^2)Q(id-I_+)Jw_c + (1+\abs{K}^2)(\tilde{A}_0 - h\tilde{E}_1+h^2\tilde{E}_0)(\id-J^+J)w_c\,.\label{A0E1E0split2'}
\end{align}
Approximating $w_l$ and $w_c$ by functions in $C^\infty_c(\RR^n)$ in their respective Sobolev space topology, a density argument together with \cref{Qpartfallsaway} shows that the first terms of \eqref{A0E1E0split1'} and \eqref{A0E1E0split2'} vanish as elements of $\mathcal D'(\RR^n_+)$. So using \cref{eq:defwlwc}, \eqref{A0E1E0split1'} and \eqref{A0E1E0split2'} become
	\begin{align}\label{A0E1E0split1}
	h^2\tilde{\Delta}_\phi B_l \itdo v &=  (1+\abs{K}^2)(\tilde{A}_0 - h\tilde{E}_1+h^2\tilde{E}_0)(\id-J^+J)Z^{-1}(1-\td{\rho}(hD'))J\td{G}_\phi\itdo v\,, \\
	h^2\tilde{\Delta}_\phi B_c f&=  (1+\abs{K}^2)(\tilde{A}_0 - h\tilde{E}_1+h^2\tilde{E}_0)(\id-J^+J)Z^{-1}(1-\td{\rho}(hD'))\mathcal{E}f\,.\label{A0E1E0split2}
	\end{align}
The operators $\tilde A_0$, $\tilde E_1$ and $\tilde E_0$ are given collectively by \eqref{eq:defofA0E1E0}.
	In order to prove \cref{mapofPlfwithDelta1} and \cref{mapofPlfwithDelta2}, we employ the following lemma to give estimates for the right side of \cref{A0E1E0split1,A0E1E0split2}. We will prove the lemma at the end of the section.
	\begin{lemma}\label{largeFreqremainders}
		We may estimate the right-hand sides in \cref{A0E1E0split1} and \cref{A0E1E0split2} as follows. 
		\begin{enumerate}
			\item \label{A0largeFreq} 
			We have
			\begin{alignat}{2}\label{eq:biglemmapart1eqv}
			&(1+\abs{K}^2)\td{A}_0(\id-J^+J)Z^{-1}(1-\td{\rho}(hD'))J\td{G}_\phi\itdo v &&= (R_1+hR_2)v\,, \\
			&(1+\abs{K}^2)\td{A}_0(\id-J^+J)Z^{-1}(1-\td{\rho}(hD'))\mathcal{E}f &&= hR_{c,1}f\,.\label{eq:biglemmapart1eqf}
			\end{alignat}
			\item \label{E1largeFreq} 
			Additionally,
			\begin{alignat}{2}\label{eq:biglemmapart2eqv}
				&(1+\abs{K}^2)h\td{E}_1(\id-J^+J)Z^{-1}(1-\td{\rho}(hD'))J\td{G}_\phi\itdo v &&= (R_3+hR_4)v\,, \\
				&(1+\abs{K}^2)h\td{E}_1(\id-J^+J)Z^{-1}(1-\td{\rho}(hD'))\mathcal{E}f &&= hR_{c,2}f\,.\label{eq:biglemmapart2eqf}
			\end{alignat}
			\item \label{E0largeFreq} 
			Finally,
			\begin{alignat}{2}\label{eq:biglemmapart3eqv}
				&(1+\abs{K}^2)h^2\td{E}_0(\id-J^+J)Z^{-1}(1-\td{\rho}(hD'))J\td{G}_\phi\itdo v &&= (R_5+hR_6)v\,, \\
				&(1+\abs{K}^2)h^2\td{E}_0(\id-J^+J)Z^{-1}(1-\td{\rho}(hD'))\mathcal{E}f &&= hR_{c,3}f\,.\label{eq:biglemmapart3eqf}
			\end{alignat}
		\end{enumerate}
		Here 
		\begin{alignat*}{2}
			R_1, R_5 &\colon L^2(\td{\Omega})\to_h L^2(\RR^n)\,,  &\qquad R_1, R_5&\colon L^{p'}(\td{\Omega}) \to_{h^0} L^p(\RR^n) \\
			R_3 &\colon L^2(\td{\Omega})\to_h L^2(\RR^n)\,, &\qquad R_3&\colon L^{p'}(\td{\Omega}) \to_{h^0} L^2(\RR^n) \\
			R_2, R_4, R_6 &\colon L^r(\td{\Omega}) \to_{h^0} L^r(\RR^n)\,,&\quad &1<r<\infty \\
			R_{c,1},R_{c,2},R_{c,3} &\colon  L^2(\RR^{n-1}) \to_{h^0} L^2(\RR^n) &{} &{}
		\end{alignat*}
	are the mapping properties of the remainders.
	\end{lemma}
Since $\td{\Omega}$ is bounded and $p>2$, we have that $L^p(\td{\Omega}) \hookrightarrow L^2(\td{\Omega})$. So applying the estimates \eqref{eq:biglemmapart1eqv}, \eqref{eq:biglemmapart2eqv}, and \eqref{eq:biglemmapart3eqv} to the three terms involving $\tilde A_0$, $\tilde E_0$, and $\tilde E_1$ in equation \eqref{A0E1E0split1} we obtain \eqref{mapofPlfwithDelta1} with the estimates \cref{Plfremainderprops}. Similarly, by applying  \eqref{eq:biglemmapart1eqf}, \eqref{eq:biglemmapart2eqf}, and \eqref{eq:biglemmapart3eqf} to the terms on the right side of \eqref{A0E1E0split2}, we get \eqref{mapofPlfwithDelta2} with the estimates \cref{Plfremainderprops}.

We move on to show \cref{Plfboundaryprops}, \cref{Plfboundarypropscpart} and \cref{eq:bdryremainBlBc}. Note that \cref{boundaryofPb} gives
	\begin{align*}
		\tbdry B_l \itdo v &= \tbdry (\id-J^+J)Z^{-1} (1-\td{\rho}(hD'))J\td{G}_\phi\itdo v \\
		&= (1+\abs{K}^2)^{-1/2}(1-(1+\abs{K}^2)\td{\rho}_1(hD'){Z}^{-1})\tau (1-\td{\rho}(hD'))J\td{G}_\phi\itdo v \\
		&= (1+\abs{K}^2)^{-1/2}(1-\td{\rho}(hD'))\tau J\td{G}_\phi\itdo v + hR_{bl} v\,,
	\end{align*}
	where $R_{bl} = -h^{-1}(1+\abs{K}^2)^{1/2}\td{\rho}_1(hD'){Z}^{-1}(1-\td{\rho}(hD'))\tau J\td{G}_\phi\itdo$. 
	Observe that by \cref{prop43}, splitting up $\td{G}_\phi = (\td{G}_\phi - \td{G}_\phi^c) + \td{G}_\phi^c$, and invoking \cref{cor:sandwich}, we have for any $0<\delta<1$,
	\begin{align*}
		&L^{p'}(\td{\Omega}) \xrightarrow{\itdo} L^{p'}(\RR^n) \xrightarrow[h^{-2}]{\td{G}_\phi^c} W^{2,p}(\RR^n) \xrightarrow{\tau J} L^p(\RR^{n-1}) \xrightarrow[h^2]{h^{-1}(1+\abs{K}^2)^{1/2}\td{\rho}_1(hD'){Z}^{-1}(1-\td{\rho}(hD'))} L^p(\RR^{n-1})\,,\\
		&L^2(\td{\Omega}) \xrightarrow{\itdo} L^{2}_{\delta}(\RR^n) \xrightarrow[h^{-1}]{\td{G}_\phi^c} H^2_{\delta-1}(\RR^n) \xrightarrow{\tau J} L^2_{\delta-1}(\RR^{n-1}) \xrightarrow[h]{h^{-1}(1+\abs{K}^2)^{1/2}\td{\rho}_1(hD'){Z}^{-1}(1-\td{\rho}(hD'))} L^2_{\delta-1}(\RR^{n-1})\,, 
	\end{align*}
	which we now take to be $R_{bl}'$.
	Furthermore, for $1<r<\infty$,
	\begin{align*}
		&L^r(\td{\Omega}) \xrightarrow{\itdo} L^{r}(\RR^n) \xrightarrow{\td{G}_\phi - \td{G}_\phi^c} W^{2,r}(\RR^n) \xrightarrow{\tau J} L^{r}(\RR^{n-1}) \xrightarrow{h^{-1}(1+\abs{K}^2)^{1/2}\td{\rho}_1(hD'){Z}^{-1}(1-\td{\rho}(hD'))} L^{r}(\RR^{n-1})\,,
	\end{align*}
	which we call $R_{bl}''$. 
	Now $R_{bl} = R_{bl}'+R_{bl}''$ where $R_{bl}', R_{bl}''$ satisfy \cref{eq:bdryremainBlBc} proving \cref{Plfboundaryprops}.

	Using \cref{boundaryofPb} again, we get
	\begin{align*}
		\tbdry B_cf &= (1+\abs{K}^2)^{-1/2}(1-\td{\rho}_1(hD'){Z}^{-1})\tau (1-\td{\rho}(hD'))\mathcal{E}f \\
		&= (1+\abs{K}^2)^{-1/2}(1-\td{\rho}(hD'))f + hR_{bc}f\,.
	\end{align*}
with $R_{bc} = -h^{-1}\td{\rho}_1(hD'){Z}^{-1}(1-\td{\rho}(hD'))$. A final application of \cref{cor:sandwich} on $R_{bc}$ proves \cref{Plfboundarypropscpart} and \cref{eq:bdryremainBlBc}.
\end{proof}

\begin{proof}[Proof of \cref{A0largeFreq} in \cref{largeFreqremainders}]
	Let $1<r<\infty$ be arbitrary. By \cref{splitupJintoF} we have $J = h\del_{x_n} + F_+$ so that
	\begin{align}\label{splitupA0thing}
		\td{A}_0(\id-J^+J){Z}^{-1}(1-\td{\rho}(hD')) &=\td{A}_0{Z}^{-1}(1-\td{\rho}(hD')) - \td{A}_0 J^{-1}{Z}^{-1}(1-\td{\rho}(hD'))I_+h\del_{x_n} \\
		&-\td{A}_0J^{-1}F_+{Z}^{-1}(1-\td{\rho}(hD'))I_+ \notag
	\end{align}
	where we used the fact that operators in the $x'$ direction commute with $I_+$ and $\del_{x_n}$.
	
	Now, taking a look at the definition of $\td{A}_0$ in \cref{eq:defofA0E1E0}, the symbol of $\td{A}_0$ is supported on $\supp\td{\rho}_0$, which is disjoint from the set $\supp(1-\td{\rho})$. Thus we will want to use disjoint support properties of pseudo-differential operators to get estimates in suitable powers of $h$ and regularity. 

	Repeating the proof of \cref{cor:sandwich} with $\td{A}_0$ in place of $\chi_1(hD')$ and $\td{\rho}$ in place of $\chi_2$, using disjoint supports, we have that
	{\small\begin{equation}\label{eq:1}
		\td{A}_0{Z}^{-1}(1-\td{\rho}(hD')) \colon W^{1,r}(\RR^{n-1}) \to_{h^2} W^{1,r}(\RR^{n-1})\,,\an \td{A}_0{Z}^{-1}(1-\td{\rho}(hD'))\colon H^1_\delta(\RR^{n-1}) \to_{h^2} H^1_\delta(\RR^{n-1})\,.
\end{equation}}

For the two terms in \cref{splitupA0thing} containing $J^{-1}$, we want to make use of \cref{lem:Jsandwich}. To do so, we use \cref{eq:Zneumann} to write 
	\[
		Z^{-1} = (Z^{-1})'+ h^k M\,, 
	\] 
	where $(Z^{-1})' \in \op[h,x']{S^{-1}_1(\RR^{n-1})}$ and
	\[
		M\colon L^r(\RR^{n-1}) \to_{h^0} W^{1,r}(\RR^{n-1})\,,\an M\colon L^2_\delta(\RR^{n-1}) \to_{h^0} H^{1}_\delta(\RR^{n-1})\,.
	\]
	
	Now we invoke \cref{lem:Jsandwich} to see that
	\begin{alignat}{3}
		&&\td{A}_0 J^{-1}({Z}^{-1})'(1-\td{\rho}(hD')) &\colon L^r(\RR^n)\to_{h^2} W^{1,r}(\RR^n)\,,&&\an L^2_\delta(\RR^n) \to_{h^2} H^1_\delta(\RR^n)\,, 
		\label{eq:21}\\
		&&\td{A}_0J^{-1}F_+({Z}^{-1})'(1-\td{\rho}(hD')) &\colon L^{r}(\RR^n)\to_{h^2} W^{1,r}(\RR^n)\,,&&\an L^2_\delta(\RR^n)\to_{h^2} H^1_\delta(\RR^n)\,. 
		\label{eq:31}
	\end{alignat}
	Furthermore, by the usual $\Psi$DO mapping properties and \cref{Jinverse},
	\begin{alignat}{3}
		&&h^2\td{A}_0 J^{-1}M(1-\td{\rho}(hD')) &\colon L^r(\RR^n)\to_{h^2} W^{1,r}(\RR^n)\,,&&\an L^2_\delta(\RR^n) \to_{h^2} H^1_\delta(\RR^n)\,, 
		\label{eq:22}\\
		&&h^2\td{A}_0J^{-1}F_+M(1-\td{\rho}(hD')) &\colon L^{r}(\RR^n)\to_{h^2} W^{1,r}(\RR^n)\,,&&\an L^2_\delta(\RR^n)\to_{h^2} H^1_\delta(\RR^n)\,. 
		\label{eq:32}
	\end{alignat}

	Note that the operators above appear in \cref{splitupA0thing} with additional operators precomposed to them. We address this now. By \cref{eq:21,eq:22},
	\begin{equation}\label{eq:4}
	\begin{aligned}
	&W^{1,r}(\RR^n) \xrightarrow{I_+h\del_{x_n}} L^r(\RR^n) \xrightarrow[h^2]{\td{A}_0 J^{-1}{Z}^{-1}(1-\td{\rho}(hD'))} W^{1,r}(\RR^n)\,,\\
	&H^1_\delta(\RR^n) \xrightarrow{I_+h\del_{x_n}} L^2_\delta(\RR^n) \xrightarrow[h^2]{\td{A}_0 J^{-1}{Z}^{-1}(1-\td{\rho}(hD'))} H^1_\delta(\RR^n)\,,
	\end{aligned}
\end{equation}
	and 
	\begin{equation}\label{eq:4second}
		C^\infty_c(\RR;L^2(\RR^{n-1})) \xrightarrow{I_+h\del_{x_n}} C^\infty_c(\RR;L^2(\RR^{n-1}))\subset L^2(\RR^n) \xrightarrow[h^2]{\td{A}_0 J^{-1}{Z}^{-1}(1-\td{\rho}(hD'))} H^1(\RR^n)\,.
	\end{equation}
	Now by \cref{eq:31,eq:32},
	\begin{equation}\label{eq:5}
	\begin{aligned}
	&W^{1,r}(\RR^n) \xrightarrow{I_+} W^{1,0,r} \xrightarrow[h^2]{\td{A}_0J^{-1}F_+{Z}^{-1}(1-\td{\rho}(hD'))} W^{1,r}(\RR^n)\,,\\
	&H^1_\delta(\RR^n) \xrightarrow{I_+} L^2_\delta(\RR^n) \xrightarrow[h^2]{\td{A}_0J^{-1}F_+{Z}^{-1}(1-\td{\rho}(hD'))} H^1_\delta(\RR^n)\,,
	\end{aligned}
\end{equation}
	and 
	\begin{equation}\label{eq:5second}
	C^\infty_c(\RR;L^2(\RR^{n-1})) \xrightarrow{I_+} C^\infty_c(\RR;L^2(\RR^{n-1})) \subset L^2(\RR^n) \xrightarrow[h^2]{\td{A}_0J^{-1}F_+{Z}^{-1}(1-\td{\rho}(hD'))} H^{1}(\RR^n)\,.
	\end{equation}
	Plugging \cref{eq:1,eq:4,eq:5} into \cref{splitupA0thing}, we get that   
	\begin{equation}\label{eq:A0isnice}
	\begin{aligned}
&\td{A}_0(\id-J^+J){Z}^{-1}(1-\td{\rho}(hD')) \colon W^{1,r}(\RR^n) \to_{h^2} W^{1,r}(\RR^n)\,,\\
&\td{A}_0(\id-J^+J){Z}^{-1}(1-\td{\rho}(hD'))\colon H^1_\delta(\RR^n)\to_{h^2} H^1_\delta(\RR^n)\,,
\end{aligned}
	\end{equation}
	and using \cref{eq:1,eq:4second,eq:5second} we have 
	\begin{equation}\label{eq:extenda0part}
\td{A}_0(\id-J^+J){Z}^{-1}(1-\td{\rho}(hD'))\colon C^\infty_c(\RR;L^2(\RR^{n-1}))\to_{h^2} L^2(\RR^n)\,.
	\end{equation}
	We will now define the remainders $R_1$ and $R_2$ in \cref{eq:biglemmapart1eqv}. Splitting up $\td{G}_\phi$ as in \cref{prop43}, by \cref{eq:A0isnice} we have
	\[
		 L^{p'}(\RR^n) \xrightarrow[h^{-2}]{\td{G}_\phi^c} W^{2,p}(\RR^n) \xrightarrow{J} W^{1,p}(\RR^n) \xrightarrow[h^2]{A_0(\id-J^+J){Z}^{-1}(1-\td{\rho}(hD'))} L^p(\RR^n)\,,
	\]
	and 
	\[
		L^{2}_\delta(\RR^n) \xrightarrow[h^{-1}]{\td{G}_\phi^c} H^{2}_{\delta-1}(\RR^n) \xrightarrow{J} H^{1}_{\delta-1}(\RR^n) \xrightarrow[h^2]{A_0(\id-J^+J){Z}^{-1}(1-\td{\rho}(hD'))} L^2_{\delta-1}(\RR^n)\,,
	\]
	so that we put $R_1\coloneqq A_0(\id-J^+J){Z}^{-1}(1-\td{\rho}(hD'))J\td{G}_\phi^c\itdo$,
	and
	\[
		L^r(\td{\Omega}) \xrightarrow{\itdo} L^{r}(\RR^n) \xrightarrow{\td{G}_\phi -\td{G}_\phi^c} W^{2,r}(\RR^n) \xrightarrow{J} W^{1,r}(\RR^n) \xrightarrow[h^2]{A_0(\id-J^+J){Z}^{-1}(1-\td{\rho}(hD'))} L^r(\RR^n)\,,
	\]
	which we take to be $R_2$.
	
	Furthermore, to take care of \cref{eq:biglemmapart1eqf}, notice that by \cref{eq:extenda0part,defofExt}, 
	\[
		\tilde{A}_0(\id-J^+J){Z}^{-1}(1-\td{\rho}(hD'))\mathcal{E} \colon L^2(\RR^{n-1}) \to_{h^2} L^2(\RR^n)\,,
	\]
	which we define to be $hR_{c,1}$.
\end{proof}

\begin{proof}[Proof of \cref{E1largeFreq} in \cref{largeFreqremainders}]
	Let $1<r<\infty$ be arbitrary. For the estimate \cref{eq:biglemmapart2eqf} we recall $\td{E}_1$ from \cref{eq:defofA0E1E0}. Now using \cref{defofExt}, and \cref{mappingpropsofPb} we may deal with $h\td{E}_1(\id-J^+J){Z}^{-1}(1-\td{\rho}(hD'))\mathcal{E}$ by
	\[
		L^2(\RR^{n-1}) \xrightarrow{(1-\td{\rho}(hD'))\mathcal{E}} C_c^\infty(\RR;L^2(\RR^{n-1})) \xrightarrow{(\id-J^+J){Z}^{-1}} H^1(\RR^n) \xrightarrow[h]{h\tilde{E}_1} L^2(\RR^n)\,,
	\]
	which we then call $hR_{c,2}$.
	
	We turn to the estimate \cref{eq:biglemmapart2eqv}. Let us commute $\tilde{E}_1$ so that it is adjacent to $\tilde{G}_\phi$:
	\begin{align}
		&h\tilde{E}_1(\id-J^+J){Z}^{-1}(1-\td{\rho}(hD'))J\tilde{G}_\phi \notag \\
		&=  h(\id-J^+J){Z}^{-1}(1-\td{\rho}(hD'))J\td{E}_1\td{G}_\phi -h\left([\td{E}_1,J^{-1}]I_+J+J^{-1}I_+[\td{E}_1,J]\right){Z}^{-1}(1-\td{\rho}(hD'))J\td{G}_\phi \notag\\
		&+ h(\id-J^+J)\Bigl([\td{E}_1,{Z}^{-1}](1-\td{\rho}(hD'))J+{Z}^{-1}[\td{E}_1,(1-\td{\rho}(hD'))J]\Bigr)\td{G}_\phi\,.  \label{E1commutatorSplit}
	\end{align}
	We first focus our attention on $h(\id-J^+J){Z}^{-1}(1-\td{\rho}(hD'))J\tilde{E}_1\tilde{G}_\phi$, and intend to distribute this term into $R_3$ and $R_4$. We use \cref{lem:EGestimate} to write
	\begin{equation}\label{eq:somesplitupfromlemma}
	\tilde{E}_1\tilde{G}_\phi = (\tilde{E}_1\tilde{G}_\phi)^c + \mathrm{Op}_h(S^1_1S^{-2}_1) + h\mathrm{Op}_h(S^0_1(\RR^{n-1}))\tilde{G}_\phi\,,
	\end{equation}
	where 
	\begin{equation}\label{eq:e1gc}
	(\tilde{E}_1\tilde{G}_\phi)^c \colon L^2(\RR^n) \to H^2(\RR^n)\,,\an (\tilde{E}_1\tilde{G}_\phi)^c \colon L^{p'}(\RR^n) \to_{h^{-1}} H^2(\RR^n)\,.
	\end{equation}
	Now, by \cref{mappingpropsofPb},
	\begin{equation}\label{eq:e1gcJ}
		h(\id-J^+J){Z}^{-1}(1-\td{\rho}(hD'))J \colon W^{2,r}(\RR^n) \xrightarrow{(1-\td{\rho}(hD'))J} W^{1,r}(\RR^n) \xrightarrow[h]{h(\id-J^+J){Z}^{-1}} W^{1,r}(\RR^n)\,,
	\end{equation}	
	so that by \cref{eq:e1gc,eq:e1gcJ},
	\begin{align*}
	&h(\id-J^+J){Z}^{-1}(1-\td{\rho}(hD'))J(\tilde{E}_1\tilde{G}_\phi)^c \colon L^{p'}(\RR^n) \to H^1(\RR^n)\,,\\
	&h(\id-J^+J){Z}^{-1}(1-\td{\rho}(hD'))J(\tilde{E}_1\tilde{G}_\phi)^c\colon L^2(\RR^n)\to_h H^1(\RR^n)\,,
	\end{align*}
	which then belongs to $R_3$.
	
	We turn to the second term in \cref{eq:somesplitupfromlemma}. First note that 
	by \cref{lem:calc}
	\[
		(1-\td{\rho}(hD'))J\mathrm{Op}_h(S^1_1S^{-2}_1) \in \op{S^1_1 S^{-1}_1}\,,
	\]
	so that with \cref{eq:definitionofZinverse} and \cref{lem:calc}
	\[
		Z^{-1}(1-\td{\rho}(hD'))J\mathrm{Op}_h(S^1_1S^{-2}_1) \colon L^r(\RR^n) \to W^{1,r}(\RR^n)\,.
	\]
	Therefore,
	\begin{align}
		&L^r(\RR^n) \xrightarrow{{Z}^{-1}(1-\td{\rho}(hD'))J\mathrm{Op}_h(S^1_1S^{-2}_1)} W^{1,r}(\RR^n) \xrightarrow[h]{h(\id-J^+J)} W^{1,r}(\RR^n)\,,\label{eq:a1}
	\end{align}
	which means that $h(\id-J^+J){Z}^{-1}(1-\td{\rho}(hD'))J\mathrm{Op}_h(S^1_1S^{-2}_1)$ falls into $R_4$.
	
	Finally, we turn to the term corresponding to $h\mathrm{Op}_h(S^0_1(\RR^{n-1}))\tilde{G}_\phi$ from \cref{eq:somesplitupfromlemma}. Splitting up $\td{G}_\phi$ with \cref{prop43} we note that
	\begin{equation}\label{eq:gphimaps}
		\tilde{G}_\phi - \tilde{G}_\phi^c \colon L^r(\RR^n) \to W^{2,r}(\RR^n)\,,\quad \tilde{G}_\phi^c \colon L^2_\delta(\RR^n) \underset{h^{-1}}{\to} H^2_{\delta-1}(\RR^n) \,, \an \td{G}_\phi^c\colon L^{p'}(\RR^n) \underset{h^{-2}}{\to} W^{2,p}(\RR^n)\,.
	\end{equation}
	Now using \cref{eq:e1gcJ}, we see 
	\[
		L^r(\RR^n) \xrightarrow{\td{G}_\phi-\td{G}_\phi^c} W^{2,r}(\RR^n) \xrightarrow[h]{h\op[h,x']{S^0_1(\RR^{n-1})}} W^{2,r}(\RR^n) \xrightarrow[h]{h(\id-J^+J){Z}^{-1}(1-\td{\rho}(hD'))J} W^{1,r}(\RR^n)\,,
	\]
	and thus $h(\id-J^+J){Z}^{-1}(1-\td{\rho}(hD'))Jh\mathrm{Op}_h(S^0_1(\RR^{n-1}))(\tilde{G}_\phi-\tilde{G}_\phi^c)$ is part of $R_4$, and
	\begin{align*}
		&L^{p'}(\RR^n) \xrightarrow[h^{-2}]{\td{G}_\phi^c} W^{2,p}(\RR^n) \xrightarrow[h]{h\op[h,x']{S^0_1(\RR^{n-1})}} W^{2,p}(\RR^n) \xrightarrow[h]{h(\id-J^+J){Z}^{-1}(1-\td{\rho}(hD'))J} W^{1,p}(\RR^n)\,,\\
		&L^{2}_\delta(\RR^n) \xrightarrow[h^{-1}]{\td{G}_\phi^c} H^{2}_\delta(\RR^n) \xrightarrow[h]{h\op[h,x']{S^0_1(\RR^{n-1})}} H^{2}_\delta(\RR^n) \xrightarrow[h]{h(\id-J^+J){Z}^{-1}(1-\td{\rho}(hD'))J} H^{1}_\delta(\RR^n)\,,
	\end{align*}
	so that $h(id-J^+J){Z}^{-1}(1-\td{\rho}(hD'))Jh\mathrm{Op}_h(S^0_1(\RR^{n-1}))\tilde{G}_\phi^c$ is part of $R_3$. 
	This concludes categorizing the first term in \cref{E1commutatorSplit} into $R_3$ and $R_4$.

	Now let us consider the terms with commutators in \cref{E1commutatorSplit}. 
	Recalling \cref{eq:Zneumann,eq:definitionofZinverse}, we have
	\begin{equation}\label{eq:zt}
		{Z}^{-1} = \op[h,x']{S^{-1}_1(\RR^{n-1})}+hM\,,
	\end{equation}
	where
	\begin{equation}\label{eq:ztm}
		M\colon W^{s,r}(\RR^{n-1}) \to W^{s+1,r}(\RR^{n-1})\,,\quad M \colon H^s_\delta(\RR^{n-1}) \to H^{s+1}_\delta(\RR^{n-1})\,,s\in\{0,1,2\}\,,
	\end{equation}
	so that \cref{eq:zt}, \cref{eq:ztm} and \cref{eq:defofA0E1E0} imply
	\begin{alignat}{3}
		[\td{E}_1,{Z}^{-1}] &\colon W^{1,r}(\RR^n) \to_h W^{1,r}(\RR^n)\,,&&\an [\td{E}_1,{Z}^{-1}] \colon H^1_\delta(\RR^n) \to_h H^1_\delta(\RR^n) \,. \label{eq:comme11}
	\shortintertext{Furthermore, by \cite[p.~116]{CT20},}
		[\td{E}_1,J^{-1}] &\colon L^r(\RR^n) \to_h L^r(\RR^n)\,,&&\an  [\td{E}_1,J^{-1}] \colon L^2_\delta(\RR^n) \to_h L^2_\delta(\RR^n) \label{eq:comme12}
	\shortintertext{and by the pseudo-differential calculus,} 
		[\td{E}_1,(1-\td{\rho}(hD'))J]&\colon W^{2,r}(\RR^n) \to_h W^{1,r}(\RR^n)\,,&&\an [\td{E}_1,(1-\td{\rho}(hD'))J] \colon H^2_\delta(\RR^n) \to_h H^1_\delta(\RR^n)\,, \label{eq:comme13}\\
		[\td{E}_1,J] &\colon W^{1,r}(\RR^n) \to_h L^r(\RR^n)\,,&&\an [\td{E}_1,J] \colon H^1_\delta(\RR^n) \to_h L^2_\delta(\RR^n)\,.
		\label{eq:comme14}
	\end{alignat}
	We perform the explicit estimate for the operator in \cref{E1commutatorSplit} that corresponds to the term in \cref{eq:comme13}: by \cref{eq:gphimaps} and \cref{mappingpropsofPb} we see that 
	\[
		L^r(\RR^n) \xrightarrow{\td{G}_\phi-\td{G}_\phi^c} W^{2,r}(\RR^n) \xrightarrow[h]{[\td{E}_1,(1-\td{\rho}(hD'))J]}  W^{1,r}(\RR^n) \xrightarrow[h]{h(\id-J^+J)Z^{-1}} W^{1,r}(\RR^n)\,,
	\]
	which can be sorted into the $R_4$ bucket,
	and 
	\[
		L^2_\delta(\RR^n) \xrightarrow[h^{-1}]{\td{G}_\phi^c} H^2_{\delta-1}(\RR^n) \xrightarrow[h]{[\td{E}_1,(1-\td{\rho}(hD'))J]}  H^1_{\delta-1}(\RR^n) \xrightarrow[h]{h(\id-J^+J)Z^{-1}} H^1_{\delta-1}(\RR^n)\,,
	\]
	and 
	\[
		L^{p'}(\RR^n) \xrightarrow[h^{-2}]{\td{G}_\phi^c} W^{2,p}(\RR^n) \xrightarrow[h]{[\td{E}_1,(1-\td{\rho}(hD'))J]}  W^{1,p}(\RR^n) \xrightarrow[h]{h(\id-J^+J)Z^{-1}}W^{1,p}(\RR^n)\,,	
	\]
	which we consider part of $R_3$.
	
	Proceeding in a similar fashion for \cref{eq:comme11,eq:comme12,eq:comme14} we see that applying any of the terms containing a commutator in \cref{E1commutatorSplit} to $\tilde{G}_\phi - \tilde{G}_\phi^c$ can be sorted into the $R_4$ bucket and applied to $\td{G}_\phi^c$ land these operators in $R_3$.
	\end{proof}

\begin{proof}[Proof of \cref{E0largeFreq} in \cref{largeFreqremainders}]
	Finally, we turn to the $h^2\tilde{E}_0$ term in \cref{eq:biglemmapart3eqv,eq:biglemmapart3eqf}. Recalling \cref{eq:defofA0E1E0} and \cref{mappingpropsofPb}, and splitting up $\td{G}_\phi$ as in \cref{prop43}, we have
	\[
		L^r(\RR^n) \xrightarrow{\tilde{G}_\phi-\tilde{G}_\phi^c} W^{2,r}(\RR^n) \xrightarrow{(1-\td{\rho}(hD'))J} W^{1,r}(\RR^n) \xrightarrow{(\id-J^+J)Z^{-1}} W^{1,r}(\RR^n) \xrightarrow[h^2]{h^2\tilde{E}_0} W^{1,r}(\RR^n)\,,
	\]
	so that we may take $R_6=h^2\tilde{E}_0(\id-J^+J)Z^{-1}(1-\td{\rho}(hD'))J(\tilde{G}_\phi-\tilde{G}_\phi^c)$.
	
	On the other hand, in a similar manner we see that
	\begin{align*}
&h^2\tilde{E}_0(\id-J^+J)Z^{-1}(1-\td{\rho}(hD'))J\tilde{G}_\phi^c\colon L^{p'}(\RR^n) \to W^{1,p}(\RR^n)\,,\\
&h^2\tilde{E}_0(\id-J^+J)Z^{-1}(1-\td{\rho}(hD'))J\tilde{G}_\phi^c\colon L^2_{\delta}(\RR^n)\to_{h} H^1_{\delta-1}(\RR^n)\,,
	\end{align*} 
	which we call $R_5$.
	
	Furthermore, using reasoning similar to that of \cref{eq:extfworks},
	\[
		h^2\tilde{E}_0(\id-J^+J)Z^{-1}(1-\td{\rho}(hD'))\mathcal{E} \colon L^2(\RR^{n-1}) \to_{h^2} H^1(\RR^n)\,,
	\]
	which becomes $hR_{c,3}$.
\end{proof}

\subsection{Combined Parametrix}\label{sec:combine}

Recall the definition of $\td{\Omega}$ and $\td{\Gamma}\subset \{x_n = 0\}$ from \cref{sec:setup}, $P_l$ from \cref{eq:definitionofPl}, and $P_s$ from \cref{eq:defofPs} and the notation that $\itdo v \in L^{p'}(\RR^n)$ is the trivial extension of any $v\in L^{p'}(\td{\Omega})$.

\begin{proposition}\label{lemm:deltahittingparas}
	Let $v\in L^{p'}(\td{\Omega})$ and $f\in L^2(\td{\Gamma})$. 
	
	If we define the operator
	\begin{equation}\label{eq:defofProp}
		P_r(v,f) \coloneqq (P_l+P_s+B_l)\itdo v + \left(B_c (1+\abs{K}^2)^{1/2} - \op[h,x']{\ell}(1+\abs{K}^2)^{-1/2}\right)f\,,
	\end{equation}
	then in the sense of distributions on $\td{\Omega}$ in the first component,
	\begin{equation}\label{eq:PrinvertsT}
		\begin{pmatrix}
			\itdo h^2\td{\Delta}_\phi \\
			\tbdry[\td{\Gamma}]
		\end{pmatrix}P_r (v,f) = (\id + R_r + hR_r')(v,f)  
	\end{equation}	
	where $R_r = (R_{r,1}, R_{r,2})$ and $R_r' = (R_{r,1}', R_{r,2}')$, and
		\begin{align}
			\notag R_{r,1}&\colon L^{p'}(\td{\Omega})\times L^2(\td{\Gamma})\to_{h^0} L^2(\td{\Omega})\,,\quad &R_{r,2}&\colon L^{p'}(\td{\Omega})\times L^2(\td{\Gamma})\to_{h^0} L^2(\td{\Gamma})\,,\\
			\label{eq:R_rs}R_{r,1}&\colon L^2(\td{\Omega})\times L^2(\td{\Gamma})\to_{h} L^2(\td{\Omega})\,,\quad &R_{r,2}&\colon L^2(\td{\Omega})\times L^2(\td{\Gamma})\to_{h} L^2(\td{\Gamma})\,,\\
			\notag R_{r,1}'&\colon L^s(\td{\Omega})\times L^2(\td{\Gamma})\to_{h^0} L^s(\td{\Omega})\,,\quad &R_{r,2}'&\colon L^s(\td{\Omega})\times L^2(\td{\Gamma})\to_{h^0} L^2(\td{\Gamma})\,,\quad  s\in \{p',2\}\,.
	\end{align}

	Furthermore,
	\[
		\itdo P_r \colon L^{p'}(\td{\Omega}) \times L^2(\td{\Gamma}) \to_{h^{-2}} L^{p}(\td{\Omega}) \cap H^1(\td{\Omega})\,,\an \itdo P_r \colon L^{2}(\td{\Omega}) \times L^2(\td{\Gamma}) \to_{h^{-1}} L^{p}(\td{\Omega}) \cap H^1(\td{\Omega})
	\]
	with bounds
	\begin{alignat}{5}
	&\norm{\itdo P_r(v,f)}_{L^p}&&\lesssim h^{-2}\norm{v}_{L^{p'}} &+& h^{-1}&&\norm{f}_{L^2} \label{Prmapprops1}\\
	&\norm{\itdo P_r(v,f)}_{H^1}&&\lesssim h^{-1}\norm{v}_{L^{2}} &+& &&\norm{f}_{L^2}\label{Prmapprops2}\,.
	\end{alignat}
\end{proposition}
\begin{proof}
	Fix $v \in L^{p'}(\td{\Omega})$, and $f\in L^2(\td{\Gamma})$ where we interpret $f\in L^2(\RR^{n-1})$ by trivial extension. We begin by proving \cref{eq:PrinvertsT} with the estimates stated in \eqref{eq:R_rs}.

	According to \cite[Prop.~5.2]{CT20}, we have
	\begin{equation}\label{eq:Rlprops}
		{\small\begin{aligned}
			\itdo h^2 \td{\Delta}_\phi P_l = (1-\td{\rho}(hD')) + R_l + hR_l' \quad \text{where}\quad R_l &\colon L^2(\td{\Omega}) \to_h L^2(\td{\Omega})\,,\quad R_l \colon L^{p'}(\td{\Omega}) \to_{h^0} L^2(\td{\Omega})\,,\\
		R_l'&\colon L^r(\td{\Omega}) \to_{h^0} L^r(\td{\Omega})\,,\quad 1<r<\infty\,,
\end{aligned}}
\end{equation}
	Furthermore, \cite[Prop.~5.6]{CT20} states that
	\begin{equation}\label{eq:Rsprops}
		\itdo h^2 \td{\Delta}_\phi P_s = \td{\rho}(hD') + hR_s\,,\quad\text{where}\quad R_s \colon L^r(\td{\Omega}) \to_{h^0} L^r(\td{\Omega})\,,\quad 1<r<\infty\,.
	\end{equation}
By \cref{mapofPlfwithDelta1} we have that 
\begin{equation}\label{eq: Delta of Bl}
	\begin{aligned}
\itdo h^2\tilde{\Delta}_\phi B_l = R_{m}+hR_{m}'\quad\text{with}\quad R_{m} &\colon L^2(\td{\Omega})\to_h L^2(\td{\Omega})\,, \quad R_{m} \colon L^{p'}(\td{\Omega}) \to_{h^0} L^2(\td{\Omega})\,, \\
		R_{m}' &\colon L^r(\td{\Omega})\to_{h^0} L^r(\td{\Omega}) ,\quad 1<r<\infty\,.
	\end{aligned}
\end{equation}
	
	Adding \eqref{eq: Delta of Bl}, \eqref{eq:Rsprops}, and \eqref{eq:Rlprops}, we see that, as an element of $\mathcal D(\tilde \Omega)$,
	\begin{equation}
		\itdo h^2\td{\Delta}_\phi\left[P_l+P_s+B_l\right]\itdo v =(1+R_l+hR_l'+hR_s+R_m+hR_m')v\,,\label{eq:deltaonpint}
	\end{equation}
	and using \cref{mapofPlfwithDelta2} and \cref{eq:remainderforsmall},
	\begin{equation}
	\itdo h^2\td{\Delta}_\phi\left[B_c(1+\abs{K}^2)^{1/2} - \op[h,x']{\ell}(1+\abs{K}^2)^{-1/2}\right]f =h(R_c(1+\abs{K}^2)^{1/2}-\itdo R_{\ell}(1+\abs{K}^2)^{-1})f\,.\label{eq:deltaonpbdry}
	\end{equation}
	
	Now adding \eqref{eq:deltaonpint} and \eqref{eq:deltaonpbdry} we get 
	\begin{equation}\label{eq:deltapr}
		\itdo h^2\td{\Delta}_\phi P_r(v,f) = v + R_{r,1}(v,f) + hR_{r,1}'(v,f)\,.
	\end{equation}
	where 
	\begin{equation*}
		R_{r,1}(v,f)\coloneqq (R_l+R_m)v\,,\quad R_{r,1}'(v,f)\coloneqq (R_l'+R_m'+R_s)v+ (R_c(1+\abs{K}^2)^{1/2}-\itdo R_\ell(1+\abs{K}^2)^{-1/2})f\,,
	\end{equation*}
	satisfy the estimates stated in \eqref{eq:R_rs}.

Now we compute the action of $\tbdry[\td{\Gamma}]$ on each of the five terms appearing in $P_r$. We observe that by \cref{traceofpl}, \cref{Plfboundaryprops}, and \cref{eq:Psprops} respectively
\begin{alignat*}{2}
	&\tbdry[\td{\Gamma}] P_lv& &= -(1+\abs{K}^2)^{1/2}(1-\td{\rho}(hD'))\tau_{\td{\Gamma}} J\td{G}_\phi v\\
	&\tbdry[\td{\Gamma}] B_lv& &= (1+\abs{K}^2)^{1/2}(1-\td{\rho}(hD'))\tau_{\td{\Gamma}} J\td{G}_\phi v + h(R_{bl}'+R_{bl}'')v \\
	&\tbdry[\td{\Gamma}] P_sv &&= 0\,,
\end{alignat*}
where in the last line we used that $\supp v \subset \RR^n_+$. 
Combining these equations, we get, denoting by $1_{\td{\Gamma}}$ the indicator of $\td{\Gamma}$,
	\begin{align}
	\label{bigmess}
	\tbdry[\td{\Gamma}]\left(P_l+P_s+B_l\right)v = h1_{\td{\Gamma}}(R_{bl}'+R_{bl}'')v\,.
	\end{align}
where $R_{bl}'$ and $R_{bl}''$ satsifies the estimates given in \eqref{eq:bdryremainBlBc}. Finally, we invoke \cref{boundaryofsmall} and \cref{Plfboundarypropscpart} so that
	\begin{align*}
		\tbdry[\td{\Gamma}] B_c (1+\abs{K}^2)^{1/2} &= 1_{\td{\Gamma}}(1-\td{\rho}(hD')) + h 1_{\td{\Gamma}}R_{bc}(1+\abs{K}^2)^{1/2}\,,\\
		\tbdry[\td{\Gamma}] \op[h,x']{\ell}(1+\abs{K}^2)^{-1/2} &= -1_{\td{\Gamma}}\td{\rho}(hD')+ h1_{\td{\Gamma}}R_b(1+\abs{K}^2)^{-1/2}\,
	\end{align*}
for  
\begin{equation}
\label{eq:Rb and Rbc estimates}
\begin{aligned}
	R_b &\colon L^r(\RR^{n-1}) \to_{h^0} L^r(\RR^{n-1})\quad\text{for}\quad 1<r<\infty\,,\\
	R_{bc} &\colon L^2(\RR^{n-1}) \to_{h^0} L^2(\RR^{n-1})\,.
\end{aligned}
\end{equation}
Adding both of the above equations to \eqref{bigmess} we obtain
\begin{equation}\label{eq:bdrypr}
		\tbdry[\td{\Gamma}] P_r(v,f) = f + R_{r,2}(v,f) + hR_{r,2}'(v,f)\,.
\end{equation}
where
	\[
		R_{r,2}(v,f) = h1_{\td{\Gamma}}R_{bl}'v\,,\an R_{r,2}'(v,f) = 1_{\td{\Gamma}}R_{bl}''v + (1_{\td{\Gamma}}R_{bc}(1+\abs{K}^2)^{1/2} - 1_{\td{\Gamma}}R_b(1+\abs{K}^2)^{-1/2})f.
	\]

	Using \eqref{eq:bdryremainBlBc} to estimate $1_{\td{\Gamma}}R_{bl}'v$ and $1_{\td{\Gamma}}R_{bl}''v$ (where $1_{\td{\Gamma}}R_{bl}' \colon L^2(\td{\Omega}) \to L^2(\td{\Gamma})$ since $\td{\Gamma}$ is bounded) and using \eqref{eq:Rb and Rbc estimates} to estimate $R_b$ and $R_{bc}$ we obtain the estimates stated in \eqref{eq:R_rs}. So combining \eqref{eq:bdrypr} with \eqref{eq:deltapr} we arrive at \eqref{eq:PrinvertsT} with the appropriate remainder estimates for $R_r$ and $R_r'$.

We begin now with the proof of \cref{Prmapprops1,Prmapprops2}. According to \cite[Prop.~5.1]{CT20} we have
	\begin{equation}\label{eq:mapprop1}
		\itdo P_l \colon L^{p'}(\RR^n) \to_{h^{-2}} L^p(\td{\Omega}) \cap H^1(\td{\Omega})\,,\an\itdo P_l \colon L^2(\RR^n) \to_{h^{-1}} H^1(\td{\Omega}) \hookrightarrow_{h^{-1}} L^p(\td{\Omega})\,,
	\end{equation}
	where these latter estimates are unweighted due to the fact they are only considered on $\td{\Omega}$. Furthermore, \cite[Prop.~5.6]{CT20} gives 
	\begin{equation}
	\begin{aligned}\label{eq:mapprop2}
		\itdo P_s &\colon L^{p'}(\RR^n) \to W^{2,p'}(\td{\Omega}) \hookrightarrow_{h^{-1}} H^1(\td{\Omega}) \hookrightarrow_{h^{-1}} L^p(\td{\Omega})\,, \\
		\itdo P_s &\colon L^2(\RR^n) \to H^2(\td{\Omega}) \hookrightarrow_{h^{-1}} L^p(\td{\Omega})\,.
	\end{aligned}
\end{equation}
Combining \cref{eq:mapprop1}, \cref{eq:mapprop2}, and \cref{mappropsofPdLpMpart}, 
	\begin{align*}
&\itdo\left(P_l+P_s+B_l\right)\itdo \colon L^{p'}(\td{\Omega}) \to_{h^{-2}} L^p(\td{\Omega}) \cap H^1(\td{\Omega})\,,\quad \itdo\left(P_l+P_s+B_l\right)\itdo\colon L^2(\td{\Omega}) \to_{h^{-1}} H^1(\td{\Omega})\,.
	\end{align*}
	Now due to \cref{mappropsofPdLpCpart} and \cref{eq:ellsemiclass}, we have
	\[
		\itdo B_c \colon L^{2}(\RR^{n-1}) \to H^1(\td{\Omega}) \hookrightarrow_{h^{-1}} L^p(\td{\Omega})\,,\quad \itdo\op{\ell} \colon L^{2}(\RR^{n-1}) \to H^1(\td{\Omega}) \hookrightarrow_{h^{-1}} L^p(\td{\Omega})\,,
	\]
	which concludes the proof of \cref{Prmapprops1,Prmapprops2}.
\end{proof}

\section{Green's Function}\label{sec:nonlinearR}

\numberwithin{equation}{section}

Recall the definitions of $\Omega,\Gamma$ and $\td{\Gamma}$ from \cref{sec:intro,changeofVars}, and let $\gamma^\ast$ and $(\gamma^{-1})^\ast$ denote the pull-backs by the change of variable $\gamma$ from \cref{eq:defofgamma} and its inverse. Let $\tau_{\Gamma}$ and $\tau_{\td{\Gamma}}$ be the traces onto $\Gamma$ and $\td{\Gamma}$ respectively.

In this section we will turn the parametrix constructed in \cref{lemm:deltahittingparas} into an exact Green's function $G_\Gamma$ for the boundary value problem 
\[
	h^2\Delta_\phi G_\Gamma(v,f) = v\,,\an \tau_{\Gamma}(\del_\nu + \del_\nu\phi) G_\Gamma(v,f) = -f\,,
\]
which is the content of the following proposition, the proof of which follows \cite[Prop.~6.1]{CT20}. As a consequence we will also prove \cref{thm:G}.

\begin{proposition}\label{GnoGlue}
	There is a linear map $G_{\Gamma}: L^{p'}(\Omega)\times L^2(\Gamma) \to L^p(\Omega) \cap H^1(\Omega)$ satisfying,
	\begin{align}\label{Ggammaisinverse}
		\langle u, h^2\Delta_\phi G_{\Gamma} (v,f)\rangle &= \langle u,v\rangle \qquad \forall u\in C_c^\infty(\Omega)\,,
		\shortintertext{and} 
		\label{Ggammaboundary}
		\tau_{\Gamma}(\del_\nu + \del_\nu\phi) G_{\Gamma}(v,f) &= -f\,.
	\end{align}
	Furthermore, 
	\[
		G_\Gamma\colon L^{p'}(\Omega)\times L^2(\Gamma) \to_{h^{-2}} L^p(\Omega) \cap H^1(\Omega)\,,\quad\text{and}\quad G_\Gamma\colon L^{2}(\Omega)\times L^2(\Gamma) \to_{h^{-1}} H^1(\Omega)\,,
	\]
	with bounds
	\begin{alignat}{5}
	&\norm{G_{\Gamma}(v,f)}_{L^p}&&\lesssim h^{-2}\norm{v}_{L^{p'}} &+& h^{-1}&&\norm{f}_{L^2} \label{Ggammamapprops1}\\
	&\norm{G_{\Gamma}(v,f)}_{H^1}&&\lesssim h^{-1}\norm{v}_{L^{2}} &+& &&\norm{f}_{L^2}\label{Ggammamapprops2}\,.
	\end{alignat}
	Finally, 
	\begin{align}
		(\gamma^{-1})^\ast &\circ G_{\Gamma}(v,f) = P_r(\id+R)(\td{v},-\td{f})\,,\label{Ggammatransform}
	\end{align}
	where $P_r$ is the operator defined in \cref{eq:defofProp} and
	\begin{alignat}{2}
		R &\colon L^2(\td{\Omega})\times L^2(\td{\Gamma}) \to_{h} L^2(\td{\Omega})\times L^2(\td{\Gamma})\,,&\an R&\colon L^{p'}(\td{\Omega})\times L^2(\td{\Gamma}) \to_{h^0} L^{2}(\td{\Omega})\times L^2(\td{\Gamma})\,.\label{GgammaRprops}
	\end{alignat}
\end{proposition}
\begin{proof}
	Let $v \in L^{p'}(\Omega)$ and $f\in L^2(\Gamma)$ be fixed. Put $\td{v} \coloneqq (\gamma^{-1})^\ast v \in L^{p'}(\td{\Omega})$ and $\td{f} \coloneqq 1_{\td{\Gamma}} (\gamma^{-1})^\ast f \in L^2(\td{\Gamma})$.

	Note that according to \cref{eq:PrinvertsT} we have
	\begin{equation}\label{eq:deltahittingpr}
		 \begin{pmatrix}
			\itdo h^2\td{\Delta}_\phi \\
			\tbdry[\td{\Gamma}]
		\end{pmatrix} P_r(\td{v},\td{f}) = (\id + R_r + hR_r')(\td{v},\td{f}),	
	\end{equation}
	with $R_r, R_r'$ satisfying the estimates given in \cref{eq:R_rs}, which in particular means that
	\begin{equation}\label{eq:Rr}
			R_{r}\colon L^{p'}(\td{\Omega})\times L^2(\td{\Gamma})\to_{h^0} L^2(\td{\Omega})\times L^2(\td{\Gamma})\,,\quad R_{r}\colon L^2(\td{\Omega})\times L^2(\td{\Gamma})\to_{h} L^2(\td{\Omega})\times L^2(\td{\Gamma})\,.
	\end{equation}
	We begin by removing the remainder $R_r'$ in \cref{eq:deltahittingpr}. Let 
	\begin{equation}\label{eq:defofS}
		S \coloneqq (\id+hR_r')^{-1}\colon L^s(\td{\Omega})\times L^2(\td{\Gamma})\to_{h^0} L^s(\td{\Omega})\times L^2(\td{\Gamma})\,,\quad s\in\{p',2\}
	\end{equation}
	so that 
	\begin{equation}\label{eq:righthasinverse}
		(\id+R_r+hR_r')S(\td{v},\td{f}) = (\id+R_rS)(\td{v},\td{f})
	\end{equation}
	and the right-hand side of \cref{eq:righthasinverse} has the inverse 
	\begin{equation}\label{neumanninversetwostep}
		(\id+R_rS)^{-1}(\td{v},\td{f}) = (\td{v},\td{f}) - \sum_{k\geq 0}(-R_rS)^kR_rS(\td{v},\td{f}) \in L^{p'}(\td{\Omega})\times L^2(\td{\Gamma})\,,
	\end{equation}
	which is well-defined because of \cref{eq:Rr}. This means that
	\begin{equation}\label{eq:sinvert}
		(\id + R_r + hR_r')S(\id+R_rS)^{-1}(\td{v},\td{f})= (\td{v},\td{f})\,.
	\end{equation}
	and from \cref{eq:deltahittingpr} we have that 
	\begin{equation}\label{eq:inverse of matrix laplace'}\begin{pmatrix}
			\itdo h^2\td{\Delta}_\phi \\
			\tbdry[\td{\Gamma}]
		\end{pmatrix}P_rS(\id+R_rS)^{-1}  = \id\,.
	\end{equation}

	Now by \cref{eq:defofS} and \cref{eq:Rr}, we have the estimates
	\begin{align*}
		&\sum_{k\geq 0}(-R_rS)^kR_rS\colon L^{p'}(\td{\Omega})\times L^2(\td{\Gamma}) \to_{h^0} L^2(\td{\Omega})\times L^2(\td{\Gamma})\,, \\
		&\sum_{k\geq 0}(-R_rS)^kR_rS\colon L^2(\td{\Omega})\times L^2(\td{\Gamma}) \to_h L^2(\td{\Omega})\times L^2(\td{\Gamma})\,.
	\end{align*}
Using \eqref{neumanninversetwostep} and \eqref{eq:defofS} we see that
	\begin{equation}\label{eq:splitupforR}
		R\coloneqq S(\id+R_rS)^{-1} - \id = \sum_{k\geq 1}(-hR_r)^k - S\sum_{k\geq 0}(-R_rS)^kR_rS\,.
	\end{equation}
	So
	\[
		R \colon L^2(\td{\Omega})\times L^2(\td{\Gamma}) \to_h L^2(\td{\Omega})\times L^2(\td{\Gamma})\,,\an R\colon L^{p'}(\td{\Omega})\times L^2(\td{\Gamma}) \to_{h^0} L^2(\td{\Omega})\times L^2(\td{\Gamma})\,,
	\]
	 which proves \cref{GgammaRprops}.
	 In view of \eqref{eq:inverse of matrix laplace'}, and \cref{eq:splitupforR}, we have 
	\begin{equation}\label{eq:prinvert}
		\langle h^2\td{\Delta}_\phi P_{r}(\id+R)(\td{v}, \td{f}),\td{u}\rangle =\langle \td{v},\td{u}\rangle\,,\quad \forall \td{u} \in C_c^\infty(\RR^n_+)\,,\an \tbdry P_r(\id+R)(\td{v},\td{f}) = \td{f}\,.
\end{equation}
	Now we define 
	\begin{equation}\label{eq:defofGgamma}
		G_{\Gamma}(v,f) \coloneqq \gamma^\ast \circ P_r(\id+R)((\gamma^{-1})^\ast v,-1_{\td{\Gamma}}(\gamma^{-1})^\ast f) = \gamma^\ast \circ P_r(\id+R)(\td{v},-\td{f})\,,
	\end{equation}
	so that \cref{Ggammatransform} is satisfied. 

	Furthermore, by \cref{eq:prinvert}, for all $u \in C_c^\infty(\Omega)$,
	\begin{align*}
		\langle u, h^2\Delta_\phi G_\Gamma(v,f)\rangle_{\Omega} &= \langle u, \gamma^\ast \circ h^2\td{\Delta}_\phi P_r(\id+R)((\gamma^{-1})^\ast v,\td{f}) \rangle_{\Omega} \\
		&= \langle \mathrm{det}\abs{(\gamma^{-1})'}(\gamma^{-1})^\ast u , h^2\td{\Delta}_\phi P_r(\id+R)((\gamma^{-1})^\ast v,\td{f}) \rangle_{\td{\Omega}} \\
		&= \langle \mathrm{det}\abs{(\gamma^{-1})'} (\gamma^{-1})^\ast u, (\gamma^{-1})^\ast v \rangle_{\td{\Omega}} \\
		&= \langle u,v\rangle_{\Omega}\,,
	\end{align*}
	which is to say that in the distributional calculus on $\Omega$, $G_\Gamma$ is the inverse of $h^2\Delta_\phi$, proving \cref{Ggammaisinverse}.
	Furthermore, by \cref{eq:boundarycalcwithgamma,eq:prinvert},
	\[
		\tau_\Gamma (h\del_\nu +\del_\nu\phi) G_{\Gamma} (v,f) = \gamma^\ast \circ \tbdry[\td{\Gamma}] P_r(\td{v},-\td{f}) = -\gamma^\ast \circ (\gamma^{-1})^\ast f = -f\,,
	\]
	which shows \cref{Ggammaboundary}.

	Note that by \cref{GgammaRprops},
	\[
		\id+R \colon L^2(\td{\Omega})\times L^2(\td{\Gamma})\to_{h^0} L^2(\td{\Omega})\times L^2(\td{\Gamma})\,,\quad \id+R\colon L^{p'}(\td{\Omega})\times L^2(\td{\Gamma})\to_{h^0} L^{p'}(\td{\Omega})\times L^2(\td{\Gamma})
	\]
	so that after the change of variables with $\gamma$, the properties \cref{Prmapprops1,Prmapprops2} from $P_r$ are passed on verbatim to $G_\Gamma$ via \cref{eq:defofGgamma}, proving \cref{Ggammamapprops1,Ggammamapprops2}.
\end{proof} 

\subsection{Proof of Theorem \ref{thm:G}}\label{gluingsection}

The only thing left is to glue together different parts of the boundary and apply \cref{GnoGlue} to each part. 
This proof follows that of \cite[Thm.~1.3]{CT20} but will be reproduced nevertheless.

\begin{proof}[Proof of \cref{thm:G}]
	Take $\Xi = \del\Omega\setminus B$. Refer to \cref{picofGamma} for the following definitions. Referring to \cref{sec:intro}, by assumption on $\Xi$, we have $\bigcup_{j=1}^N \Gamma_j = \Xi$ where each $\Gamma_j$ is defined by some function $g_j \in C_c^\infty(\RR^{n-1})$, refer to \cref{eq:gdescribesboundary}, and the $\Gamma_j$ are pairwise positively separated. 

	Now fix $1 \leq l \leq N$ and put $\Gamma = \Gamma_l$, $g = g_l$.
	Without loss of generality we may assume that there exists some open neighborhood $\Gamma \Subset \Gamma' \subset \partial\Omega$ so that $y_n \geq g(y')$ for all $y\in \Gamma'$, and $\Gamma'$ is positively separated from all $\Gamma_j\neq\Gamma, 1\leq j\leq N$. Were this not the case, we could replace $g$ by some function $g' \in C_c^\infty(\RR^{n-1})$ so that the boundary of $\Omega$ is also defined by $g'$ in a neighborhood of $\Gamma$ and this $\Gamma'$ exists.
	
	The existence of $\Gamma'$ implies there is some open neighborhood $W\subset \RR^n$ of $\Gamma$ so that $W\cap\Omega \subset \{y\in W\colon y_n > g(y')\}$ and $W\cap \partial\Omega = \Gamma'$. 

	Let $\chi \in C_c^\infty(\RR^n)$ so that $\chi$ is identically one in a neighborhood of $\{y\in W\cap\partial\Omega\colon y_n=g(y')\}$ and $\supp \chi \Subset W$. These properties imply that there is some $\varepsilon > 0$ so that 
	\begin{equation}\label{eq:epsilonforchi}
		\supp(1_\Omega \nabla\chi) \subset \{y_n \geq g(y')+\varepsilon\}\,.
	\end{equation}
	In fact, we choose the support of $\chi$ so that \cref{eq:epsilonforchi} holds and $\supp \chi$ is positively separated from every $\Gamma_j\neq\Gamma,1\leq j\leq N$.
	
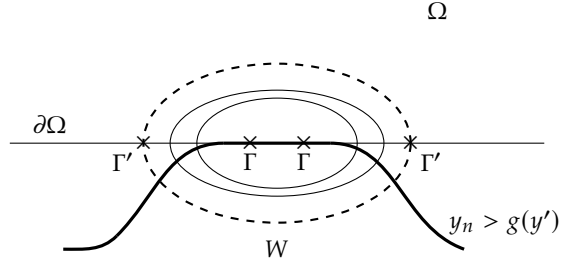
\begin{figure}
	\centering
\begin{tikzpicture}
	%\node[ellipse, draw, minimum width=120pt,minimum height=80pt] (e) at (0,0) {};
	%\fill[pattern={Dots},pattern color={darkgray}] (0,0) --  (0:60pt and 40pt) arc(0:180:60pt and 40pt) -- cycle;
	%\node[ellipse, draw, minimum width=140pt,minimum height=100pt] (e) at (0,0) {};
	%\fill[white] (-125pt,0)rectangle(125pt,-65pt);
	\draw (-100pt,0) -- (100pt,0);

	\node[ellipse, dashed, line width={0.8pt}, draw, minimum width=100pt,minimum height=60pt] (e) at (0,0) {};
	\node[ellipse, draw, minimum width=80pt,minimum height=40pt] (e) at (0,0) {};
	\node[ellipse, draw, minimum width=60pt,minimum height=34pt] (e) at (0,0) {};

	\node[below] at (10pt,0) {\footnotesize$\Gamma$};
	\node[below] at (-10pt,0) {\footnotesize$\Gamma$};
	%\node[below] at (20pt,0) {\footnotesize$\chi$};
	%\node[below] at (-20pt,0) {\footnotesize$\chi$};
%	\node[below] at (30pt,0) {\footnotesize$\chi$};
%	\node[below] at (-30pt,0) {\footnotesize$\chi$};
	%\node[below] at (40pt,0) {\footnotesize$\chi=0$};
	%\node[below] at (-40pt,0) {\footnotesize$\chi=0$};
	\node[below left] at (-50pt,0) {\footnotesize$\Gamma'$};
	\node[below right] at (50pt,0) {\footnotesize$\Gamma'$};
	\node at (50pt,0) {\footnotesize$\times$};
	\node at (-50pt,0) {\footnotesize$\times$};
	%\node at (-20pt,0) {\footnotesize$\times$};
	%\node at (20pt,0) {\footnotesize$\times$};
	\node at (-10pt,0) {\footnotesize$\times$};
	\node at (10pt,0) {\footnotesize$\times$};
	\node at (85pt,-30pt) {\footnotesize $y_n > g(y')$};
	\node at (60pt,50pt) {\footnotesize $\Omega$};
	\node at (0, -40pt) {\footnotesize $W$};
	%\node at (0, 50pt) {\footnotesize $W$};
	\node[above] at (-85pt,0) {\footnotesize $\del\Omega$};
	\draw[very thick] (-80pt,-40pt) to [out=0,in=225] (-60pt,-35pt);
	\draw[very thick] (-60pt,-35pt) to [out=45,in=180] (-20pt,0);
	\draw[very thick] (-20pt,0pt) to [out=0,in=180] (20pt,0pt);
	\draw[very thick] (20pt,0pt) to [out=0,in=160] (70pt,-40pt);
	\end{tikzpicture}
	\caption{To simplify sketching it, we consider the boundary $\del\Omega$ to be flat, where $\Omega$ sits above the horizontal line. The  sets $\Gamma$ and $\Gamma'$ are denoted by their endpoints
	%, and the marks with label $\chi$ correspond to the endpoints (inside of $\del\Gamma$) of the set where $\chi \equiv 1$. 
	The interior of the inner-most ellipse is the set in which $\chi = 1$, the area inside the next smallest ellipse is the support of $\chi$. The set inside the ellipse marked by the dashed line denotes $W$. Finally, the area above the thick line is where $y_n > g(y')$.}
	\label{picofGamma}
\end{figure}

Fix $v\in L^{p'}(\Omega)$ and $f\in L^2(\Xi)$, and let $1_{\Gamma}f \in L^2(\Gamma)$ be the restriction of $f$ onto $\Gamma$. 

Let $G_{\Gamma}$ be the Green's function for $\Omega$ from \cref{GnoGlue}, and define 
\[
	\Pi_{\Gamma}(v,f) \coloneqq \chi 1_{\Omega}(G_\phi 1_\Omega v-G_{\Gamma}(v,1_{\Gamma}f))\,,
\]
where $G_\phi$ is the inverse of $h^2\Delta_\phi$ from \cref{prop43}.

Note that by \cref{prop43}, $\Pi_{\Gamma}$ inherits the same mapping properties as $G_{\Gamma}$ (see \cref{Ggammamapprops1,Ggammamapprops2}), namely that
\[
\Pi_\Gamma\colon L^{p'}(\Omega)\times L^2(\td{\Gamma}) \to_{h^{-2}} L^p(\Omega)\,,\an \Pi_\Gamma\colon L^{2}(\Omega)\times L^2(\td{\Gamma}) \to_{h^{-1}} H^1(\Omega)\,,
\]
with bounds
\begin{alignat}{5}
	&\norm{\Pi_{\Gamma}(v,f)}_{L^p}&&\lesssim h^{-2}\norm{v}_{L^{p'}} &+& h^{-1}&&\norm{f}_{L^2} \label{eq:Pi1}\\
	&\norm{\Pi_{\Gamma}(v,f)}_{H^1}&&\lesssim h^{-1}\norm{v}_{L^{2}} &+& &&\norm{f}_{L^2}\,. \label{eq:Pi2}
\end{alignat}

Furthermore, by \cref{Ggammaboundary} and due to the fact that $\chi \equiv 1$ on $\Gamma$, 
\[
	\tau_{\Gamma}(h\del_\nu + \del_\nu\phi)\Pi_{\Gamma} = \tau_{\Gamma}(h\del_\nu+\del_\nu\phi)G_\phi + 1_{\Gamma}f \,.
\]

At the end of this section we will show the following lemma.
\begin{lemma} \label{lemma62}
	The operator 
	\begin{equation}\label{eq:defofRfgamma}
		R_{f,\Gamma} \coloneqq h^2\Delta_\phi 1_\Omega\Pi_{\Gamma}(\cdot, f)
	\end{equation}
	satisfies the mapping property
	\begin{equation}\label{eq:Rfgammaprops}
		R_{f,\Gamma}\colon L^{p'}(\Omega)\to_{h^0} L^2(\Omega)\,,\an R_{f,\Gamma}\colon L^2(\Omega) \to_h L^2(\Omega)\,.
	\end{equation}
\end{lemma}

Recall that $\Xi = \bigcup_{j=1}^N \Gamma_j$, and each $\Gamma_j$ is defined by $g_j \in C^\infty_c(\RR^{n-1})$. Construct $\Pi_{\Gamma_j}$ as above for each $\Gamma_j$ with $\chi_j$ that has disjoint support from every other $\chi_s$ for $j\neq s$. Defining
\[
	\Pi(v,f) \coloneqq G_\phi v - \sum_{j=1}^N \Pi_{\Gamma_j}(v,f)\,,
\]
we have 
\begin{equation}\label{eq:boundaryPi}
	\tau_{\Xi} (h\del_\nu + \del_\nu\phi)\Pi (v,f) = -f \,,\an
	h^2\Delta_\phi 1_\Omega \Pi(v,f) = (\id + R_f)v\,,
\end{equation}
where for $R_{f,\Gamma_j}$ from \cref{eq:defofRfgamma} and by \cref{eq:Rfgammaprops},
\[
	R_f \colon L^2(\Omega)\to_h L^2(\Omega)\,,\an R_f\colon L^{p'}(\Omega) \to L^2(\Omega)\,,\an R_f = \sum_{j=1}^N R_{f,\Gamma_j}\,.
\]
Inverting by Neumann series similar to \cref{neumanninversetwostep} we set 
\begin{equation}\label{eq:defofG}
	G_{\Xi}(v,f) \coloneqq \Pi ((\id+R_f)^{-1}v,f)\,.
\end{equation}
so that \cref{eq:Pi1,eq:Pi2} imply \cref{eq:hG1,eq:hG2}. Furthermore, \cref{eq:defofG,eq:boundaryPi} imply \cref{eq:boundaryG,eq:inverseG}, and \cref{eq:Gshorth} is implied by \cref{eq:hG1,eq:hG2}.
\end{proof}

\begin{proof}[Proof of \cref{lemma62}]
In this proof we will act like $f$ is a fixed variable and the mapping identities we state will only be about the first variable $v$. This will result in us finding a remainder $R_f$ that depends on and may be considered a function of $f$.

Note that in the calculus of distributions on $\Omega$, $1_\Omega$ commutes with $h^2\Delta_\phi$ so that
\[
	\chi h^2\Delta_\phi 1_\Omega G_{\Gamma} = \chi 1_\Omega h^2\Delta_\phi G_{\Gamma} = \chi
\]
as distributions on $\Omega$ by \cref{Ggammaisinverse}. 
Additionally, by \cref{prop43}, $G_\phi$ inverts $h^2\Delta_\phi$ so that as distributions on $\Omega$, we have $\chi h^2\Delta_\phi 1_\Omega G_\phi = \chi$. Combining this yields, distributionally on $\Omega$,
\begin{equation}\label{eq:beforecoordswap}
	h^2\Delta_\phi \Pi_{\Gamma}(v,f) = [h^2\Delta_\phi,\chi]1_{\Omega}(G_\phi 1_{\Omega}v - G_{\Gamma}(v,1_{\Gamma}f))\,.
\end{equation}
Changing coordinates, and noting that $[h^2\td{\Delta}_\phi,\td{\chi}]$ and $\itdo$ commute in distribution, \cref{eq:beforecoordswap,Ggammatransform} give us
\begin{equation}\label{thetermsinlemma62}
	h^2\Delta_\phi \Pi_{\Gamma}(v,f) = \gamma^\ast \circ 1_{\td{\Omega}}[h^2\td{\Delta}_\phi,\td{\chi}](\td{G}_\phi \itdo \td{v} - P_r(\id+R)(\td{v},-\td{f}))
\end{equation}
where $P_r$ is defined in \cref{eq:defofProp}, $R$ satisfies \cref{GgammaRprops}, $\td{v} = (\gamma^{-1})^\ast v$ and $\td{f} = 1_{\td{\Gamma}}(\gamma^{-1})^\ast (1_\Gamma f)$, and we remark that $h^{-1}[h^2\td{\Delta}_\phi,\td{\chi}]$ is a pseudo-differential operator of order $1$ on $\RR^n$. 

Notice that
\begin{equation}\label{jointlylinearPr}
	1_{\td{\Omega}}[h^2\td{\Delta}_\phi,\td{\chi}]P_r(id+R)(\td{v},-\td{f}) = 1_{\td{\Omega}}[h^2\td{\Delta}_\phi,\td{\chi}]\left(P_rR(\td{v},-\td{f}) + P_r(\td{v},-\td{f})\right)\,,
\end{equation}
where due to \cref{GgammaRprops} the first term of \cref{jointlylinearPr} maps
\begin{align*}
	&L^{p'}(\td{\Omega}) \xrightarrow{R(\cdot,-\td{f})} L^2(\td{\Omega})\times L^2(\td{\Gamma}) \xrightarrow[h^{-1}]{P_r} H^1(\td{\Omega}) \xrightarrow[h]{1_{\td{\Omega}}[h^2\td{\Delta}_\phi,\td{\chi}]} L^2(\td{\Omega})\,,\\
	&L^{2}(\td{\Omega}) \xrightarrow[h]{R(\cdot,-\td{f})} L^2(\td{\Omega})\times L^2(\td{\Gamma})\xrightarrow[h^{-1}]{P_r} H^1(\td{\Omega}) \xrightarrow[h]{\itdo [h^2\td{\Delta}_\phi,\td{\chi}]} L^2(\td{\Omega})\,,
\end{align*}
which, after applying $\gamma^\ast$, can be sorted into $R_{\Gamma,f}$.

Thus, the only term in \cref{thetermsinlemma62} we still have to estimate according to \cref{eq:Rfgammaprops} is 
\[
	1_{\td{\Omega}}[h^2\td{\Delta}_\phi,\td{\chi}](\td{G}_\phi \itdo \td{v} - P_r(\td{v},-\td{f}))\,.
\]

Note that 
\begin{align}\label{splitPfagain}
	P_r(\td{v},\td{f}) &= \left(P_l+P_s+B_l\right)\itdo\td{v} + \left(B_c(1+\abs{K}^2)^{1/2} - \op[h,x']{\ell}(1+\abs{K}^2)^{-1/2}\right)\td{f}\,,
\end{align}
where by \cite[Prop.~5.6]{CT20}, for the second term in \cref{splitPfagain},
\[
	P_s\colon L^2(\RR^n) \to_{h^0} H^2(\RR^n)\,,\an P_s\colon L^{p'}(\RR^n)\to W^{2,p'}(\RR^n) \hookrightarrow_{h^{-1}} H^1(\RR^n)\,,
\]
and thus applying $[h^2\td{\Delta}_\phi,\td{\chi}]$ to this output grants another $h$ and costs one regularity, which fits the profile of $R_{f,\Gamma}$.

For the fifth term in  \cref{splitPfagain} we argue that by \cref{eq:ellsemiclass}, 
\[
	\itdo\op[h,x']{\ell}\colon L^2(\RR^{n-1}) \to_{h^0} H^2(\td{\Omega})\,,
\]
so that $\itdo [h^2\td{\Delta}_\phi,\td{\chi}]\op[h,x']{\ell} (1+\abs{K}^2)^{-1/2}f \in hL^2(\td{\Omega})$, which we consider part of $R_{f,\Gamma}$.

We consider the third term in \cref{splitPfagain}. 
Since $J^{-1}J = \id$ we have $\id- J^+J = J^{-1}(\id -I_+)J = J^{-1}1_{\RR^n_-}J$. Due to the fact that only terms containing derivatives of $\td{\chi}$ survive in $[h^2\td{\Delta}_\phi, \td{\chi}]$ and \cref{eq:epsilonforchi}, the support of $\itdo [h^2\td{\Delta}_\phi, \td{\chi}]$ lies in $x_n > \varepsilon > 0$. Now \cite[Lem.~3.4]{CT20} gives
\begin{equation}\label{chih2}
	\itdo [h^2\td{\Delta}_\phi,\td{\chi}]J^{-1}1_{\RR^n_-} \colon L^r(\RR^n) \to_{h^2} L^r(\RR^n)\,, \qquad 1<r<\infty\,,
\end{equation}
and $B_l = (id-J^+J){Z}^{-1}(1-\td{\rho}(hD'))J\td{G}_\phi = J^{-1}1_{\RR^n_-}J{Z}^{-1}(1-\td{\rho}(hD'))J\td{G}_\phi$, so using \cref{chih2} and \cref{prop43}, we have
\[
	L^2(\td{\Omega}) \xrightarrow[h^{-1}]{J{Z}^{-1}(1-\td{\rho}(hD'))J\td{G}_\phi^c} L^2(\RR^n) \xrightarrow[h^2]{\itdo[h^2\td{\Delta}_\phi,\td{\chi}]J^{-1}1_{\RR^n_-}} L^2(\td{\Omega})\,,
\]
and
\[
	L^{p'}(\td{\Omega}) \xrightarrow[h^{-2}]{J{Z}^{-1}(1-\td{\rho}(hD'))J\td{G}_\phi^c} L^p(\RR^n) \xrightarrow[h^2]{\itdo [h^2\td{\Delta}_\phi,\td{\chi}]J^{-1}1_{\RR^n_-}} L^p(\td{\Omega})\subset L^2(\td{\Omega}) \,,
\]
which fits \cref{eq:Rfgammaprops} for $R_{f,\Gamma}$.

Note that $[h^2\td{\Delta}_\phi,\td{\chi}]$ does not shift support with respect to $x_n$, and that by \cref{mappingpropsofPb}, $(\id-J^+J)Z^{-1} \colon W^{1,r}(\RR^n) \to W^{2,r}(\RR^n_+)$, so
\[
L^{p'}(\RR^n) \xrightarrow{\td{G}_\phi -\td{G}_\phi^c} W^{2,p'}(\RR^n) \xrightarrow{(1-\td{\rho}(hD'))J} W^{1,p'}(\RR^n) \xrightarrow{(\id-J^+J)Z^{-1}} W^{2,p'}(\RR^n_+) \hookrightarrow_{h^{-1}} H^1(\RR^n_+) \xrightarrow[h]{[h^2\td{\Delta}_\phi,\td{\chi}]} L^2(\RR^n_+)\,,
\]
and
\[
	L^{2}(\RR^n) \xrightarrow{\td{G}_\phi -\td{G}_\phi^c} H^2(\RR^n) \xrightarrow{(1-\td{\rho}(hD'))J} H^1(\RR^n) \xrightarrow{(\id-J^+J)Z^{-1}} W^{2,2}(\RR^n_+) \hookrightarrow_{h^{-1}} W^{1,2}(\RR^n_+) \xrightarrow[h]{[h^2\td{\Delta}_\phi,\td{\chi}]} L^2(\RR^n_+)\,,
\]
which means we have taken care of the third term in \cref{splitPfagain} if we consider these maps to be on $\td{\Omega}$ by pre- and post-composing with $\itdo$.

In dealing with the fourth term in \cref{splitPfagain} we can carry on as above: by 
\cref{splitupJintoF}, we have $J = h\del_{x_n} + F_+$ so that
\[
	JZ^{-1}(1-\td{\rho}(hD'))\mathcal{E} = Z^{-1}(1-\td{\rho}(hD'))h\del_{x_n}\mathcal{E} + F_+Z^{-1}(1-\td{\rho}(hD'))\mathcal{E} \colon L^2(\RR^{n-1}) \to L^2(\RR^n)\,,
\]
and then
\[
	L^2(\RR^{n-1}) \xrightarrow{JZ^{-1}(1-\td{\rho}(hD'))\mathcal{E}} L^2(\RR^n) \xrightarrow[h^2]{\itdo[h^2\td{\Delta}_\phi,\td{\chi}]J^{-1}1_{\RR^n_-}} L^2(\td{\Omega})\,,
\]
which yields that $\itdo [h^2\td{\Delta}_\phi,\td{\chi}]B_c(1+\abs{K}^2)^{1/2}f \in hL^2(\td{\Omega})$, which we consider part of $R_{f,\Gamma}$.

Finally, we turn to the first term of \cref{splitPfagain}, which is to say that we want to consider the expression $1_{\td{\Omega}}[h^2\td{\Delta}_\phi,\td{\chi}](\td{G}_\phi - P_l)\itdo\td{v}$. The identical estimate is handled in \cite[Lemm.~6.2]{CT20}, but for completeness, we perform the estimates here as well:
\begin{equation}\label{eq:lastpart}
	\td{G}_\phi - P_l = \td{\rho}(hD')\td{G}_\phi + (1-\td{\rho}(hD'))(id -J^+J)\td{G}_\phi\,,
\end{equation}
where the last term of \cref{eq:lastpart} can be dealt with using \cref{chih2} exactly as above. 

For the term $\td{\rho}(hD')\td{G}_\phi$ in \cref{eq:lastpart}, we use \cite[Lemm.~4.4]{CT20} so that $\td{\rho}(hD')\tilde{G}_\phi^c = (\td{\rho}(hD')\tilde{G}_\phi^c)_1 + (\td{\rho}(hD')\tilde{G}_\phi^c)_2$ where
\[
	 (\td{\rho}(hD')\tilde{G}_\phi^c)_1 \in \mathrm{Op}_h(S^{-\infty}_1(\RR^n)) \,, \an (\td{\rho}(hD')\tilde{G}_\phi^c)_2 = h \mathrm{Op}_h(S^{-\infty}_1(\RR^n))\tilde{G}_\phi\,.
\]
The first term maps $L^2(\RR^n) \to H^k(\RR^n)$ and $L^{p'}(\RR^n) \to W^{k,p'}(\RR^n) \hookrightarrow_{h^{-1}} H^{k-1}(\RR^n)$, whereas the second term maps $L^2(\RR^n) \to H^k(\RR^n)$ and $L^{p'}(\RR^n) \to_{h^{-1}} W^{k,p}(\RR^n) \hookrightarrow H^{k}(\RR^n)$ for all $k\in\mathbb{N}$. Applying the commutator $[h^2\td{\Delta}_\phi,\td{\chi}]$ provides an $h$ and removes one order of regularity.
\end{proof}

\section{CGO Solutions}\label{sec:CGO}

Take $\omega\in S^{n-1}$ and $\Xi = \del\Omega\setminus B$ as they are defined in \cref{sec:intro}. The goal of this section is to prove \cref{cgosolsprop} for which we will need a preperatory lemma.
The only differnce between the following proof and its version in \cite{CT20} lies in the boundary conditions, and thus the proof will be kept short.

\begin{lemma}[{\cite[Prop.~7.1]{CT20}}]\label{getr0lemm}
	For any $f\in L^2(\Xi)$ with $\norm{f}_{L^2} \leq Ch$, any $L \in L^2(\Omega)$ so that $\norm{L}_{L^2} \leq Ch^2$, any $q\in L^{n/2}(\Omega)$ and any $a = a_h \in L^\infty(\Omega)$ with $\norm{a_h}_{L^\infty} \leq C$ and any $\delta>0$ there is a solution $r \in H^1(\Omega)$ of 
	\[
		h^2(\Delta_\phi+q)r = h^2qa+L\,,\an \tau_\Xi(h\del_\nu+\del_\nu\phi)r=-f\,,
	\]
	and $\lim_{h\to 0}\norm{r}_{L^2} \leq \delta$ and $\norm{r}_{L^p} \leq C$ for $h$ small enough.
\end{lemma}
\begin{proof}
	In this proof $C>0$ is a generic constant independent of $h$ and may change from one equation to the next. Let $\varepsilon > 0$ be arbitrary and write 
	\begin{equation}\label{eq:splitofq}
		\sqrt{\abs{q}} = \sqrt{\abs{q}}^\flat + \sqrt{\abs{q}}^\sharp\quad\text{where}\quad\norm{\sqrt{\abs{q}}^\flat}_{L^\infty} \leq C_\varepsilon\,,\an\norm{\sqrt{\abs{q}}^\sharp}_{L^{n/2}} \leq \varepsilon\,.
\end{equation}
Furthermore, let $\theta \colon \Omega \to \RR$ so that $e^{i\theta}\abs{q}=q$. 

We choose the ansatz 
\begin{equation}
	r = G_{\Xi}(\sqrt{\abs{q}}w+L,f)\quad\text{where}\quad (1+h^2e^{i\theta}\sqrt{\abs{q}}G_{\Xi}(\sqrt{\abs{q}}\cdot,0))w = h^2e^{i\theta}\sqrt{\abs{q}}(a-G_{\Xi}(L,f))\,.\label{eq:defofrandw}
\end{equation}
Assuming such a $w$ exists, we see from \cref{thm:G} that
\begin{equation}
	h^2(\Delta_\phi +q)r = \sqrt{\abs{q}}\left(1+h^2e^{i\theta}\sqrt{\abs{q}}G_\Xi(\sqrt{\abs{q}}\cdot,0)\right)w+h^2qG_\Xi(L,f)+L = h^2qa+L\,,
\end{equation}
and $\tau_\Xi(h\del_\nu+\del_\nu\phi)r = -f$.

We first construct $w$ in \cref{eq:defofrandw} and show that $\norm{w}_{L^2} \leq Ch^2$. This begins by proving that the norm of $h^2e^{i\theta}\sqrt{\abs{q}}G_{\Xi}(\sqrt{\abs{q}}\cdot,0) \colon L^2(\Omega) \to L^2(\Omega)$ can be made arbitrarily small. Splitting $\sqrt{\abs{q}}$ as in \cref{eq:splitofq} and using \cref{eq:hG1,eq:hG2}, we have
\begin{align*}
	\sqrt{\abs{q}}^\flat G_{\Xi}(\sqrt{\abs{q}}^\flat\cdot,0) &\colon L^2(\Omega)\xrightarrow{\sqrt{\abs{q}}^\flat} L^2(\Omega) \xrightarrow[h^{-1}]{G_\Xi(\cdot,0)} L^{2}(\Omega) \xrightarrow{\sqrt{\abs{q}}^\flat} L^2(\Omega)\\
	\sqrt{\abs{q}}^\sharp G_{\Xi}(\sqrt{\abs{q}}^\flat\cdot,0) &\colon L^2(\Omega)\xrightarrow{\sqrt{\abs{q}}^\flat} L^2(\Omega) \hookrightarrow L^{p'}(\Omega) \xrightarrow[h^{-2}]{G_\Xi(\cdot,0)} L^{p}(\Omega) \xrightarrow[\varepsilon]{\sqrt{\abs{q}}^\sharp} L^2(\Omega)\\
	\sqrt{\abs{q}}^\flat G_{\Xi}(\sqrt{\abs{q}}^\sharp\cdot,0) &\colon L^2(\Omega)\xrightarrow[\varepsilon]{\sqrt{\abs{q}}^\sharp} L^{p'}(\Omega) \xrightarrow[h^{-2}]{G_\Xi(\cdot,0)} L^{p}(\Omega) \xrightarrow{\sqrt{\abs{q}}^\flat} L^p(\Omega) \hookrightarrow L^2(\Omega)\\
	\sqrt{\abs{q}}^\sharp G_{\Xi}(\sqrt{\abs{q}}^\sharp\cdot,0) &\colon L^2(\Omega)\xrightarrow[\varepsilon]{\sqrt{\abs{q}}^\sharp} L^{p'}(\Omega) \xrightarrow[h^{-2}]{G_\Xi(\cdot,0)} L^{p}(\Omega) \xrightarrow[\varepsilon]{\sqrt{\abs{q}}^\sharp} L^2(\Omega)\,.
\end{align*}
We see that each of the four terms is bounded either by $Ch^{-1}$ or $Ch^{-2}\varepsilon$, which shows that we may choose $\varepsilon>0$ and $h>0$ small enough so that $\norm{h^2e^{i\theta}\sqrt{\abs{q}}G_{\Xi}(\sqrt{\abs{q}}\cdot,0)}_{L^2\to L^2} < 1$. Thus, an application of Neumann inversion proves that there exists a $w \in L^2(\Omega)$ solving the right equation in \cref{eq:defofrandw}. Due to the fact that $\norm{L}_{L^{p'}} \leq \norm{L}_{L^2} \leq Ch^2$ and $\norm{f}_{L^2}\leq Ch$, we may use \cref{eq:hG1} to see that $\norm{G_{\Xi}(L,f)}_{L^p}\leq C$. Now $\sqrt{\abs{q}} \in L^{n}$ and $\norm{a}_{L^\infty}<C$ so that $\norm{h^2e^{i\theta}\sqrt{\abs{q}}(a-G_{\Xi}(L,f))}_{L^2} \leq Ch^2$. Therefore, this solution $w$ admits $\norm{w}_{L^2} \leq Ch^2$.

The only thing left to prove is that $r$ has the properties we stated. By \cref{eq:hG1,eq:hG2}, the facts that $\norm{L}_{L^2}\leq Ch^2$, $\norm{f}_{L^2}\leq Ch$ and $\norm{w}_{L^2}\leq Ch^2$, we have
\begin{align*}
	\norm{r}_{L^p} &\leq \norm{G_\Xi(L,f)}_{L^p} + \norm{G_\Xi(\sqrt{\abs{q}}^\flat w,0)}_{L^p} + \norm{G_{\Xi}(\sqrt{\abs{q}}^\sharp w,0)}_{L^p} \leq C+C+C_\varepsilon\,.
	\shortintertext{Furthermore,}
	\norm{r}_{L^2} &\leq \norm{G_\Xi(L,f)}_{L^2} + \norm{G_\Xi(\sqrt{\abs{q}}^\flat w,0)}_{L^2} + \norm{G_{\Xi}(\sqrt{\abs{q}}^\sharp w,0)}_{L^2} \leq Ch+C_\varepsilon h+C\varepsilon\,,
\end{align*}
which can be made arbitrarily small (in particular smaller than $\delta>0$) by first choosing $\varepsilon>0$ appropriately and then $h>0$ small.
\end{proof}

\subsection{Proof of Proposition~\ref{cgosolsprop}}
The next proof follows that of \cite[Prop.~7.2]{CT20} and \cite[Prop.~1.4]{chungpartial}, but we restate it here for convenience of the reader.

\begin{proof}[Proof of \cref{cgosolsprop}]
Take $\Xi = \del\Omega\setminus F$. Let $\phi(y),\psi(y)$ be linear functions with $\nabla (\phi+i\psi) \cdot \nabla (\phi+i\psi) = 0$ and $\nabla \phi \cdot \nu(y) \leq -\varepsilon < 0$ for some $\varepsilon>0$ and all $y\in \Xi$. If instead $\Xi = \del\Omega\setminus B$, the sign of $\varepsilon$ would change, but the remainder of the proof would stay the same.

Take $\xi \in \RR^n$ orthogonal to $\nabla \phi$ and $\nabla \psi$, and put $\psi_h(y)=(\xi-\omega_h')\cdot y$, where
\begin{equation}\label{eq:defofomega}
	\omega_h' = \frac{1-\sqrt{1-h^2\abs{\xi}^2\abs{\nabla\psi}^{-2}}}{h}\nabla\psi\,,
\end{equation}
which by the mean value theorem has length $\mathcal{O}(h)$.

By direct calculation we see that 
\begin{equation}\label{eq:fallsaway}
	\nabla (\phi+i\psi + ih\psi_h)\cdot\nabla (\phi +i\psi+ih\psi_h) = 0\,,\quad\text{and so}\quad
h^2\Delta_\phi e^{\frac{i\psi+ih\psi_h}{h}} = 0\,.
\end{equation}

The proof of the following lemma is delegated to the appendix.
\begin{lemma}\label{lem:constructlb}
	\qquad
	\begin{enumerate} 
		\item\label{constructl} There is a function $t\in C^\infty(\RR^n)$ so that 
	\begin{equation}\label{eq:constructlprops}
	\nabla t \cdot \nabla t = d(y,\Xi)^\infty\,,\an \tau_\Xi t = \tau_\Xi(\phi + i\psi)\,,\an \tau_\Xi \del_\nu t = -\tau_\Xi \partial_\nu (\phi+i\psi)\,.
	\end{equation}
	In addition, 
	\begin{equation}\label{eq:distfort}
		-(\mathrm{Re}t-\phi) \simeq d(y,\Xi)
	\end{equation}
	for all $y$ in some neighborhood $U\subset \RR^n$ of $\Xi$.
	\item \label{constructb}
		There is a function $b\in C^\infty(\RR^n)$ supported in $U$ so that
		\begin{equation}\label{eq:bdefiningeqn}
			e^{-t/h}h^2\Delta(e^{t/h}e^{i\psi_h}b) = d(y,\Xi)^\infty + O_{L^\infty}(h^2)\,, \an \tau_\Xi b = 1\,,
		\end{equation}
		and $\tau_\Xi\del_\nu b = \mathcal{O}_{L^\infty}(1)$.
	\end{enumerate}
\end{lemma}
The first part of \cref{eq:bdefiningeqn} is equivalent to
\begin{equation}\label{eq:someintermediaryb}
	h^2\Delta e^{t/h}e^{i\psi_h}b = e^{\phi/h} e^{(t-\phi)/h}(d(y,\Xi)^\infty+\mathcal{O}_{L^\infty}(h^2))\,.
\end{equation}
Due to \cref{eq:distfort} and the support property of $b$, if $d(y,\Xi) < h^{1/2}$ then the right hand side of \cref{eq:someintermediaryb} is bounded by $e^{\phi/h}\mathcal{O}_{L^\infty}(h^2)$, whereas if $d(y,\Xi) \geq h^{1/2}$, then $e^{\frac{t-\phi}{h}} \leq h^2$, which results in the same bound. Thus we may rephrase \cref{eq:someintermediaryb} into
\begin{equation}\label{eq:defofL}
	h^2\Delta_\phi e^{(t-\phi)/h}e^{i\psi_h}b \eqqcolon L \in \mathcal{O}_{L^\infty}(h^2)\,.
\end{equation}

Put
\begin{equation}\label{eq:defofah}
	a_h \coloneqq e^{\frac{t-\phi-i\psi}{h}}b \in C^\infty(\RR^n)\,, \quad\text{so that}\quad \tau_\Xi a_h = \tau_\Xi b = 1\,,
\end{equation}
and because by \cref{eq:fallsaway}, $h^2\Delta_\phi e^{\frac{i\psi+ih\psi_h}{h}}1 = 0$,
\begin{equation}\label{eq:Lpart}
	h^2\Delta_\phi e^{\frac{i\psi+ih\psi_h}{h}}(1+a_h) = h^2\Delta_\phi e^{\frac{i\psi+ih\psi_h}{h}}a_h = h^2\Delta_\phi e^{\frac{t-\phi+ih\psi_h}{h}}b = L\,.
\end{equation}

Furthermore, we may calculate
\begin{align}
	\tau_\Xi e^{-\phi/h}h\del_\nu e^{\frac{\phi+i\psi+ih\psi_h}{h}}(1+a_h) 
	&= \tau_\Xi e^{\frac{i\psi+ih\psi_h}{h}}(2ih\del_\nu\psi_h+h\del_\nu b) \eqqcolon f\,,\label{eq:fpart}
\end{align}
for some $f\in L^\infty(\Xi)\subset L^2(\Xi)$ with norm $\mathcal{O}_{L^\infty}(h)$, since $\del_\nu\psi_h = \mathcal{O}_{L^\infty}(1)$ and $\del_\nu b = \mathcal{O}_{L^\infty}(1)$.

Note again that on the support of $b$ and when $d(y,\Xi) \geq h^{1/2}$, then $\abs{a_h} \leq \mathcal{O}(e^{\frac{t-\phi}{h}}) \leq \mathcal{O}(h^2)$, which vanishes as $h\to 0$, and $\lim_{h\to 0} \{y \in \Omega\colon d(y,\Xi) \geq h^{1/2}\} = \Omega$. This is to say that $\lim_{h\to 0}a_h = 0$ pointwise in $\Omega$.

Using \cref{getr0lemm}, find the $r_0\in H^1(\Omega)$ so that
\begin{equation}\label{eq:foundr0}
	h^2(\Delta_\phi + q)r_0 = -h^2qe^{\frac{i\psi+ih\psi_h}{h}}a_h - L\,,\an \tau_\Xi(h\del_\nu + \del_\nu\phi) r_0 = -f\,,
\end{equation}
and put 
\begin{equation}\label{eq:defofu}
r_1 \coloneqq e^{\frac{-i\psi - ih\psi_h}{h}}r_0 \in H^1(\Omega)\,,\an u \coloneqq e^{\frac{\phi+i\psi+ih\psi_h}{h}}(1+a_h+r_1) \in H^1(\Omega)\,.
\end{equation}
Calculating with \cref{eq:foundr0,eq:Lpart,eq:defofu}, we see that 
\begin{align*}
	h^2(\Delta + q)u &= (h^2e^{\phi/h}\Delta_\phi e^{-\phi/h} + h^2q)u = 0\,,
\end{align*}
and with \cref{eq:fpart},
\begin{align*}
	\tau_\Xi e^{-\phi/h}h\del_\nu u &= \tau_\Xi e^{-\phi/h}h\del_\nu e^{\frac{\phi+i\psi+ih\psi_h}{h}}(1+a_h) + \tau_\Xi e^{-\phi/h}h\del_\nu e^{\phi/h}r_0 \\
	&= f + \tau_\Xi(h\del_\nu + \del_\nu\phi)r_0 = 0\,.
\end{align*}

Taking $\phi(y) = \omega \cdot y = e_n \cdot y = y_n$ and $\psi(y) = \omega'\cdot y$ the proof is complete.
\end{proof}

\section{Uniqueness of Coefficients}\label{sec:uniq}

Let $\omega$ be as in \cref{sec:intro}, and let $\omega'\in \mathbb{S}^{n}$ with $\omega'\perp\omega$. This proof closely follows \cite[Thm.~1.1]{CT20}, but is included for convenience.

\begin{proof}[Proof of \cref{thm:uniqueness}]
	While \cref{cgosolsprop} is stated only for $\Xi = \del\Omega\setminus B$, exchanging $\omega$ with $-\omega$, and thus also $\phi$ with $-\phi$, it is not difficult to see that all results leading up to and including \cref{cgosolsprop} can be stated analogously where $\Xi = \del\Omega\setminus F$.

	Thus, we may construct
\begin{equation}\label{eq:defofupm}
	u_{\pm} = e^{\frac{\pm\omega \pm i\omega'+ih\psi_h^{\pm}}{h}}(1+a_h^\pm + r_\pm) \in H^1(\Omega)
\end{equation}
so that 
\begin{align}\label{eqgotuplusuminus}
	(\Delta+q_\pm)u_\pm &= 0 \text{ in } \Omega\,,\notag\\
	\del_\nu u_+ &= 0 \text{ on } \partial\Omega\setminus F \eqqcolon \Xi_+\,, \\
	\del_\nu u_- &= 0 \text{ on } \partial\Omega\setminus B \eqqcolon \Xi_-\,.\notag
\end{align}
and where $\psi_h^\pm(y) = (\xi \mp \omega_h')\cdot y$, with $\omega_h'$ from \cref{eq:defofomega}.

We follow an argument in \cite[Lem.~3.1]{recentProg}.

By the assumptions on this theorem we have that there exists some $\td{u}_- \in H^\sharp_\Delta(\Omega)$ with $(\Delta+q_-)\td{u}_-=0$ in $\Omega$, $\supp(\del_\nu \td{u}_-) \subset F$, and 
\[
	(u_+\vert_B, \del_\nu u_+\vert_F) = (\td{u}_-\vert_B, \del_\nu \td{u}_-\vert_F)
\]
which, by \cref{eqgotuplusuminus}, means that 
\begin{equation}\label{eq:assumptionsgetB}
	u_+|_B = \tilde{u}_-|_B\,,\an \del_\nu u_+\vert_{\del\Omega} = \del_\nu \td{u}_-\vert_{\del\Omega}\,.
\end{equation}

Before we compute, let us state the following result, proved in the appendix.
\begin{lemma}\label{lemm:hsharpprops}
	The map $C^\infty(\bar\Omega) \ni u \mapsto \del_\nu u\vert_{\partial\Omega} \in H^{-1/2}(\del\Omega)$ has an extension that is continuous on $H^\sharp_\Delta(\Omega)$.
\end{lemma}

Note that by Sobolev embedding $u_\pm,\td{u}_- \in L^p(\Omega)$ and so $-\Delta u_\pm = q_\pm u_\pm \in L^{p'}(\Omega)$ and $q_\pm u_+u_- \in L^1(\Omega)$. Using \cref{eqgotuplusuminus}, we calculate that
\begin{align}\label{eq:pairingsvanish}
	\int_\Omega (q_+-q_-)u_+u_- \dd y &= \int_\Omega q_+u_+u_- - u_+q_-u_- \dd y= -\int_\Omega \Delta u_+ u_- - u_+\Delta u_-\dd y \notag\\
	\shortintertext{and due to the fact that $\Delta \tilde{u}_- = -q_- \tilde{u}_-$,}
	&= -\int_\Omega \Delta(u_+-\tilde{u}_-) u_- - (u_+-\tilde{u}_-)\Delta u_- \dd y \notag\\
	&= -\int_{\partial\Omega}\del_\nu (u_+-\tilde{u}_-)  u_- - (u_+-\tilde{u}_-) \del_\nu u_- \dd S\,,
\end{align}
where the penultimate line makes sense because $\Delta(u_+-\tilde{u}_-),\Delta u_- \in L^{p'}(\Omega)$, and $u_-, (u_+-\tilde{u}_-) \in L^p(\Omega)$, and the final line is meant as pairings between functions in $H^{1/2}(\del\Omega)$ and $H^{-1/2}(\del\Omega)$. 

Now \cref{eq:pairingsvanish} vanishes because by \cref{eqgotuplusuminus,eq:assumptionsgetB}, $\del_\nu (u_+ - \td{u}_-)=0$ on $\partial\Omega$, and $(u_+-\td{u}_-) = 0$ on $B$ and $\del_\nu u_- = 0$ on $\partial\Omega\setminus B$.

From here on, the proof follows \cite[Thm~1.1]{CT20}.

Put $q = q_+-q_-$, and plugging in the expressions for $u_+$ and $u_-$ from \cref{eq:defofupm} into the outcome of \cref{eq:pairingsvanish}, we have that 
\begin{equation}\label{eq:zeroallofthese}
	0=\int_\Omega e^{2i\xi\cdot y}q(1+a_h^-a_h^++a_h^-+a_h^++a_h^-r_++a_h^+r_-+r_-+r_++r_+r_-)\dd y\,.
\end{equation}

Note that by the construction from \cref{cgosolsprop}, $\lim_{h\to 0} a_h^\pm (y) =0$ pointwise for all $y\in\Omega$, and $\norm{a_h^\pm}_{L^\infty} \leq C$, and $q\in L^{n/2}(\Omega) \subset L^1(\Omega)$, so that by the dominated convergence theorem,
\[
	\lim_{h\to 0} \int_\Omega e^{2i\xi\cdot y} q (a_h^+a_h^-+a_h^-+a_h^+) \dd y = \int_\Omega e^{2i\xi\cdot y} q \lim_{h\to 0}(a_h^+a_h^- + a_h^- + a_h^+) \dd y =0\,, 
\]
which takes care of the second, third and fourth term in \cref{eq:zeroallofthese}.

Since for small $h$, $\norm{r_\pm}_{L^p} \leq C$ and $q\in L^{n/2}(\Omega)$, we have that $qr_\pm \in L^{p'}(\Omega) \subset L^1(\Omega)$, we may use dominated convergence again to see that 
\[
	\lim_{h\to 0} \int_\Omega e^{2i\xi\cdot y} qr_\pm a_h^\mp \dd y = 0\,,
\]
which deals with the fifth and sixth term in \cref{eq:zeroallofthese}.

For the remaining terms, decompose $q = q^\flat+q^\sharp$ with $q^\flat \in L^\infty$ and $\norm{q^\sharp}_{L^{n/2}} < \varepsilon$. Remember also that $r_\pm$ can be chosen so that $\lim_{h\to 0} \norm{r_\pm}_{L^2}\leq \varepsilon$ for this arbitrary choice of $\varepsilon$, which means that
\begin{equation}\label{eq:smth}
	\lim_{h\to 0} \abs{\int_\Omega e^{2i\xi\cdot y} qr_\pm \dd y} \leq\lim_{h\to 0}  \left(\norm{q^\flat}_{L^\infty}\norm{r_\pm}_{L^2} + \norm{q^\sharp}_{L^{n/2}}\norm{r_\pm}_{L^p}\right) \leq C\varepsilon
\end{equation}
for every $\varepsilon >0$, wherefore the LHS of \cref{eq:smth} vanishes. Thus the seventh and eigth term in \cref{eq:zeroallofthese} vanish.

Similarly, 
\[
	\lim_{h\to 0}\abs{\int_\Omega e^{2i\xi \cdot y} qr_+r_- \dd y} \leq\lim_{h\to 0}  \left( \norm{q^\flat}_{L^\infty}\norm{r_+}_{L^2}\norm{r_-}_{L^2} + \norm{q^\sharp}_{L^{n/2}}\norm{r_+}_{L^p}\norm{r_-}_{L^p}\right) \leq C\varepsilon\,,
\]
which then must also be zero, taking care of the final term in \cref{eq:zeroallofthese}.

Since the only term surviving in \cref{eq:zeroallofthese} is the first, we may conclude that 
\[
	\mathcal{F}(q)(-2\xi) = \int_\Omega e^{2i\xi \cdot y} q \dd y = 0\,.
\]

This calculation is true for all $\xi$ with $\xi \perp \omega$. If we perturb $\omega$ by small amounts, $\Xi_+$ will still be a subset of $\{y\in \del\Omega\colon \omega\cdot \nu(y) < 0\}$. This means perturbing $\omega$ slightly, we get an open set of $\xi$ in which $\mathcal{F}(q)(-2\xi) = 0$, which by the analyticity of $\mathcal{F}$ means that $q = q_+-q_- = 0$.
\end{proof}

As a direct consequence we are able to prove the uniqueness of conductivies as well. 
\begin{proof}[Proof of \cref{cor:calderon}]
	This proof follows an argument in \cite{zbMATH04015323}. For $q_\pm \coloneqq - \gamma_\pm^{-1/2}\Delta(\gamma^{1/2}_\pm) \in L^{n/2}(\Omega)$ let $u_\pm$ satisfy 
	\[
		(\Delta+q_\pm)u_\pm=0 \text{ in } \Omega\,,\, u_\pm\in H_{\Delta}^\sharp(\Omega)\,,\, \supp(\del_\nu u_\pm\vert_{\partial\Omega}) \subset F\,.
	\]
	Defining $w_\pm \coloneqq \gamma^{-1/2}_\pm u_\pm$, a calculation shows 
	\[
		\mathrm{div}(\gamma_\pm\nabla w_\pm)=0 \text{ in } \Omega\,,\, w_\pm\in H^1(\Omega)\,,\, \supp(\del_\nu w_\pm\vert_{\partial\Omega}) \subset F\,.
	\]
	Due to $\del_\nu \gamma_+\vert_{\del\Omega} = \del_\nu \gamma_-\vert_{\del\Omega} = 0$, we have
	\[
		(u_\pm\vert_B, \del_\nu u_\pm\vert_{\del\Omega}) = (\gamma_\pm^{1/2}w_\pm\vert_B, \gamma^{1/2}_\pm\del_\nu w_\pm\vert_{\del\Omega})\,,
	\]
	and defining the map $M_{\gamma_\pm} \colon (v_1,v_2) \mapsto (\gamma_\pm^{1/2} v_1, \gamma_\pm^{-1/2} v_2)$, we see
	\[
		C_{q_\pm}^{B,F} \subset M_{\gamma_\pm} C_{\gamma_\pm}^{B,F}\,.
	\]
	Together with a Sobolev embedding, a similar argument will lead to the reverse inclusion, so that given the assumptions $C_{\gamma_+}^{B,F} = C_{\gamma_-}^{B,F}$ and $\gamma_+\vert_{\del\Omega} = \gamma_-\vert_{\del\Omega}$, we may conclude 
	\[
		C_{q_+}^{B,F} = M_{\gamma_+}C_{\gamma_+}^{B,F} = M_{\gamma_-}C_{\gamma_-}^{B,F} = C_{q_-}^{B,F}\,,
	\]
	and an application of \cref{thm:uniqueness} gives $q_+ = q_-$. 

	Writing out the definitions of $q_+, q_-$, we thus have
	\[
		0 = 2\gamma_+^{-1/2}\Delta(\gamma_+^{1/2})-2\gamma_-^{-1/2}\Delta(\gamma_-^{1/2})= \Delta(\log \gamma_+ - \log \gamma_-) + \frac{1}{2} \nabla(\log \gamma_+ + \log \gamma_-) \cdot \nabla (\log \gamma_+ - \log \gamma_-)\,.
	\]
	Putting $\gamma \coloneqq \log\gamma_+ - \log\gamma_- \in W^{2,n/2}(\Omega) \subset W^{1,2}(\Omega)$, we have $\gamma\vert_{\del\Omega} = 0$ and we may apply the weak maximum principle (see \cite[Cor.~8.2]{zbMATH03561752}) to conclude that $\gamma_+ = \gamma_-$.
\end{proof}

\section{Appendix}
\appendix

\renewcommand{\thesubsection}{\Alph{subsection}}
\renewcommand{\theequation}{\thesubsection.\arabic{equation}}
\renewcommand{\thelemma}{\thesubsection.\arabic{lemma}}

\subsection{Borel Lemma Constructions}

\begin{proof}[Proof of \cref{constructl} in \cref{lem:constructlb}]
	We follow \cite[p.~16]{hoermander1}, and \cite[Prop.~9.2]{chungmagnetic}. Let $\varepsilon_j, j \in\mathbb{N}_0$ be a positive increasing sequence with $\varepsilon_0 > 1$ which we will define later. Take $\eta \in C_c^\infty(\RR)$ so that $\eta$ is identically $1$ on $(-1/4,1/4)$ and identically $0$ outside of $(-1/2,1/2)$. Let $f_j \in C^\infty(\RR^{n-1})$ be a sequence of functions with bounded derivative of every order $\geq 0$ that we have yet to define. Change coordinates in such a way that $z'$ is the tangential part to $\Xi$ and $z_n$ is the normal part. That is to say we work in coordinates $(z',z_n)$ where $\del_\nu = \del_{z_n}$ in the new coordinates. 
	
	Consider $\eta_j(z_n) \coloneqq \eta(z_n\varepsilon_j) \frac{z_n^j}{j!}$. Let $\beta \in \mathbb{N}_0^{n-1}, \alpha \in \mathbb{N}_0$ be multi-indices with $\alpha + \abs{\beta} < j$, then 
	\begin{align}
		\del^\beta_{z'}\del_{z_n}^\alpha \eta_j(z_n)f_j({z'}) &= \sum_{\gamma \leq \alpha}{\alpha \choose \gamma} \varepsilon_j^\gamma \eta^{(\gamma)}(z_n\varepsilon_j) \frac{z_n^{j-\alpha+\gamma}}{(j-\alpha+\gamma)!}\del_{z'}^\beta f_j({z'}) \notag \\
		&=  \sum_{\gamma \leq \alpha}{\alpha \choose \gamma} \eta^{(\gamma)}(z_n\varepsilon_j) \frac{(\varepsilon_j z_n)^{j-\alpha+\gamma}}{(j-\alpha+\gamma)!}\varepsilon_j^{\alpha-j}\del_{z'}^\beta f_j({z'})\,. \label{estimatel}
	\end{align}
	Note that the support properties of $\eta$ imply that $\abs{\eta^{(\gamma)}(z_n\varepsilon_j)(z_n\varepsilon_j)^{j-\alpha+\gamma}} < 2^{-(j-\alpha+\gamma)}$, and if we choose $\varepsilon_j$ (depending on $f_j$) in such a way that 
	\[
	\abs{\sum_{\gamma\leq \alpha}{\alpha\choose\gamma} \varepsilon_j^{\alpha-j}\del_{z'}^\beta f_j({z'})} \leq 2^{-\alpha+\gamma}\,,
	\]
	then $\abs{\del_{z'}^\beta\del_{z_n}^\alpha \eta_j({z_n})f_j({z'})} \leq 2^{-j}$ for all $\alpha + \abs{\beta} < j$.
	
	This means that if we define
	\[
	t(z',{z_n}) \coloneqq \sum_{j\geq 0}\eta_j({z_n})f_j({z'})\,,
	\]
	then 
	\[
	\abs{\del_{z'}^\beta\del_{z_n}^\alpha l({z'},{z_n})} \leq \sum_{j\leq \alpha +\abs{\beta}}\abs{\del_{z'}^\beta\del_{z_n}^\alpha\eta_j({z_n})f_j({z'})} + \sum_{\alpha + \abs{\beta} < j}\abs{\del_{z'}^\beta\del_{z_n}^\alpha\eta_j({z_n})f_j({z'})}
	\]
	where the first is a finite sum and the second is bounded by $2$.
	
	In conclusion this means that the series defining $t$ is absolutely convergent for every derivative. 
	
	Furthermore, from \cref{estimatel} we can calculate to see that 
	\[
	\del_{z_n}^j \eta_j(0) f_j({z'}) = f_j({z'}) \qquad \forall j\in\mathbb{N}\,.
	\]
	The conditions we desire for $t$ imply we take $f_0({z'}) = \tau_\Xi(\phi+i\psi)$, and $f_1({z'}) =-\tau_\Xi \del_{z_n}(\phi+i\psi)$. 
	Manually calculating, we have that 
	\begin{equation}\label{eq:nablatnablat}
	\begin{aligned}		
		\nabla t \cdot \nabla t &= \sum_{w\geq 0} \sum_{j+k=w}\frac{1}{j!k!}\left[\eta({z_n}\varepsilon_j)\eta(\varepsilon_k)\nabla_{z'} f_j({z'})\nabla_{z'} f_k({z'})\right. \\
		&+ \left.\left(\varepsilon_j\eta'({z_n}\varepsilon_j)f_j({z'})+\eta({z_n}\varepsilon_{j+1})f_{j+1}({z'})\right)\left(\varepsilon_k\eta'({z_n}\varepsilon_k)f_k({z'})+\eta({z_n}\varepsilon_{k+1})f_{k+1}({z'})\right)\right] 
	\end{aligned}
\end{equation}
	which we want to be equal to $0$ approximately.
	
	For $w = 0$ and $\abs{{z_n}} \leq (4\varepsilon_1)^{-1}$ the part of \cref{eq:nablatnablat} pertaining to $w=0$ is equal to
	\[
	\nabla_{z'} f_0 \nabla_{z'} f_0 + f_1({z'})f_1({z'}) = \nabla(\phi+i\psi)\cdot\nabla(\phi+i\psi) = 0\,.
	\]
	For $w\geq 1$ and $\abs{{z_n}} \leq (4\varepsilon_{w+1})^{-1}$ and $j+k = w$, we want to solve for 
	\[
	0 = \sum_{j+k=w}\frac{1}{j!k!}\left(\nabla_{z'} f_j({z'})\nabla_{z'} f_k({z'}) + f_{j+1}({z'})f_{k+1}({z'})\right) = \frac{1}{w!}f_1({z'})f_w({z'}) + R_{j,k < w}
	\]
	where by $R_{j,k< w}$ we denote some term that only depends on $f_j$ with $j < w$. Thus if the $f_j$ have been defined recursively we must only solve for $f_w$, which we may do because the coefficient in front of it, being $f_1({z'})$, is bounded away from $0$. 
	
	This means that for any $w\in\mathbb{N}$ and $\abs{{z_n}} \leq (4\varepsilon_w)^{-1}$ we have that $t = \mathcal{O}(z_n^{w+1})$. Furthermore, \cref{eq:distfort} is implied by \cref{eq:constructlprops} which concludes the proof.
\end{proof}

\begin{proof}[Proof of \cref{constructb} in \cref{lem:constructlb}]
	By direct computation we have that 
	\begin{align*}
		e^{-t/h}h^2\Delta e^{\frac{t+ih\psi_h}{h}}b &= -e^{i\psi_h} \left(\nabla (t +ih\psi_h)\nabla(t+i\psi_h)b - h(\Delta (t+ih\psi_h))b+2 h\nabla(t+ih\psi_h)\nabla b - h^2\Delta b\right) \\
		\shortintertext{where $\abs{\omega'_h} \leq Ch$ (see \cref{eq:defofomega}), and $\nabla t \cdot \nabla t = d(y,\Xi)^\infty$ by \cref{eq:constructlprops},  which means that}
		e^{-t/h}h^2\Delta e^{\frac{t+ih\psi_h}{h}}b &= -e^{i\psi_h}\left(hb (\Delta t + 2i \xi\cdot \nabla t) + 2h\nabla b\cdot \nabla t + \mathcal{O}(h^2) + d(y,\Xi)^\infty\right)
	\end{align*}
	and so rephrasing \cref{eq:bdefiningeqn}, we want to approximately solve 
	\begin{equation}\label{eq:wantforb}
	b(\Delta t + 2i\xi\cdot \nabla t) + 2\nabla b\cdot \nabla t = 0\,, \an \tau_\Xi b = 1\,.
	\end{equation}
	
	We will use the same recipe as in the above proof. If, after change of variables as above, we take $b({z'},{z_n}) = \sum_{j\geq 0}\eta_j({z_n})f_j({z'})$, we will want $f_0 = 1$ by the boundary condition we impose on $b$ and the first part of \cref{eq:wantforb} is rephrased as
	{\small\[
		0 = \sum_{j\geq 0} \frac{z_n^j}{j!}\left[\eta({z_n}\varepsilon_j)(\Delta t+2i\xi\cdot\nabla t)f_j({z'}) + (\varepsilon_j\eta'({z_n}\varepsilon_j)f_j({z'})+\eta({z_n}\varepsilon_{j+1})f_{j+1}({z'}))\del_{z_n} t + \eta({z_n}\varepsilon_j)\nabla_{z'} f_j({z'})\nabla_{z'} t\right].
	\]}
	For $\abs{{z_n}} \leq (4\varepsilon_1)^{-1}$ and $j=0$ we have the equation
	\[
	(\Delta t + 2i\xi \nabla t)f_0 + f_1({z'})\del_{z_n} t + \nabla_{z'} f_0 \nabla_{z'} t = 0\,,
	\]
	where $\del_{z_n} t({z'},0) = -\tau_\Xi\del_\nu(\phi+i\psi)$. Thus, for ${z_n}$ near $0$ this is bounded away from $0$ and so we may divide by it to get some $f_1 \in \mathcal{O}_{L^\infty}(1)$ with respect to $h$. 
	
	For every other $j \geq 1$, and having solved recursively for $f_1\dots,f_{j}$, we must solve an equation of the form
	\[
	(\Delta t + 2i\xi \nabla t)f_j + f_{j+1}({z'})\del_{z_n} t + \nabla_{z'} f_j \nabla_{z'} t = 0\,.
	\]
	This is possible because for small enough $z_n$, $\del_{z_n} t$ is non-zero. This recursive definition done, for $z_n$ small enough, $b$ is equal to $\mathcal{O}(z_n^{w+1})$, and to get the support properties of $b$, we cut it off with a smooth function.
\end{proof}

\subsection{Properties of The Partial Cauchy Data Set}

The proof here is standard and follows \cite{recoveringCauchydata}.
\begin{proof}[Proof of \cref{lemm:hsharpprops}]
	Fix $v\in H^{1/2}(\del\Omega)$ and assume $u \in C^\infty(\bar\Omega)$. By the standard extension theorem there is some $e_v \in H^1(\Omega)$ so that $e_v\vert_{\partial\Omega} = v$ and 
	\begin{equation}\label{eq:extensionbounded}
		\norm{e_v}_{H^1(\Omega)} \leq C\norm{v}_{H^{1/2}(\partial\Omega)}\,.
	\end{equation}
	
	Define
	\[
		w(v) \coloneqq \int_{\del\Omega} \bar v\del_\nu u\dd S\,,
	\]
	where $S$ is the surface measure and remark that by partial integration,
	\[
		w(v) = \int_\Omega \nabla u \nabla \bar e_v - \bar e_v \Delta u\dd x
	\]
	so that by Sobolev embedding $H^1(\Omega) \subset L^p(\Omega)$ and H\"older's inequality, 
	\begin{equation}\label{eq:boundedw}
		\abs{w(v)} \leq C\norm{\nabla u}_{L^2}\norm{\nabla e_v}_{L^2} + \norm{e_v}_{L^p}\norm{\Delta u}_{L^{p'}} \leq C\norm{u}_{H^\sharp_\Delta}\norm{v}_{H^{1/2}(\del\Omega)}\,,
	\end{equation}
	where the second inequality in \cref{eq:boundedw} used Poincar\'e's inequality and \cref{eq:extensionbounded}. Note also that the definition of $w$ is independent of the choice of $e_v$.
	
	Thus \cref{eq:boundedw} lets us conclude that $w \in H^{-1/2}(\del\Omega)$ and $C^\infty(\bar\Omega) \ni u \mapsto w \in H^{-1/2}(\del\Omega)$ is bounded in the $H^\sharp_\Delta(\Omega)$ norm. Using the density of $C^\infty(\bar \Omega)$ in $H^\sharp_\Delta(\Omega)$ concludes the proof. 
\end{proof}

\subsection{Seminorm Bound on Symbols}

Let $j \in S^{-\infty}_1(\RR^{n-1})$ be a symbol satisfying 
\begin{equation}\label{eq:goodsymbol}
	\abs{\del^{\alpha}_{x'}\del^\beta_{\xi'}j(x',\xi')} \leq C_{\alpha,\beta}\,,\an \supp j \subset \RR^{n-1}\times K\,,
\end{equation}
for some compact $K\subset \RR^{n-1}$ and some $C_{\alpha,\beta}\geq 0$.

\begin{lemma}\label{lemm:semi}
	Let $a \in S^{-\infty}_1(\RR^{n-1})$ be as in \cref{eq:goodsymbol}.
	We have $a(x',hD) \colon L^r(\RR^{n-1}) \to W^{k,r}(\RR^{n-1})$ for all $k\in\mathbb{N}_0$, $1<r<\infty$ with the bound 
	\begin{equation*}
		\norm{a(x',hD) u}_{W^{k,r}} \leq C_nC_K\max_{\abs{\beta}\leq 2n} C_{0,\beta} \norm{u}_{L^r}
	\end{equation*}
	where $C_{\alpha,\beta}$ are the seminorm bounds for the symbol $\langle \xi'\rangle^k a(x',\xi')$, $C_n$ depends only on the dimension $n$ and $C_K$ depends only on the measure of $K$.
\end{lemma}
\begin{proof}
	Assume that $u\in \Ss(\RR^{n-1})$ and fix $1<r<\infty, k\in \mathbb{N}_0.$

	Remark that by partial integration $a(x',hD')$ has the integral kernel
	\begin{equation}\label{eq:kernel}
		k_a(x',y') = (2\pi)^{-n+1}\int_{K} e^{i(x'-y')\cdot\xi'}(1+\Delta_{\xi'})^n a(x',h\xi')\dd\xi' (1+\abs{x'-y'}^2)^{-n}
	\end{equation}
	which satisfies the bounds 
	\begin{equation}\label{eq:boundonkernel}
		\abs{k_a(x',y')} \leq \max_{\abs{\beta}\leq 2n} C_{0,\beta} C_n C_K(1+\abs{x'-y'}^2)^{-n}\,,
	\end{equation}
	where $C_K = \abs{K}$ is the measure of $K$, and $C_n = (2\pi)^{-n+1}$.

	Using \cref{eq:boundonkernel}, we find that 
	\begin{equation}\label{eq:constskernel}
		\int_{\RR^{n-1}} \abs{k_a(x,y)} \dd x' \leq C_n' C_K \max_{\abs{\beta}\leq 2n} C_{0,\beta}\,,\an \int_{\RR^{n-1}} \abs{k_a(x,y)} \dd y' \leq C_n' C_K \max_{\abs{\beta}\leq 2n} C_{0,\beta}
	\end{equation}
	where $C_n' = C_n \int_{\RR^{n-1}} (1+\abs{x'-y'}^2)^{-n}\dd x'$.

	By \cite[Prop.~5.1]{zbMATH05771952} and \cref{eq:constskernel} we have 
	\begin{equation}
		\norm{a(x',hD')u(x')}_{L^r} \leq C_n' C_K \max_{\abs{\beta}\leq 2n} C_{0,\beta} \norm{u}_{L^r}\,.
	\end{equation}

	Finally, if we replace the symbol $a$ by $\langle \xi'\rangle^{k} a(x',\xi')$, then the same arguments with different seminorm bounds $C_{\alpha,\beta}'$ will give 
	\begin{equation}
		\norm{a(x',hD')u(x')}_{W^{k,r}} = \norm{\langle hD'\rangle^{k} a(x',hD')u(x')}_{L^r} \leq C_n' C_K \max_{\abs{\beta}\leq 2n} C_{0,\beta}' \norm{u}_{L^r}\,.
	\end{equation}
	where we used the known fact that Bessel potential spaces correspond to Sobolev spaces for integer regularity exponents. Using the density of $\Ss$ in $L^r$ concludes the proof.
\end{proof}

\raggedright
\printbibliography[title=Bibliography] 

\end{document}